\documentclass[a4paper,11pt,twoside]{book}

\usepackage[utf8]{inputenc}
\usepackage[T1]{fontenc}
\usepackage{amsfonts}
\usepackage{amsthm}
\usepackage{amsmath}
\usepackage[english,francais]{babel}
\usepackage[all]{xy}

\usepackage{graphicx}
\usepackage{amscd}
\usepackage{latexsym}
\usepackage{hyperref}
\usepackage{mathrsfs}

\usepackage{color}

\usepackage{calc}

\theoremstyle{definition}
 \newtheorem{defi}{Définition}[chapter]

\theoremstyle{remark}
 \newtheorem{remark}[defi]{Remarque}
 \newtheorem{exam}[defi]{Exemple}
 \newtheorem{question}[defi]{Question}

\theoremstyle{plain}
\newtheorem{theo}[defi]{Théorème}
\newtheorem{prop}[defi]{Proposition}
\newtheorem{cor}[defi]{Corollaire}
\newtheorem{lemma}[defi]{Lemme}
\newtheorem{conj}[defi]{Conjecture}

\newcommand{\zz}{\mathbb{Z}}
\newcommand{\qq}{\mathbb{Q}}
\newcommand{\rr}{\mathbb{R}}
\newcommand{\cc}{\mathbb{C}}

\newcommand{\Z}[1]{\zz_{#1}}

\newcommand{\abs}[1]{\left\vert #1 \right\vert}

\newcommand{\red}{/\!\!/}

\newcommand{\A}{\mathscr{A}}

\newcommand{\Mg}{\mathscr{M}^\mathfrak{g}}

\newcommand{\G}{\mathscr{G}}
\newcommand{\Gc}{\mathscr{G}^c}

\newcommand{\N}{\mathscr{N}}
\newcommand{\Nc}{\mathscr{N}^c}

\newcommand{\I}{\mathcal{I}}

\DeclareMathOperator{\supp}{supp}

\newcommand{\Symp}{\textbf{Symp}}
\newcommand{\Cob}{\textbf{Cob}}  
\newcommand{\Cobelem}{\textbf{Cobelem}}
\newcommand{\Coincob}{\textbf{Coincob}}  
\newcommand{\Coincobelem}{\textbf{Coincobelem}} 

\newcommand{\arnaque}{\begin{flushright}
$\Box$
\end{flushright}}

\newcommand{\OSz}{Ozsv{\'a}th et Szab{\'o} }
\newcommand{\MW}{Manolescu et Woodward }
\newcommand{\WW}{Wehrheim et Woodward }

\title{Homologie Instanton Symplectique : somme connexe, chirurgie entière, et applications induites par cobordismes}
\author{Guillem Cazassus}

\begin{document}

\maketitle

\frontmatter
\chapter*{Preamble}
Except from this preamble, this is  the manuscript of the author's PhD. thesis, which was defended at Université Toulouse 3 Paul Sabatier (UT3 Paul Sabatier) on April 12, 2016. 

Chapters~\ref{chapfft} to \ref{chapchir} have been translated in \cite{surgery}, and chapter~\ref{chapcob} and section~\ref{sectionnaturalite} will appear in \cite{cobordisms}.

\vspace{.2in}

\textbf{Updates on Chapter~\ref{chappersp} (Prospects).} 
\begin{itemize}
\item After this thesis was defended, C. Manolescu pointed out to the author that the three differentials $\partial ^\infty$, $\partial ^+$ and $\partial ^-$ of section~\ref{sec:autresversions} might not be well-defined if we don't replace the rings $\Z{2}[U]$,  $\Z{2}[U,U^{-1}]$ and $\Z{2}[U,U^{-1}]/ U \Z{2}[U]$ by their power series completion, as the sum that defines them might not be finite. We thank him for this remark.
\item The spectral sequence of section\ref{sectionsspec} has now been proven in \cite{Schmaschke}.
\end{itemize}

\chapter*{Remerciements}
Je tiens en premier lieu à remercier chaleureusement mes deux directeurs d'avoir accompagné mes premiers pas dans la recherche, de mes premiers mémoires de master, jusqu'à la rédaction de cette thèse.
Merci à Paolo, pour m'avoir suggéré un sujet riche et vaste, qui m'aura permis d'apprendre une quantité sans fin de belles mathématiques.  
Merci à Michel Boileau, pour m'avoir toujours encadré avec enthousiasme malgré le grand virage symplectique qu'aura pris cette thèse. 
Merci à tous les deux, pour votre patience, pour avoir toujours été disponibles malgré une situation géographique compliquée, pour nos rendez-vous skypes finissant parfois très tard, ainsi que pour vos relectures minutieuses,  
dont la forme finale de ce manuscrit doit beaucoup.

Merci à Christopher Woodward et Frédéric Bourgeois d'avoir accepté de rapporter cette thèse. Merci en particulier à Christopher Woodward d'avoir accepté de lire ce manuscrit en français, d'avoir partagé ses idées lors de nos conversations à Montréal, et notamment de m'avoir suggéré la définition de la version tordue, ce fut une étape clé pour cette thèse. Thanks a lot!

Merci à Christian Blanchet, pour avoir accepté de faire partie de mon jury, et pour m'avoir suggéré d'appliquer la suite exacte aux entrelacs quasi-alternés.
Un grand merci également à Jean-François Barraud, qui n'aura jamais compté son temps pour partager son expertise sur l'homologie de Floer, cette thèse lui doit également beaucoup et je suis très heureux de le compter parmi les membres de mon jury.

Merci aussi à toutes les personnes avec qui j'ai eu la chance de discuter de maths et qui ont partagé leurs connaissances avec moi. Sans être exhaustif, je pense notamment à Julien, Laurent, Joan, sans oublier Raphael et son hospitalité à Regensburg.

Merci aux thésards toulousains, Jéremy, Matthieu, Ilies, Dima, Valentina, Eleonora, Eric, Anton, Fabrizio, Danny,  Anne, Damien, Laura, Maxime, Julie, Jules, Kévin, Audunn, pour avoir fait régner la Bonne Ambiance dans le couloir du 1R2, sans oublier Arnaud, Thibaut et Bérénice. Merci aussi aux "gens d'en face", pour les bons moments lors des séminaires RANDO : je pense à Mélanie, Claire ($\times$3), Matthieu, Malika, Antoine, Laurent, Ioanna, Guillaume. Merci également à Jocelyne, Marie-Laure, Agnès et Martine, pour m'avoir accompagné avec efficacité dans mes démarches administratives.

Merci aussi aux membres du laboratoire de Nantes, et en particulier à son équipe stimulante de géométrie symplectique, qui aura fait de mes quelques séjours passés là-bas des séjours agréables et enrichissants, je pense notamment à Baptiste, Hossein, Vini, Sam, Vincent, Vincent, Guillaume, François.

Merci aussi à ceux que j'ai connus à l'occasion de conférences (notamment les inconditionnels de la Llagonne) je pense à Julien, Léo, Ma\"{y}lis, Delphine, Xavier, Ramanujan, Fathi, Renaud, Ben, et tant d'autres  ...

Merci aussi aux amis de longue et très longue date, Bruno, Oliv, Lucie, Mathilde, Roméo, Ahmed, Remy, Florence, Matthieu, Morgane, Morgan, Gaby, Pif, Juan Pablo, Ivan, Fred, Sylvie, Xav, Laurie, Nina, Louise, Pierre, Céline, Anne, Guillaume, Nicolas, Cyril, Mélodie, Romain, et j'en oublie, pour les (trop rares) bons moments où nous nous sommes retrouvés.

Merci enfin à tous les professeurs de mathématiques qui m'ont donné le gôut pour cette discipline, je pense notamment à Vincent Bayle et Monsieur "Jimmy", et sans oublier Jacques Jourliac, mon professeur de lycée et très cher ami, qui est probablement la personne qui m'aura convaincu d'en faire mon métier.

Merci à ma famille, pour votre soutien pendant cette thèse et depuis toujours, je vous dois tout.

Enfin, merci à Claire, pour remplir ma vie de bonheur au quotidien.

\newpage

\vspace*{1in}

\begin{flushright}
\textit{à Jacques Jourliac,}
\end{flushright}

\vspace*{1in}

\begin{figure}[!h]
    \centering

 \includegraphics[width=.65\textwidth]{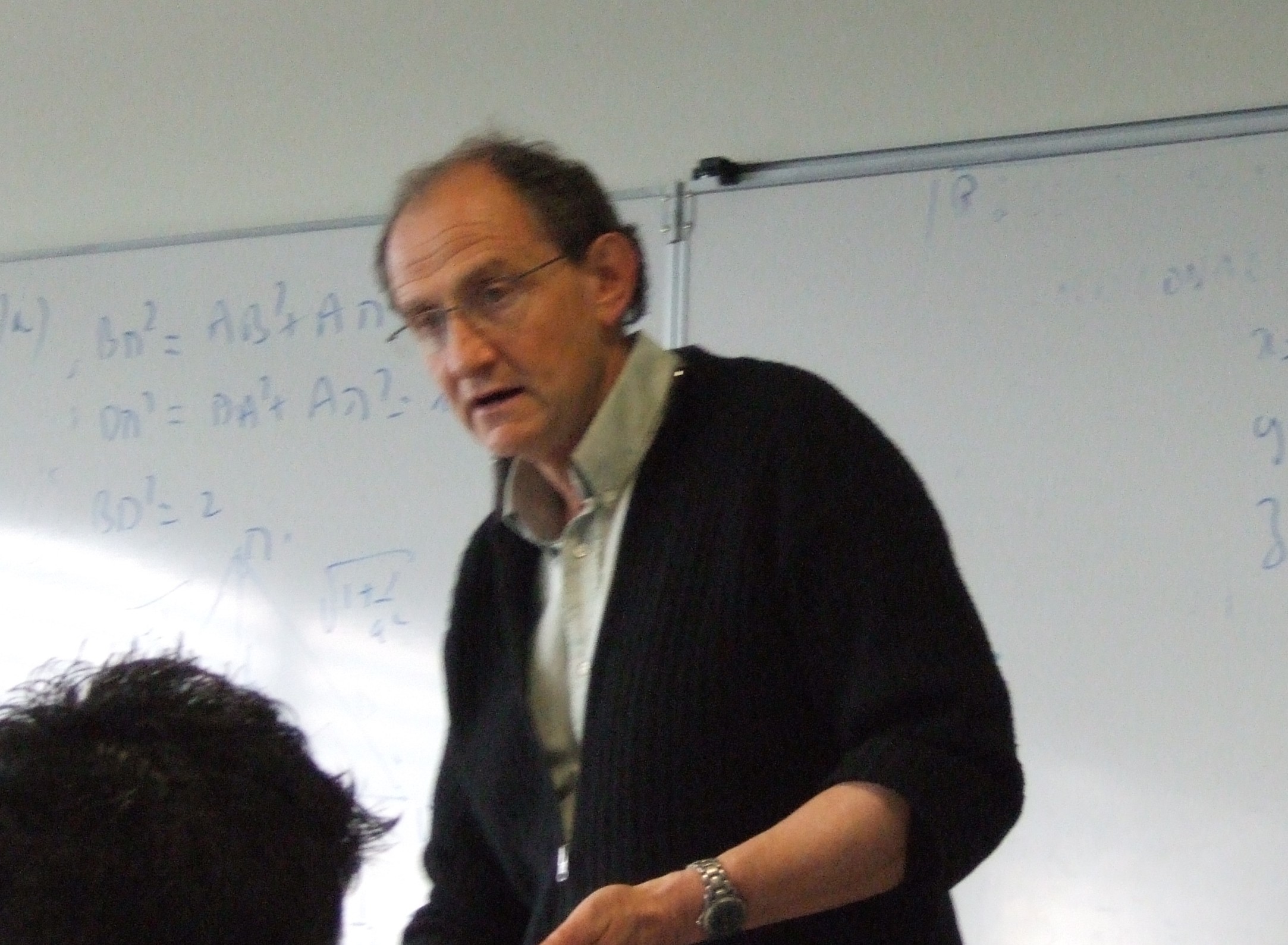}
  
\end{figure}

\tableofcontents

\mainmatter

\chapter{Introduction}\label{chapintro}

\section{Bref historique}

Les instantons ont été introduits en 1975 par 't Hooft en physique des particules \cite{tHooft}, comme solutions particulières des équations de Yang-Mills : ce sont des connexions $A$ sur un fibré principal au-dessus d'une 4-variété Riemanienne, dont la courbure est anti-autoduale. Ils font leur première apparition en topologie de basse dimension en 1983, lorsque Donaldson \cite{Donaldsonannounce}, en étudiant leur espace des modules, prouve son célèbre théorème sur la diagonalisation des formes d'intersection des 4-variétés différentiables. Ce théorème a pour conséquence la non-différentiabilité de nombreuses 4-variétés topologiques. Il construit par la suite une famille d'invariants permettant de distinguer des structures différentielles \cite{Donaldsonpoly}.

Dans la fin des années 80, Floer définit plusieurs groupes d'homologie qui porteront son nom. Les premiers, associés à une 3-variété close orientée $Y$, sont construits comme (des substituts à) une homologie de Morse de la fonctionnelle de Chern-Simons (pour une sphère d'homologie entière, sur le $SU(2)$-fibré trivial), dont le complexe est engendré par des connexions plates (perturbées), et la différentielle compte des instantons sur le tube $Y\times \rr$ \cite{Floerinst}. Ces groupes $I_*(Y)$ sont des réceptacles naturels d'invariants pour 4-variétés à bord, généralisant les invariants de Donaldson, et obéissent aux axiomes d'une "(3+1)-théorie quantique des champs topologique". Par ailleurs, la reformulation par Taubes en théorie de jauge \cite{Taubescasson} de l'invariant de Casson montre que ces groupes catégorifient cet invariant.

Les second groupes sont de nature symplectique : Floer définit une homologie associée à une variété symplectique munie d’une Lagrangienne \cite{Floerlagr}, construite en utilisant la théorie des courbes pseudo-holomorphes de Gromov \cite{Gromov}. Le calcul de ces groupes dans certains cas lui permet d'établir des versions de la conjecture d’Arnold. Il généralise ensuite sa construction à une variété symplectique munie d'une paire de Lagrangiennes $L_0,L_1\subset M$, et parvient, sous certaines hypothèses, à définir des groupes similaires, notés $HF(L_0,L_1)$, dont le complexe est engendré par les points d'intersection $L_0\cap L_1$, et la différentielle compte des disques de Whitney pseudo-holomorphes. 

En "étirant" une 3-variété Riemanienne $Y$ le long d'un scindement de Heegaard, Atiyah met en évidence une correspondance entre limites d'instantons et courbes pseudo-holomorphes dans l'espace des modules de la surface scindement. Il suggère dans \cite{Atiyahfloer} que le groupe $I_*(Y)$ peut être calculé comme une homologie d'intersection Lagrangienne dans l'espace des modules des connexions plates associé à la surface, pour la paire de Lagrangiennes correspondant aux connexions s'étendant de manière plate à chaque corps à anses. C'est cette idée qui sera retenue comme la "conjecture d'Atiyah-Floer", dont une difficulté majeure est de définir l'homologie d'intersection Lagrangienne dans l'espace des modules, qui est singuliers. Pour plus de détails sur cette conjecture nous renvoyons à \cite{Wehrheimsurvey} \cite{SalamonICM} \cite{duncan2013introduction}. Des  analogues de cette conjecture sont  démontrés dans différents cas \cite{DostoglouSalamon}, mais toujours lorsque l'espace des modules n'a pas de singularités. Toutefois, bien que des progrès récents aient été faits, cette conjecture demeure toujours ouverte dans sa version originale. 
\section{Panorama des théories de Floer actuelles}
Les techniques utilisées par Floer pour définir les deux homologies précédentes se sont avérées fécondes, et ont permis de définir de nombreux groupes d'homologie associés à des fonctionnelles définies sur des espaces de dimension infinie. Le raisonnement d'Atiyah a inspiré également la définition de nombreux invariants très riches. Nous en décrivons quelques uns dans les paragraphes qui suivent.

\paragraph{Homologie des monopoles} Quelques années après l'introduction des instantons en topologie, Seiberg et Witten proposent une nouvelle équation aux dérivées partielles linéaire plus simple à manipuler, qui permettra de retrouver de nombreux résultats, et de définir des invariants similaires aux invariants de Donaldson, voir \cite{Wittenmonopoles}. Puis, Kronheimer et Mrowka définissent des groupes d'homologie faisant intervenir cette équation, qu'ils appellent "homologie des monopoles". Cela leur permet notamment de démontrer la propriété P pour les n{\oe}uds \cite{KMpropertyP}.

Le raisonnement d'Atiyah inspire alors des variantes symplectiques de ces groupes, parmi lesquels l'homologie de Heegaard-Floer d' \OSz et les "Lagrangian matching invariants" de Perutz \cite{Perutzmatching}.

\paragraph{Homologies de Heegaard-Floer} L'homologie de Heegaard-Floer est un ensemble de théories. La première, introduite dans \cite{OSzholodisk1} et \cite{OSzholodisk2},  associe des invariants pour des 3-variétés closes connexes orientées. C'est une homologie d'intersection Lagrangienne  associée à un diagramme de Heegaard de genre $g$ : la variété symplectique considérée est le $g$-ième produit symétrique, dans laquelle vivent deux tores Lagrangiens, correspondants aux produits des courbes du diagramme.

Peu après, \OSz  \cite{OSzholodiskknots} et Rasmussen \cite{Rasmussen} définissent indépendamment des invariants analogues associés aux n{\oe}uds, puis Juh{\'a}sz unifie ces deux théories en définissant une homologie associée à certaines variétés suturées, appelées "Balanced".

Du fait de leur richesse structurelle (structures  $\mathrm{Spin}_\cc$, graduations absolues, structure de $\zz[U]$-module, applications induites par cobordismes), de leur calculabilité (suite exacte de chirurgie, formule d'adjonction, versions combinatoires), et de leurs relations avec des invariants antérieurs (invariant de Casson, norme de Thurston, polynome d'Alexander), ces invariants suscitent une grande activité chez les topologues. Parmi les applications frappantes de ces invariants on peut mentionner une caractérisation des nœuds fibrés (Ghiggini \cite{Paolofibknots}, et  Ni \cite{Nifibknots}), ou le calcul de la norme de Thurston d’une 3-variété (\cite{OSzthurstonnorm} et \cite{Nithurstonnorm}).

Dans ce contexte la version correspondante de la conjecture d’Atiyah-Floer a été démontrée par  Kutluhan, Lee et Taubes. Dans \cite{KLTHF=HM1} et une série de papiers ils construisent un isomorphisme de l’homologie de Heegaard-Floer vers l’homologie des monopoles.

\paragraph{Variantes de l'homologie des instantons} L'homologie des instantons n'a été définie par Floer que pour certaines 3-variétés (sphères d'homologie entières, et d'autres en utilisant des $SO(3)$-fibrés non-triviaux), en raison de singularités apparaissant au niveau des connexions réductibles. Kronheimer et Mrowka, en considérant une somme connexe avec un 3-tore, définissent des généralisations pour toutes les 3-variétés, ainsi que des versions nouées et suturées, voir \cite{KMknots} et \cite{KMsut}.

\paragraph{Pillowcase homology} Récemment, dans l'espoir d'obtenir une version symplectique des variantes de Kronheimer et Mrowka de l'homologie des instantons, Hedden, Herald et Kirk ont défini dans \cite{HHKpillowcase2} des invariants pour certains n{\oe}uds, comme une hommologie d'\emph{immersions} Lagrangiennes dans la "variété des caractères $SU(2)$ sans traces" de la sphère privée de quatre points, variété homéomorphe à une "taie d'oreiller". Ils appellent cet invariant "pillowcase homology".

\paragraph{Théorie des quilts}
Dans \cite{WWfft} et \cite{WWffttangles}, Wehrheim et Woodward proposent un nouveau cadre général pour construire les différentes versions symplectiques d'invariants de théorie de jauge : leur construction part non plus d'un scindement de Heegaard, mais d'une décomposition de Cerf de la 3-variété et utilise la théorie des "quilts pseudo-holomorphes" développée dans \cite{WWquilts}. Ils appliquent leur théorie et constr\-uisent des invariants en utilisant des espaces de modules de connexions à courbure centrale sur des $U(N)$-fibrés de degrés premiers à $N$. Nous nous inspirerons de leurs travaux pour adapter leur théorie au cadre utile pour l’homologie HSI dans le chapitre \ref{chapfft}.

\section{Brève présentation de l'homologie instanton-symplectique}
Nous présentons l’homologie Instanton-Symplectique, définie par Manolescu et Woodward dans \cite{MW}, et qui est l’objet principal de cette thèse. Rappelons que la version symplectique de l’homologie des instantons suggérée par Atiyah devrait être définie comme une homologie d’intersection Lagrangienne dans la variété des caractères d’un scindement, qui est une variété symplectique singulière. Jeffrey a remarqué dans \cite{jeffrey} que la variété des caractères peut se réaliser comme le quotient symplectique d’une variété symplectique de dimension finie munie d’une action Hamiltonienne de $SU(2)$, dont le niveau zéro du moment est contenu dans la partie lisse. Cet espace, appelé "espace des modules étendu", correspond à un espace des modules de connexions plates sur le $SU(2)$-fibré principal trivial au-dessus de la surface privée d’un disque, présentant une forme particulière au voisinage du bord.

 L'idée de Manolescu et Woodward est alors de définir l'homologie Lagrangienne dans un ouvert de cet espace. Ils y parviennent, et conjecturent que leur invariant est isomorphe à la version "chapeau" de l'homologie de Heegaard-Floer, ainsi qu'à une version de l'homologie des instantons définie par Donaldson, correspondant au  cône de l'application $u$ (après produit tensoriel avec $\qq$).

\section{Résultats dans cette thèse, perspectives}
Dans l'optique de pouvoir calculer l'homologie HSI, cette thèse est motivée par les questions suivantes :

\begin{itemize}
\item[(A)] Décrire l'homologie HSI d'une somme connexe de deux 3-variétés à partir de l'homologie HSI des variétés initiales,
\item[(B)] Décrire l'influence d'une chirurugie de Dehn sur l'homologie HSI,
\item[(C)] Définir des invariants associés à des 4-cobordismes.
\end{itemize}

Rappelons que d'après le théorème de Lickorish-Wallace, toute 3-variété peut s'obtenir par chirurgie de Dehn sur un entrelacs, ce qui souligne l'importance de $(B)$. Un résultat central pour le calcul de ce type d’invariants est le "triangle de Floer" : une suite exacte longue entre les invariants d’une triade de chirurgie, \cite{Floertri} pour les instantons, puis dans les autres théories   \cite{OSzholodisk2} \cite{KMOSzchir}. De plus, les morphismes intervenant dans de telles suites exactes ont généralement une interprétation topologique : ils sont souvent associés aux trois cobordismes correspondants aux attachements d’anses entre les trois 3-variétés. En théorie des instantons, ces applications sont définies en comptant des instantons sur les cobordismes correspondants. De telles interprétations permettent parfois d’obtenir des critères d’annulation, facilitant les calculs, notamment des formules d’adjonction et des formules d’éclatement.

L'existence d'un tel triangle de chirurgie pour l’homologie des instantons est confrontée au problème suivant : trois variétés formant une triade ne peuvent être simultanément des sphères d’homologie entière. Floer contourne ce problème en utilisant un $SO(3)$-fibré principal non-trivial sur la variété ayant l'homologie de $S^2\times S^1$, afin de pouvoir y définir une homologie des instantons. Bien que l'homologie HSI soit définie pour toutes les 3-variétés, nous verrons que le même phénomène apparait dans cette théorie, ce qui, suivant une suggestion de C. Woodward, nous a conduit à introduire une variante $HSI(Y,c)$ de l'homologie HSI, associée à une 3-variété $Y$ munie d'une classe $c$ dans $H_1(Y,\Z{2})$, ou de manière équivalente, d'une classe d'isomorphisme de fibrés $SO(3)$ au-dessus de $Y$.  Nous expliquons brièvement la construction de cette variante, qui sera  définie précisément dans le chapitre \ref{chapHSI}.

Soit $\Sigma \subset Y$ un scindement de Heegaard de genre $g$, séparant $Y$ en deux corps à anses $H_0$ et $H_1$, et $C_0$, $C_1$ deux n{\oe}uds dans $H_0$ et $H_1$ respectivement, tels que la classe de leur réunion dans $H_1(Y;\Z{2})$ vaut $c$. Soit $\Sigma'$ la surface à bord obtenue en retirant un disque à $\Sigma$, et $*\in \partial\Sigma'$ un point base. On considère $\N(\Sigma')$ un certain espace des modules de connexions au-dessus de $\Sigma'$, admettant la description suivante :

\[\N(\Sigma') = \lbrace  \rho \in \mathrm{Hom}(\pi_1(\Sigma',*) , SU(2) ) :  \rho(\partial \Sigma')   \neq -I \rbrace . \]

Cet espace admet une structure symplectique naturelle, pour laquelle les sous-variétés suivantes sont des Lagrangiennes :
\begin{align*}
 L_0 &= \lbrace  \rho \circ i_{0,*} : \rho  \in \mathrm{Hom}(\pi_1(H_0\setminus C_0,*), SU(2) ) ,  \rho(\mu_0)  = -I \rbrace  \\
 L_1 &= \lbrace  \rho \circ i_{1,*} : \rho  \in \mathrm{Hom}(\pi_1(H_1\setminus C_1,*), SU(2) ) , \rho(\mu_1)  = -I \rbrace ,
\end{align*}
avec $i_{0,*}$ et $i_{1,*}$ induits par les inclusions, et $\mu_0$, $\mu_1$ des méridiens de $C_0$ et $C_1$ respectivement.

Le groupe  $HSI(Y,c)$ peut alors être défini comme l'homologie d'intersection Lagrangienne $HF(L_0,L_1)$. Afin de pouvoir utiliser la théorie des quilts de Wehrheim et Woodward, nous  définirons cet invariant directement dans le cadre de cette théorie, puis nous prouverons dans la proposition \ref{calculscindement} que cette définition correspond bien à celle que l'on vient de donner. 

\paragraph{Énoncé des principaux résultats} Nous décrivons à présent les principaux résultats
de cette thèse.

Concernant la question $(A)$, nous obtenons la formule de Künneth suivante : 
\begin{theo}(Formule de Künneth pour la somme connexe) Soient $Y$ et $Y'$ deux 3-variétés orientées closes, et $c$, $c'$ deux classes dans $H_1(Y;\Z{2})$ et $H_1(Y';\Z{2})$ respectivement. Alors,
\begin{align*}
HSI(Y \# Y' , c+ c') \simeq & HSI(Y, c) \otimes HSI(Y' , c' )  \\
 & \oplus \mathrm{Tor}(HSI(Y, c) , HSI(Y' , c' ))[-1].
\end{align*}
\end{theo} 

Pour présenter une réponse à la question $(B)$, nous rappelons ce qu'est une triade de chirurgie (dont l'archétype correspond aux chirurgies $\infty$, 0 et 1 sur un n{\oe}ud muni d'une longitude). 

\begin{defi} Une \emph{triade de chirurgie} est un triplet de 3-variétés $Y_\alpha$, $Y_\beta$ et $Y_\gamma$ obtenues à partir d'une 3-variété $Y$ compacte, orientée, avec un bord de genre 1, en recollant un tore solide le long du bord, de façon à envoyer son méridien sur trois courbes simples $\alpha$, $\beta$ et $\gamma$ respectivement, telles que $\alpha. \beta = \beta. \gamma =\gamma. \alpha = -1$.
\end{defi}

Notre suite exacte s'énonce alors  de la manière suivante :

\begin{theo}[Suite exacte de chirurgie]\label{suiteintro} Soit $(Y_\alpha , Y_\beta, Y_\gamma)$ une triade de chirurgie obtenue à partir de $Y$ comme dans la définition précédente, $c\in H_1(Y;\Z{2})$, et pour $\delta \in \lbrace \alpha,\beta,\gamma\rbrace$, $c_\delta \in H_1(Y_\delta;\Z{2})$ la classe induite à partir de $c$ par les inclusions. Soit également  $k_\alpha \in H_1(Y_\alpha;\Z{2})$ la classe correspondant à l'\^ame du tore solide. Alors, il existe une suite exacte longue :
\[ \cdots\rightarrow HSI(Y_\alpha ,c_\alpha+ k_\alpha) \rightarrow HSI(Y_\beta,c_\beta) \rightarrow HSI(Y_\gamma,c_\gamma)\rightarrow \cdots . \]
\end{theo}

Concernant la question $(C)$, nous nous limitons à l'homologie à coefficients dans $\Z{2}$. Dans le Chapitre \ref{chapcob}, nous construisons des invariants pour des 4-cobordismes prenant la forme suivante : Soit $W$ un 4-cobordisme compact orienté de $Y$ vers $Y'$, muni d'une classe $c_W \in H^2(W; \Z{2})$. En notant $c$ et $c'$  les classes d'homologie duales aux restrictions de $c_W$ à $Y$ et $Y'$, on construit un morphisme 
\[ F_{W,c_W} \colon HSI(Y,c)\to HSI(Y',c').\]

Nous montrons alors que deux parmi trois des morphismes intervenant dans la suite exacte sont du type précédent. Plus précisément :


\begin{theo} Soient une triade et des classes d'homologie comme dans le théorème \ref{suiteintro}. Les deux morphismes de $HSI(Y_\alpha ,c_\alpha + k_\alpha)$ vers $HSI(Y_\beta,c_\beta)$, puis de $HSI(Y_\beta,c_\beta)$ vers $HSI(Y_\gamma,c_\gamma)$ de la suite exacte du théorème précédent, que l'on a construit dans la partie \ref{flechestri}, correspondent aux applications $F_{W_{\alpha\beta}, c_{\alpha\beta} + d_\alpha}$ et $F_{W_{\beta\gamma}, c_{\beta\gamma} }$, où :

\begin{itemize}

\item $W_{\alpha\beta}$ et $W_{\beta\gamma}$ désignent les cobordismes d'attachement d'anse, allant respectivement de $Y_\alpha$ vers $Y_\beta$, et de $Y_\beta$ vers $Y_\gamma$.

\item $c_{\alpha\beta}$ et $c_{\beta\gamma}$ désignent les classes induites par $c\times [0,1] \in H_2(Y\times [0,1], Y\times \lbrace 0,1\rbrace ;\Z{2} )$ par l'inclusion de $Y\times [0,1]$ dans $W_{\alpha\beta}$ et $W_{\beta\gamma}$ respectivement.

\item $d_\alpha$ désigne la classe fondamentale de l'anse attachée le long de $\alpha$, de sorte que $\partial_* d_\alpha = k_\alpha \in H_1(Y_\alpha;\Z{2})$.
\end{itemize}
\end{theo}

Nous donnons enfin un critère d'annulation pour de tels morphismes :

\begin{theo}
Soit $W$ un 4-cobordisme connexe orienté entre variétés connexes, 
\begin{enumerate}
\item $F_{W\# \cc P^2,c} =  0 $, pour toute classe $c\in H^2(W\# \cc P^2;\Z{2})$.
\item Si $c\in H^2(W\# \overline{\cc P}^2;\Z{2})$ est non-nulle en restriction à $\overline{\cc P}^2$, alors $F_{W\# \overline{\cc P}^2,c} =  0 $, sinon $F_{W\# \overline{\cc P}^2,c} =  F_{\overline{W},c_{|W}} $.
\end{enumerate}
\end{theo}

\paragraph{Applications}

La suite exacte de chirurgie, combinée à la connaissance de la caractéristique d'Euler de l'homologie HSI et d'une observation d'Ozsv{\'a}th et Szab{\'o}, permet de calculer l'homologie HSI de plusieurs variétés. On en présentera quelques unes dans la section \ref{sectionapplic}, notamment des variétés obtenues par plombage de fibrés en disques sur des sphères le long d'un graphe, des chirurgies sur des noeuds, et des revêtements doubles d'entrelacs quasi-alternés.

\paragraph{Idée des preuves}
La stratégie pour prouver la formule de Künneth pour la somme connexe est de réinterpréter la construction de l'homologie $HSI$ dans le cadre de la "théorie des champs de Floer" de Wehrheim et Woodward. Similaire à une TQFT, une telle théorie consiste en un foncteur de la catégorie des cobordismes de dimension $2+1$ vers la catégorie symplectique de Weinstein-Wehrheim-Woodward. Cette approche facilitera par ailleurs la preuve de la suite exacte de chirurgie, en offrant la possibilité de ne plus travailler à partir d'un scindement de Heegaard.

Pour établir la suite exacte de chirurgie, nous étudions l’effet d’un twist de Dehn sur le tore privé d’un disque $T'$ au niveau de l’espace des modules $\N(T')$. Nous remarquons que la transformation induite est similaire à un "twist de Dehn symplectique", ce qui permet d’appliquer une variante de la suite exacte de Seidel \cite{Seidel}, suite exacte originellement inspirée par la conjecture d’Atiyah-Floer et le triangle de Floer.

Enfin, la définition des applications induites par cobordismes suit la démarche d’\OSz : nous définissons ces applications anses par anses. Les applications correspondant
aux 2-anses sont construites en comptant des "triangles matelassés".

\paragraph{Perspectives} Dans le dernier chapitre nous discutons de quelques prolongements naturels de cette thèse :

\begin{itemize}
\item \emph{Naturalité}. Etant-donné un difféomorphisme entre deux 3-variétés, peut-on lui associer canoniquement un isomorphisme entre les groupes HSI correspondants? La question de la naturalité, bien qu'implicitement utilisée  pour nombre d'applications, n'a été résolue que tardivement par Juhasz et Thurston en théorie d'Heegaard-Floer. Nous nous  attendons à ce que les groupes d'homologie HSI ne dépendent pas de la seule 3-variété, mais aussi d'un point base, de même les applications induites par cobordisme devraient dépendre d'une classe d'homotopie  de chemin reliant les points bases. Nous expliquons dans la section \ref{sectionnaturalite} ce qu'il est suffisant de vérifier pour garantir la naturalité, dans le cadre général d'une "théorie des champs de Floer".

\item \emph{Analogue des autres versions d'Heegaard-Floer}. La construction de l'homologie HSI n'est pas sans rappeler celle de la version "chapeau" de l'homologie d'Heegaard-Floer : en effet il faut compactifier $\N(\Sigma')$ en ajoutant une hypersurface symplectique, et la différentielle ne compte que les disques pseudo-holomorphes n'intersectant pas cette hypersurface. Nous définissons des variantes prenant en compte les courbes intersectant cette hypersurface, analogues aux versions +,- et $\infty$ de l'homologie de Heegaard-Floer.

\item \emph{Relations avec la théorie des représentations.} Un lien profond existe entre la théorie de jauge et la théorie des représentations. Il se manifeste notamment dans la description des espaces de modules, et dans la reformulation de Taubes de l'invariant de Casson. Nous définissons un espace topologique, correspondant à une  version "tordue" de la variété des représentations dans $SU(2)$, puis, comme déjà suggéré par \MW et utilisé dans le calcul de l'homologie HSI des lenticulaires, nous proposons une suite spectale reliant l'homologie de cet objet et l'homologie HSI, et nous en déduisons une borne entre le rang de HSI et l'invariant de Casson pour les sphères de Brieskorn.

\item \emph{Invariants pour les entrelacs et les variétés suturées.} Nous proposons une généralisation de la construction de l'homologie HSI faisant intervenir des espaces de modules de surfaces à plusieurs composantes de bord, puis nous indiquons une possible marche à suivre  donnant lieu à des invariants pour des variétés suturées et des entrelacs.

\end{itemize}

\chapter{Théorie des champs de Floer}\label{chapfft}

\section{Grandes lignes, idée générale de la construction}\sectionmark{Grandes lignes}

Commençons par rappeler brièvement ce qu'est une "théorie des champs de Floer", développée par Wehrheim et Woodward dans \cite{WWfft} : une telle théorie est un foncteur de la catégorie des cobordismes connexes entre surfaces connexes vers une catégorie que nous allons préciser dont les objets sont des variétés symplectiques, et les morphismes des suites de correspondances Lagrangiennes.

A une telle suite de correspondances Lagrangiennes ayant même variété symplectique de départ et d'arrivée, Wehrheim et Woodward associent une "quilted Floer homology". Etant donné un tel foncteur, on peut associer à une 3-variété close $Y$ un invariant de la façon suivante : on commence par retirer deux boules ouvertes à $Y$, de sorte à avoir un cobordisme de la sphère $S^2$ vers elle-même, puis on applique le foncteur, et enfin on prend la "quilted Floer homology" de la chaîne de correspondances obtenue.

Si l'on se donne un scindement de Heegaard de la variété $Y$, ce type d'invariant peut être calculé comme une homologie de Floer Lagrangienne "classique", voir la proposition \ref{calculscindement}. La présente formulation à l'aide de la "quilted Floer homology" était suggérée par Manolescu et Woodward pour démontrer que la construction de l'homologie Instanton-Symplectique est invariante par stabilisation. C'est un cadre plus flexible, et pratique pour démontrer la formule de Künneth (Chapitre 4), et la suite exacte de chirurgie (Chapitre 5).

Néanmoins, en ce qui concerne la construction de l'homologie Instanton-Symplectique de Manolescu et Woodward, il sera nécessaire de procéder à quelques adaptations, essentiellement pour trois raisons :

\begin{itemize}
\item Les espaces des modules sont associés à des surfaces ayant une composante de bord : il faudra retirer un petit disque à une surface close, de même qu'il faudra retirer un petit tube reliant ces disques à  un cobordisme entre ces surfaces. La correspondance Lagrangienne ainsi obtenue dépendra du choix du tube, comme nous le verrons dans l'exemple \ref{exemchgtchemin}. Il faudra de plus fixer un paramétrage du bord de ce tube, la dépendance du paramétrage sera illustrée dans l'exemple \ref{exemreparam}.

\item Afin d'établir une suite exacte de chirurgie, il sera nécessaire de munir un cobordisme $W$ de dimension 3  d'une classe d'homologie dans $H_1(W;\Z{2})$. Ceci revient à considérer des fibrés $SO(3)$ non-triviaux au-dessus de $W$, qui apparaissaient déjà dans le triangle de Floer \cite[Theorem 1]{BraamDonaldson} sur le terme correspondant à la $S^2\times S^1$ - variété d'homologie.

\item La catégorie d'arrivée sera elle aussi un peu plus commpliquée : afin de pouvoir y définir une homologie de Floer, il sera nécessaire d'imposer à ses objets et ses morphismes des hypothèses techniques supplémentaires,  qui sont essentiellement celles de  \cite[Assumption 2.5]{MW}.

\end{itemize}

\section{Quilts, Homologie de Floer matelassée, Catégorie symplectique}\sectionmark{Quilts, Homologie matelassée, Catégorie symp}

\begin{defi}
Une \emph{correspondance Lagrangienne} entre deux variétés symplectiques $M$ et $M'$ est une Lagrangienne $L\subset M^- \times M'$, où $M^-$ désigne la variété $M$ munie de l'opposée de sa forme symplectique.
\end{defi}

Ce type de correspondances, parfois appelées relations canoniques dans la littérature, apparaît fréquemment en géométrie symplectique : un difféomorphisme entre deux variétés symplectiques est un symplectomorphisme si et seulement si son graphe est une correspondance Lagrangienne. Par ailleurs, si une variété symplectique est munie d'une action $G$-Hamiltonienne de moment $\Psi\colon M\rightarrow \mathfrak{g}^*$, le niveau zéro du moment $\Psi^{-1}(0)$ induit une correspondance Lagrangienne entre $M$ et le quotient symplectique $M\red G$, lorsque celui-ci est lisse.

\begin{defi}
Conformément à la terminologie de Wehrheim et Woodward, nous appellerons \emph{correspondance Lagrangienne généralisée} entre deux variétés symplectiques $M$ et $M'$ la donnée de variétés symplectiques intermédiaires $M_1$, $M_2$, ..., $M_{k-1}$, ainsi que des correspondances Lagrangiennes pour $0\leq i \leq k-1$ $L_{i(i+1)} \subset M_i ^- \times M_{i+1}$, avec $M_0 = M$ et $M_k = M'$. On notera $\underline{L}$ une telle suite de correspondances :
\[\underline{L} = \left(  \xymatrix{   M_0 \ar[r]^{L_{01}} &  M_1 \ar[r]^{L_{12}} &  M_2 \ar[r]^{L_{23}} & \cdots \ar[r]^{L_{(k-1)k}} & M_k } \right), \]

On appellera \emph{longueur de $\underline{L}$} l'entier $k$. Si $L$ (resp. $\underline{L}$) désigne une correspondance Lagrangienne (resp. généralisée) de $M$  vers $M'$, on notera $L^T$ (resp. $\underline{L}^T$) la correspondance allant de $M'$ vers $M$, obtenue en renversant les flèches.
\end{defi}

On notera $pt$ la variété symplectique réduite à un point. Une correspondance Lagrangienne entre $pt$ et $M$ est simplement une Lagrangienne de $M$.

Si $\underline{L}$ est une correspondance Lagrangienne généralisée allant de $pt$ à $pt$, l'homologie de Floer matelassée de $\underline{L}$ peut être définine (lorsque celà est possible) comme l'homologie de Floer Lagrangienne
\[ HF(\underline{L}) = HF( L_{01}\times L_{23} \times \cdots , L_{12}\times L_{34} \times \cdots ), \]
où la variété ambiante est le produit de toutes les variétés $M_0^- \times M_1 \times M_2^- \times \cdots $  (voir \cite[Rem 5.2.7 (e)]{WWorient} pour la justification de cette définition lorsque les coefficients sont dans $\zz$). Lorsque $L_{01}\times L_{23} \times \cdots $ et $L_{12}\times L_{34} \times \cdots $ s'intersectent transversalement, le complexe est engendré par les \emph{points d'intersection généralisés}
\[ \I(\underline{L}) =\lbrace (x_0, \cdots , x_k) ~|~ \forall i, (x_i,x_{i+1}) \in L_{i(i+1)}   \rbrace, \]
et la différentielle compte des trajectoires de Floer d'indice 1,  qui ici peuvent s'interpréter comme des "bandes matelassées". Rappelons les définitions de  surface matelassée et de quilts pseudo-holomorphes, qui constituent les notions fondamentales de la théorie de Wehrheim et Woodward :

\begin{defi}Une \emph{surface matelassée}  $\underline{S}$ est la donnée :

\begin{enumerate}
\item[$(i)$] d'une collection de "morceaux" $\underline{S} = (S_k)_{k=1\cdots m}$, c'est-à-dire des surfaces de Riemann munies de structures complexes $j_k$. On indexe ses composantes de bord par un ensemble $\mathcal{B}(S_k)$ : $\partial S_k = \bigcup_{b\in\mathcal{B}(S_k) }{I_{k,b}}$.

\item[$(ii)$] d'une collection $\mathcal{S}$ de coutures : ensembles à deux éléments disjoints deux-à-deux :

$ \sigma \subset \bigcup_{k=1}^{m}\bigcup_{b\in\mathcal{B}(S_k)} I_{k,b}$, et pour chaque $\sigma = \lbrace  I_{k,b}, I_{k',b'}\rbrace$, un difféomorphisme analytique réel $\varphi_\sigma \colon I_{k,b} \rightarrow I_{k',b'}$.

\end{enumerate}\end{defi}

\begin{defi} Soit $\underline{S}$ une surface matelassée comme précédemment, $\underline{M} = (M_k)_{k=1\cdots m}$ une collection de variétés symplectiques, une pour chaque morceau de $\underline{S}$, et  $\underline{L} = \left(  L_\sigma \subset M_k ^- \times M_{k'} , L_{k,b} \subset M_k\right) $ une collection de  correspondances Lagrangiennes, une par couture $\sigma = \lbrace  I_{k,b}, I_{k',b'}\rbrace$, et de Lagrangiennes, une par composante de bord $I_{k,b}$ qui n'est pas dans une couture. Un \emph{quilt (pseudo-holomorphe)} $  \underline{u}\colon  \left( \underline{S} \rightarrow \underline{M}, \underline{L} \right) $ est alors une collection d'applications (pseudo-holomorphes si les $M_k$ sont munies de structures presque-complexes)  $u_i\colon S_i\rightarrow M_i$ satisfaisant aux conditions aux bords et aux coutures suivantes :
 
 \[ (u_k(x), u_k'(\varphi_\sigma (x)) ) \in L_\sigma,  x \in I_{k,b}  ,\]
 \[ u_k(x) \in L_{k,b}, x \in I_{k,b} .\]

\end{defi}
Tout cela peut être résumé dans un diagramme comme celui de la figure \ref{cylindre} :

\begin{figure}[!h]
    \centering
    \def\svgwidth{.65\textwidth}
    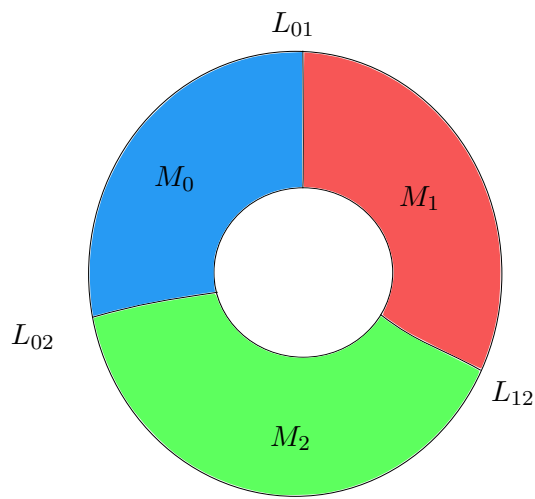
      \caption{Un cylindre matelassé.}
      \label{cylindre}
\end{figure}

\begin{remark}
\begin{enumerate}

\item On dit que $I_{k,b}$ est un bord de $\underline{S}$ s'il n'est pas dans une couture.

\item Parfois on confondra $\underline{S}$ avec la surface obtenue en recollant tous les morceaux entre eux le long des coutures.

\end{enumerate}
\end{remark}

Ainsi, dans ce formalisme, les trajectoires de Floer peuvent être vues comme des bandes matelassées comme dans la figure \ref{bande}.

\begin{figure}[!h]
    \centering
    \def\svgwidth{.65\textwidth}
   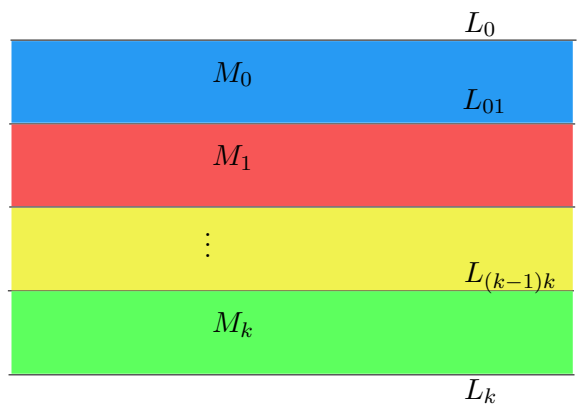
      \caption{Une bande matelassée intervenant dans la différentielle de l'homologie matelassée.}
      \label{bande}
\end{figure}

Weinsten suggère dans \cite{weinsteincat} que les correspondances Lagrangiennes doivent former les morphismes d'une catégorie, dont la composition serait donnée de la manière suivante :

\begin{defi}[Composition géométrique]

Soient $M_0$, $M_1$, $M_2$ trois variétés symplectiques, et $L_{01} \subset M_0\times M_1$, $L_{12} \subset M_1\times M_2$ des correspondances Lagrangiennes. On appelle \emph{composition géométrique} de $L_{01}$ et $L_{12}$  le sous-ensemble :
\[L_{01} \circ L_{12} = \pi_{02}(L_{01}\times M_2 \cap M_0\times L_{12}),\]
où l'on note $\pi_{02}$ la projection
\[\pi_{02}\colon M_0\times M_1\times M_2 \rightarrow M_0\times  M_2 .\]
\end{defi}

\begin{remark}La diagonale $\Delta_{M} = \lbrace (x,x)~|~ x\in M \rbrace$ joue le rôle de l'identité pour cette composition.
\end{remark}

Malheureusement, la correspondance ainsi obtenue n'est pas toujours une correspondance Lagrangienne, elle n'est d'ailleurs pas nécessairement lisse. L'idée de Wehrheim et Woodward est de ne s'autoriser à composer deux correspondances uniquement lorsque le critère suivant est satisfait :

\begin{defi}[Composition géométrique plongée]
Une composition géométrique $L_{01} \circ L_{12}$ est dite \emph{plongée} si :

\begin{itemize}
\item $L_{01} \times M_2$ et $M_0 \times L_{12}$ s'intersectent transversalement,

\item $\pi_{02}$ induit un plongement de $L_{01}\times M_2 \cap M_0\times L_{12}$ dans $M_0\times  M_2$.
\end{itemize}

\end{defi}

En plus de garantir que $L_{01} \circ L_{12}$ est à son tour une correspondance Lagrangienne, l'homologie de Floer matelassée se comporte bien dans ces conditions : nous verrons dans le théorème \ref{compogeom} que si $L_{01} \circ L_{12}$ est plongée et sous certaines hypothèses supplémentaires, 
\[ HF( \cdots,  L_i, L_{i+1}, \cdots ) \simeq HF( \cdots,  L_i \circ L_{i+1}, \cdots ). \]

Ainsi, Wehrheim et Woodward définissent leur catégorie symplectique étendue $ \mathrm{Symp}_\tau^ \#$ (dans \cite[Def. 3.1.7]{WWfft}) comme la catégorie ayant pour objets des variétés symplectiques (soumises à certaines hypothèses de monotonie), et pour morphismes des classes d'équivalences de correspondances Lagrangiennes généralisées (elles aussi soumises à des hypothèses), où la relation d'équivalence est engendrée par :
\[ ( \cdots,  L_i, L_{i+1}, \cdots ) \sim ( \cdots,  L_i \circ L_{i+1}, \cdots ), \text{si la composition est plongée,}\]
ainsi que des hypothèses de monotonie garantissant que l'on puisse définir l'homologie de Floer.

\begin{remark}Il résultera de la théorie de Cerf que les correspondances Lagrangiennes  généralisées que nous allons construire sont équivalentes à des correspondances Lagrangiennes généralisées de longueur 2. Ceci est un fait général pour toutes correspondances Lagrangiennes  généralisées, observé par Weinstein dans \cite{Weinsteinremark}. 
\end{remark}

La catégorie $ \mathrm{Symp}_\tau^ \#$ constituera un modèle que l'on va adapter pour être dans le cadre de \cite[Assumption 2.5]{MW}. Nous donnons quelques définitions préliminaires avant de définir (c.f. définition \ref{defcat}) la catégorie $\Symp$ que nous utiliserons.

Les variétés symplectiques que nous considèrerons seront munies d'hypersurfaces symplectiques. Afin que l'homologie de Floer soit bien définie, les correspondances Lagrangiennes devront vérifier la condition suivante de compatibilité (\cite[Def. 6.2]{MW}):

\begin{defi}[Correspondance Lagrangienne compatible avec une paire d'hypersurfaces] \label{correspcompat}
Soient $M_0$, $M_1$ deux variétés symplectiques, $R_{0}\subset M_0$, $R_{1} \subset M_1 $ deux hypersurfaces symplectiques. Une correspondance Lagrangienne $L_{01} \subset M_0\times M_1$ est dite \emph{compatible avec la paire $(R_0,R_1)$} si, d'une part,
\[ (R_0 \times M_1 ) \cap L_{01} = (M_0 \times R_1 ) \cap L_{01} = (R_0 \times R_1 ) \cap L_{01}, \]
et si $L_{01}$ intersecte $R_0 \times M_1$ et $M_0 \times R_1$ transversalement.\end{defi}

Dans ces conditions, pour  des voisinages tubulaires $\tau_0$ et $\tau_1$ de $R_0$ et $R_1$ suffisamment petits :
\[ \tau_0\colon N_{R_0} \rightarrow M_0,\   \tau_1\colon N_{R_1} \rightarrow M_1, \]
l'image réciproque $ (\tau_0 \times \tau_1)^{-1}(L_{01}) \subset N_{R_0} \times N_{R_1} $ est le graphe d'un isomorphisme de fibrés
\[ \varphi \colon (\widetilde{N}_{R_0})|_{(R_0 \times R_1 ) \cap L_{01}} \rightarrow (\widetilde{N}_{R_1})|_{(R_0 \times R_1 ) \cap L_{01}}, \]
où $\widetilde{N}_{R_0}$ est le fibré normal de $R_0 \times M_1 \subset M_0 \times M_1$, et $\widetilde{N}_{R_1}$ est le fibré normal de $M_0 \times R_1 \subset M_0 \times M_1$.


Notons que si l'une des variétés symplectiques est un point, une Lagrangienne $L$ est compatible avec une hypersurface $R$ si et seulement si $L$ et $R$ sont disjointes.

Soient   $\underline{S}$ une surface matelassée,  $\underline{u}\colon \underline{S}\rightarrow (\underline{M},\underline{L})$ un quilt, et $\underline{R}\subset \underline{M}$ une famille d'hypersurfaces telle que les correspondances Lagrangiennes aux coutures soient compatibles avec les hypersurfaces correspondantes, rappelons la définition du nombre d'intersection $\underline{u}. \underline{R}$ :

Soit $\underline{U}\subset \underline{S}$ un voisinage ouvert de $\underline{u}^{-1}(\underline{R})$ dont l'image de chaque morceau $S_i$ est contenue dans les voisinages tubulaires $\tau_i$ de $R_i$. Chaque application $u_i$ peut alors être vue comme une section du fibré en droites complexes tiré en arrière $u_i^* N_{R_i}$. Tous ces fibrés se recollent en un fibré sur $\underline{U}$ en utilisant les isomorphismes $\varphi$, et les sections $u_i$ se recollent en une section globale de ce fibré, non-nulle au-dessus du bord $\partial \underline{U}$.  De ce fait, le fibré est trivial sur ce bord et s'étend à un fibré au-dessus de $\underline{S}$, et les sections se prolongent en des sections globales, non-nulles en dehors de $\underline{U}$. Le nombre d'intersection $\underline{u}\cdot \underline{R}$ est alors défini comme étant le nombre d'Euler  de ce fibré.

Le lemme suivant est prouvé dans \cite[Lemma 6.4]{MW} lorsque la surface matelassée est plusieurs bandes parallèles, sa preuve s'adapte à n'importe quelle surface matelassée. 

\begin{lemma}(\cite[Lemma 6.4]{MW}) Le nombre d'intersection $\underline{u}\cdot \underline{R}$ ne change pas si l'on perturbe $\underline{u}$ par une homotopie préservant les conditions aux bords et coutures. De plus, si $\underline{u}$ est pseudo-holomorphe, $\underline{R}$ presque complexes et $\underline{u}$ et $\underline{R}$ s'intersectent transversalement, alors  ce nombre est donné par :
\[ \underline{u} \cdot \underline{R} = 
\sum_{j=0}^k 
\# \{ z_j \in \operatorname{int}(S_j) | u_j(z_j)
\in R_j \} + \frac{1}{2} \# \{ z_j \in \partial S_j | u_j(z_j) \in R_j \} .\]
\end{lemma}


\begin{defi}\label{defcat}
On appelle $\Symp$ la catégorie suivante :  

Ses objets sont les 5-uplets $(M, \omega , \tilde{\omega}  , R, \tilde{J})$ satisfaisant les hypothèses $(i)$, $(ii)$, $(iii)$, $(iv)$, $(v)$, $(x)$, $(xi)$ et $(xii)$ de  \cite[Assumption 2.5]{MW}. Rappelons ces dernières :
\begin{enumerate}
\item[$(i)$] $(M,\omega)$ est une variété symplectique compacte.

\item[$(ii)$] $\tilde{\omega}$ est une 2-forme fermée sur $M$.

\item[$(iii)$] Le lieu de dégénérescence $R\subset M$ de $\tilde{\omega}$ est une hypersurface symplectique pour $\omega$.

\item[$(iv)$] $\tilde{\omega}$ est $\frac{1}{4}$-monotone, c'est-à-dire $\left[ \tilde{\omega} \right] = \frac{1}{4} c_1(TM) \in H^2(M; \rr)$.

\item[$(v)$] Les restrictions de $\tilde{\omega}$ et $\omega$ à $M\setminus R$ définissent la même classe de cohomologie dans $H^2(M\setminus R; \rr)$.

\item[$(x)$] Le nombre de Chern minimal $N_{M\setminus R}$ (par rapport à $\omega$) est un multiple strictement positif de 4, de sorte que le nombre de Maslov minimal $N = 2N_{M\setminus R}$ soit un multiple strictement positif de $8$.

\item[$(xi)$] $\tilde{J}$ est une structure presque complexe $\omega$-compa\-tible sur $M$, $\tilde{\omega}$-compa\-tible sur $M\setminus R$, et telle que $R$ est une hypersurface presque complexe pour $\tilde{J}$.

\item[$(xii)$] Toute sphère $\tilde{J}$-holomorphe dans $M$ d'indice nul, nécessairement incluse dans $R$ par monotonie, a un nombre d'intersection avec $R$ égal à un multiple strictement négatif de 2.

\end{enumerate}

L'ensemble des morphismes entre deux objets consiste en des chaînes de morphismes élémentaires $\underline{L} = (L_{01}, L_{12}, \cdots ) $, modulo une relation d'équivalence :

Les morphismes élémentaires sont des correspondances $L_{i(i+1)} \subset  M_i^- \times M_{i+1}$,   qui sont Lagrangiennes pour les formes monotones $\tilde{\omega}_i$, simplement con\-nexes, $(R_i, R_{i+1})$-compatibles au sens de  la définition \ref{correspcompat}, telles que $L_{i(i+1)} \setminus \left( R_i \times R_{i+1} \right)$ est spin, et telles que tout disque pseudo-holomorphe de $M_i^- \times M_{i+1}$ à bord dans $L_{i(i+1)}$ et d'aire nulle intersecte $(R_i, R_{i+1})$ en un multiple strictement positif de $-2$.

La relation d'équivalence sur les chaînes de morphismes est engendrée par l'identification suivante : on identifie $(L_{01},\cdots ,  L_{(i-1)i},L_{i(i+1)}, \cdots )$ à la composée $(L_{01},\cdots  L_{(i-1)i}\circ  L_{i(i+1)}, \cdots )$ si la composition de $L_{(i-1)i}$ et $L_{i(i+1)}$ est plongée, simplement connexe, $(R_{i-1} , R_{i+1} )$-compatible, spin en dehors de $R_{i-1}\times R_{i+1}$, vérifie l'hypothèse ci-dessus concernant les disques pseudo-holomorphes, ainsi que l'hypothèse suivante : tout cylindre matelassé pseudo-holomorphe comme dans la figure \ref{cylindre} d'aire nulle et de conditions aux coutures dans $L_{(i-1)i}$,   $L_{i(i+1)}$ et $L_{(i-1)i}\circ  L_{i(i+1)}$ intersecte $(R_{i-1} , R_i, R_{i+1})$ en un nombre plus petit que $-2$.
\end{defi}

\begin{remark}\label{monotonie} Sous ces hypothèses, les correspondances Lagrangiennes généralisées sont automatiquement monotones : si $\underline{x},\underline{y} \in  \I(\underline{L})$, sont des points d'intersections généralisés, et   $\underline{u}$ une bande matelassée à conditions aux coutures données par $\underline{L}$ et ayant pour limites $\underline{x},\underline{y}$, alors l'aire symplectique de $\underline{u}$ vaut $A(\underline{u}) = \frac{1}{8} I(\underline{u}) + c(\underline{x},\underline{y})$, avec $c(\underline{x},\underline{y})$ ne dépendant que des points $\underline{x}$ et $\underline{y}$. Cela provient de la monotonie des variétés symplectiques et de la simple connexité des correspondances, voir \cite[Lemma 2.8]{MW}.

\end{remark}

\subsection{Définition de l'homologie matelassée}\label{defhomolgiematelassée}

Soit $\underline{L} = (L_{i(i+1)})_{i = 0\cdots k} $ un morphisme de $\Symp$ de longueur $k+1$ de $pt$ vers $pt$. Notons $\underline{M} = (M_i, \omega_i , \tilde{\omega}_i, R_i , \tilde{J}_i)_{i = 0\cdots k+1}$ les objets intermédiaires, avec $M_0 = M_{k+1} = pt$. On notera :

\[\underline{L} = \left(  \xymatrix{   M_0 \ar[r]^{L_{01}} &  M_1 \ar[r]^{L_{12}} &  M_2 \ar[r]^{L_{23}} & \cdots \ar[r]^{L_{k(k+1)}} & M_{k+1} } \right). \]

Notons $\mathcal{J}(M_i, \mathrm{int}\left\lbrace \omega_i = \tilde{\omega}_i \right\rbrace ,  \tilde{J}_i)$ l'ensemble des structures presque complexes sur $M_i$ qui sont $\omega_i$-compatibles, et qui coïncident avec $\tilde{J}_i$ en dehors de $\mathrm{int}\left\lbrace \omega_i = \tilde{\omega}_i \right\rbrace$. Notons également \[\mathcal{J}_t(M_i, \mathrm{int}\left\lbrace \omega_i = \tilde{\omega}_i \right\rbrace ,  \tilde{J}_i) = \mathcal{C}^\infty([0,1], \mathcal{J}(M_i, \mathrm{int}\left\lbrace \omega_i = \tilde{\omega}_i \right\rbrace ,  \tilde{J}_i))\] l'espace des structures presque complexes dépendantes du temps.

Introduisons des perturbations Hamiltoniennes, afin de garantir que les intersections soient transverses. Soient $\underline{H} = (H_i)_{i = 1\cdots k}$ des  Hamiltoniens, $H_i : M_i\times \rr \rightarrow \rr$ à support contenus dans $\mathrm{int}\left\lbrace \omega_i = \tilde{\omega}_i \right\rbrace$. Notons $\varphi_i$ le temps 1 du flot de $X_{H_i}$ et
\begin{align*}
& \widetilde{L}_{i(i+1)} = \left\lbrace (\varphi_i(x_i),x_{i+1}) ~|~ (x_i,x_{i+1}) \in L_{i(i+1)} \right\rbrace   \\
 & \widetilde{L}_{(0)} = \widetilde{L}_{0} \times \widetilde{L}_{12} \times \cdots \\
& \widetilde{L}_{(1)} = \widetilde{L}_{01} \times \widetilde{L}_{23} \times \cdots.
\end{align*}

Supposons les points d'intersection généralisés $\I(\underline{L})$ contenus dans le produit des $\mathrm{int}\left\lbrace \omega_i = \tilde{\omega}_i \right\rbrace$,  ce qui sera le cas pour définir l'homologie HSI. Pour un choix approprié des Hamiltoniens, l'intersection $\widetilde{L}_{(0)} \cap \widetilde{L}_{(1)}$ est transverse. L'ensemble fini $\widetilde{L}_{(0)} \cap \widetilde{L}_{(1)}$ est alors en bijection avec l'ensemble des points d'intersection généralisés perturbés $\I_{\underline{H}}(\underline{L})$ consistant en des $k$-uplets $p_i\colon [0,1] \rightarrow M_i$ tels que $\frac{\mathrm{d}p_i}{\mathrm{d}t} = X_{H_i}$, et $(p_i(1), p_{i+1}(0) )\in L_{i(i+1)}$. En effet, étant donné que $p_i(1) = \varphi_i(p_i(0))$, ils correspondent aux points $(x_1, \cdots , x_k)$ de $\widetilde{L}_{(0)} \cap \widetilde{L}_{(1)},$ si l'on pose $x_i = p_i(0)$.

Soit $\widetilde{\mathcal{M}}(\underline{x},\underline{y})$ l'ensemble des applications $u_i\colon \rr \times [0,1]\to M_i$ telles que, avec $s\in\rr$ et $t\in [0,1]$ :

\[ \begin{cases} 0 = \partial_s u_i + J_t ( \partial_t u_i - X_{H_i}) \\
\mathrm{lim}_{s\rightarrow - \infty}{u_i(s,t)} = \varphi_i^t (y_i)\\
\mathrm{lim}_{s\rightarrow + \infty}{u_i(s,t)} = \varphi_i^t (x_i)\\
(u_i(s,1), u_{i+1}(s,0) )\in L_{i(i+1)} \\
\underline{u} \cdot \underline{R} = 0 \\
I(\underline{u}) = 1.
\end{cases} \]
L'espace des trajectoires de Floer matelassées est alors le quotient $\mathcal{M}(\underline{x},\underline{y})  = \widetilde{\mathcal{M}}(\underline{x},\underline{y}) /\rr$  par la reparamétrisation en $s$. Pour des choix génériques d'Hamiltoniens $\underline{H}$ et de structures presque complexes $\underline{J}$, c'est un ensemble fini.

Le complexe de Floer est alors défini comme \[CF(\underline{L},\underline{H}, \underline{J}) = \bigoplus_{\underline{x}\in \I_{\underline{H}}(\underline{L})} \zz \underline{x},\] et muni de la différentielle  définie par \[\partial \underline{x} = \sum_{\underline{y}} \# \mathcal{M}(\underline{x},\underline{y}) \underline{y},\] 
où $\# \mathcal{M}(\underline{x},\underline{y}) = \sum_{u\in \mathcal{M}(\underline{x},\underline{y}) }{o(u)}$, avec $o(\underline{u}) = \pm 1$ l'orientation du point $\underline{u}$ dans l'espace des modules construite dans \cite{WWorient} à partir de l'unique structure spin relative sur $\underline{L}$.

Rappelons le résultat de Manolescu et Woodward qui rend l'homologie de Floer bien définie pour des éléments de $Hom_{Symp}(pt,pt)$ :

\begin{theo}(\cite[Theorem 6.5]{MW}) Soient $\underline{L}$ et $\underline{H}$ comme ci-dessus. Il existe un sous-ensemble $G_\delta$-dense 
\[\mathcal{J}_t^{reg}(M_i, \mathrm{int}\left\lbrace \omega_i = \tilde{\omega}_i \right\rbrace ,  \tilde{J}_i) \subset \mathcal{J}_t(M_i, \mathrm{int}\left\lbrace \omega_i = \tilde{\omega}_i \right\rbrace ,  \tilde{J}_i)\] de structures presque complexes régulières pour lesquelles la différentielle est finie et vérifie $\partial ^2 = 0$. Alors, l'homologie matelassée $HF (\underline{L})$ est bien définie pour des structures presque complexes et des perturbations Hamiltoniennes génériques, et est indépendante de ces choix, hormis éventuellement de la structure presque complexe de référence.
\end{theo}

\begin{remark}Pour les variétés qui interviendont par la suite, le choix de la structure presque complexe de référence n'interviendra pas, en effet celle-ci est choisie dans un espace contractile, voir \cite[Remark 4.13]{MW}.
\end{remark}

\paragraph{Graduation}L'hypothèse sur le nombre de Maslov minimal permet de définir une $\Z{8}$-graduation relative sur le complexe de chaînes : si $x$ et $y$ sont deux points d'intersection généralisés, et $u$, $v$ deux trajectoires de Floer matelassées (non-nécessairement pseudo-holomorphes) les reliant, $I(u) = I(v)$ modulo 8. On note alors $I(x,y) \in \Z{8}$ cette quantité commune. La différentielle est alors de degré $1$. Il s'en suit que $I(x,y)$ définit une graduation relative sur $HF(\underline{L})$.

Rappelons enfin le résultat suivant, qui prouve l'invariance de l'homologie de Floer matelassée par composition géométrique plongée :

\begin{theo}(\cite[Theorem 6.7]{MW})\label{compogeom} Soit $\underline{L}$ une  correspondance Lagrangienne généralisée comme précédemment. Si de plus la composition géométrique $L_{i-1,i} \circ  L_{i, i + 1}$ est plongée, simplement connexe, $(R_{i-1} , R_{i+1})$-compa\-tible, et telle que le nombre d'intersection de tout cylindre matelassé pseudo-holomorphe avec $(R_{i-1} , R_i,  R_{i+1})$ est plus petit que $-2$, alors $HF (\underline{L})$ est canoniquement isomorphe à $HF (\cdots L_{i-1,i} \circ L_{i, i + 1} \cdots )$.
\end{theo}


%

\section{Cobordismes à bords verticaux et théorie de Cerf connexe}\sectionmark{théorie de Cerf}

Commençons par définir la catégorie de cobordisme qui nous intéressera.

\begin{defi}[Catégorie des cobordismes à bords verticaux]\label{chemincob}
On appelle \emph{catégorie des cobordismes à bords verticaux, avec classe d'homologie de degré 1 à coefficients dans $\Z{2}$}, que l'on notera $\Cob$, la catégorie dont :

\begin{itemize}

\item les objets sont les couples $(\Sigma,p)$, où $\Sigma$ est une surface compacte, connexe, orientée, à bord connexe, et $p\colon \rr/\zz \rightarrow \partial\Sigma$ est un difféomorphisme (paramétrage).

\item les morphismes de $(\Sigma_0,p_0)$ vers $(\Sigma_1,p_1)$ sont des classes de difféomorphismes  de 5-uplets $(W, \pi_{\Sigma_0},  \pi_{\Sigma_1}, p,  c)$, où $W$ est une 3-variété compacte orientée à bord, $\pi_{\Sigma_0}$,  $\pi_{\Sigma_1}$ et $p$ sont des plongements de $\Sigma_0$, $\Sigma_1$ et $\rr/\zz \times [0,1]$ dans $\partial W$, le premier renversant l'orientation, les deux autres préservant l'orientation, et tels que 

\[\partial W = \pi_{\Sigma_0}(\Sigma_0)\cup  \pi_{\Sigma_1}(\Sigma_1)\cup  p(\rr/\zz \times [0,1]), \]
 $\pi_{\Sigma_0}(\Sigma_0)$  et $\pi_{\Sigma_1}(\Sigma_1)$ sont disjoints, pour $i=0,1$, \[\pi_{\Sigma_i}(\Sigma_i)\cap p(\rr/\zz \times [0,1])  = \pi_{\Sigma_i}(p_i(\rr/\zz)) = p(\rr/\zz \times \lbrace i\rbrace),\]
  $p(s,i) = \pi_{\Sigma_i}(p_i(s))$, et  $c \in H_1 (W, \Z{2})$.
  
On appellera   $p(\rr/\zz \times [0,1]) $ partie verticale de $\partial W$, et on la notera $\partial^{vert} W$.

Deux tels 5-uplets $(W, \pi_{\Sigma_0},  \pi_{\Sigma_1}, p,  c)$ et $(W', \pi'_{\Sigma_0},  \pi'_{\Sigma_1}, p',  c')$ sont dits équivalents s'il existe un difféomorphisme $\varphi \colon W \rightarrow W'$ compatible avec les plongements et préservant la classe. 

\item la composition des morphismes consiste à recoller le long des plongements, et à ajouter les classes d'homologie, on la notera $\cup$, ou $\cup_\partial$.
\end{itemize}

\end{defi}

On n'associera pas une correspondance Lagrangienne à n'importe quel cobordisme, mais seulement à certains, dits "élémentaires" (au sens de la théorie de Morse). Nous verrons qu'un cobordisme quelconque pourra toujours se décomposer en un nombre fini de cobordismes élémentaires, définissant ainsi une succession de correspondances Lagrangiennes, et donc un morphisme de $\Symp$. Ce paragraphe a pour objectif de préparer la démonstration du fait que cette construction ne dépend pas du découpage utilisé.

Il s'agit essentiellement d'adapter les résultats de \cite{connectedcerf} au cadre qui nous intéresse. Rappelons leur résultat principal, qui est faux en dimension $1+1$ :

\begin{theo}[\cite{connectedcerf}]\label{GWWtheo}
Soit $n\geq 2$,
\begin{enumerate}

\item Tout $(n+1)$-cobordisme connexe entre $n$-variétés connexes admet une décomposition en cobordismes élémentaires dont tous les niveaux intermédiaires sont connexes. Une telle décomposition sera appelée \emph{décomposition de Cerf}.
\item Etant données deux telles décompositions, il est possible de passer de l'une à l'autre par un nombre fini de \emph{mouvements de Cerf}, à savoir une difféo-équivalence, un ajout ou une suppression de cylindres, un ajout ou une suppression de paires de naîssance-mort, ou une interversion de points critiques ("critical point switch") (voir \cite{connectedcerf} pour les définitions, ou la définition \ref{mvtcerf} qui suit).
\end{enumerate}
\end{theo}

Définissons une catégorie de cobordisme intermédiaire :

\begin{defi}[Catégorie des cobordismes à bords verticaux élémentaires]\label{chemincobelem}

On appelle $\Cobelem$ la catégorie dont les objets sont les mêmes que ceux de $\Cob$, et dont les morphismes sont des chaînes de 6-uplets 
\[(W_k, \pi_{\Sigma_k},  \pi_{\Sigma_{k+1}}, p_k, f_k,  c_k),\]
où $(W_k, \pi_{\Sigma_k},  \pi_{\Sigma_{k+1}},p_k, c_k)$ est un cobordisme de $(\Sigma_k,p_k)$ vers $(\Sigma_{k+1},p_{k+1})$ comme précédemment. La fonction  $f_k\colon W_k \rightarrow [0,1]$ est une fonction de Morse telle que, pour $i=0,1$, $f_k^{-1}(i) = \pi_{\Sigma_{k+i}}(\Sigma_{k+i})$, admettant au plus un point critique dans l'intérieur de $W_k$ et pas de point critique sur $\partial W_k$, et $f(p(s,t)) = t \in [0,1]$. Enfin $c_k \in H_1 (W_k, \Z{2}))$. De tels cobordismes sont appelés \emph{élémentaires}. On notera ces morphismes 

\[\underline{W} = (W_1, f_1, p_1,  c_1) \odot (W_2, f_2, p_2,  c_2) \odot \cdots \odot (W_k, f_k, p_k,  c_k),\]
et la composition, que l'on notera $\odot$, consiste à concaténer deux chaînes.
\end{defi}

\begin{remark}\begin{enumerate}
\item Lorsqu'il n'y aura pas d'ambiguïté, on omettra parfois de préciser les plongements, et on notera simplement $(W,c)$ la donnée de $(W, \pi_{\Sigma_0},  \pi_{\Sigma_1}, f, p,  c)$.

\item Il y a un foncteur $\Cobelem\rightarrow\Cob$ qui ne change pas les objets et qui consiste à recoller une cha\^ine $\underline{W}$ en un seul cobordisme, à ajouter les classes, et à oublier les fonctions de Morse.

\end{enumerate}
\end{remark}

Afin d'obtenir dans la proposition \ref{cerfchemin} un résultat analogue au théorème \ref{GWWtheo} pour des cobordismes à bord vertical munis d'une classe d'homologie de degré 1 à coefficients dans $\Z{2}$, nous définissons des mouvements similaires pour de tels cobordismes.

\begin{defi}[Mouvements de Cerf] \label{mvtcerf}
Si $\underline{W} = W_1 \odot \cdots \odot W_k$ et $\underline{W'} = W_1' \odot \cdots\odot W_l'$ sont des morphismes de  $\Cobelem$, on appelle \emph{mouvements de Cerf} le fait de remplacer $\underline{W}$ par $\underline{W'}$ en opérant l'une des modifications suivantes :

\begin{enumerate}
\item[$(i)$] \emph{Difféo-équivalence} : On dit qu'un difféomorphisme $\varphi \colon W\rightarrow W'$ est une difféo-équivalence entre deux cobordismes $(W, \pi_{\Sigma_0},  \pi_{\Sigma_1}, f, p,  c)$ et $(W', \pi'_{\Sigma_0},  \pi'_{\Sigma_1}, f', p',  c')$ si $f'\circ\varphi = f$, $\varphi \circ \pi_{\Sigma_i} = \pi'_{\Sigma_i}$, pour $i=0,1$, $p' = \varphi \circ p$, et $c' = \varphi_* c$.

\item[$(ii)$] \emph{Ajout ou suppression de cobordismes triviaux} : Le fait d'ajouter ou de supprimer un cobordisme de la forme : 

\[(W = \Sigma \times [0,1], \pi_{\Sigma_i} = id_{\Sigma} \times \lbrace i \rbrace,   f(s,t) = t, p(s,t) = (p(s),t),  c = 0)\]

\item[$(iii)$] \emph{Ajout ou suppression d'une paire de naîssance-mort} : Une paire de naîssance-mort est  une chaîne \[(W, \pi_{\Sigma_0},  \pi_{\Sigma_1}, f, z,  c =0) \odot(W', \pi_{\Sigma_1}',  \pi_{\Sigma_0}', f', z',  c' = 0)\] telle que la réunion $(W\cup_{\Sigma_1} W',f\cup f',c + c' =0)$ est difféo-équivalente à un cylindre.

\item[$(iv)$] \emph{Interversion de points critiques} ("critical point switch") : Soient $s_1$ et $s_2 \subset$ deux sphères d'attachement d'anses de $\Sigma$ disjointes, $h_1, h_2$ les anses correspondantes, $W_1$ le cobordisme correspondant à l'attachement de $h_1$, $W_2$ le cobordisme correspondant à l'attachement de $h_2$ après avoir attaché $h_1$, $W_2'$ le cobordisme correspondant à l'attachement de $h_2$, $W_1'$ le cobordisme correspondant à l'attachement de $h_1$ après avoir attaché $h_2$. Le mouvement consiste à remplacer $W_1 \odot W_2$ par $W_2' \odot W_1'$.

\item[$(v)$] \emph{Glissement de classe d'homologie} :
Si $c_i + c_{i+1} = d_i + d_{i+1}$ dans $H_1(W_i \cup W_{i+1};\Z{2})$, et $W_i$ ou $W_{i+1}$ est un cylindre, remplacer $(W_i, c_i)\odot (W_{i+1}, c_{i+1})$ par $(W_i, d_i)\odot (W_{i+1}, d_{i+1})$ dans $(\underline{W}, \underline{c})$ .
\end{enumerate}
\end{defi}

\begin{prop}\label{cerfchemin}

\begin{enumerate}
\item Tout cobordisme à bord vertical $(W,p,c)$ admet une décomposition de Cerf (i.e. $\Cobelem \rightarrow \Cob$ est surjectif).
\item Une fois recollées, deux chaînes de cobordismes élémentaires définissent le même morphisme dans $\Cob$ si et seulement si l'on peut passer de l'une à l'autre par l'un des cinq mouvements de Cerf.
\end{enumerate}
\end{prop}

\begin{proof}

1. De même que pour le cas sans bord vertical (\cite[Lemma 2.5]{connectedcerf}) une telle décomposition est obtenue à partir d'une fonction de Morse "excellente" (c'est-à-dire injective sur l'ensemble de ses points critiques), sans points critiques d'indice $0$ et $3$, et telle que si $p$ et $q$ sont deux points critiques tels que $ ind~p <ind~q $, alors $f(p) < f(q)$. Alors, si $b_0 = min~f < b_1 <\cdots b_k = max~f$ est une suite de valeurs régulières telle que $[b_i, b_{i+1} ]$ contient au plus un point critique, $W_i = f^{-1}([b_i, b_{i+1} ])$ est un cobordisme connexe (garanti par la condition qu'il n'y a pas de points critiques d'indice $0$ et $3$) entre surfaces connexes. En effet, la première surface non-connexe correspondrait à l'attachement d'une 2-anse, puis la première surface à nouveau connexe correspondrait à l'attachement d'une 1-anse, or on a supposé que les 1-anses étaient attachées avant les 2-anses.

Il faut ici supposer de plus que sur le bord vertical, $f(p(s,t)) = K t$, pour une constante $K >0$. Nous prétendons qu'il est possible de trouver une telle fonction. En effet, partant d'une fonction de Morse telle que $f(p(s,t)) = K t$ au voisinage du bord vertical, il est possible de réarranger les points critiques et de supprimer les minimums et les maximums sans affecter les valeurs au bord : il suffit de prendre un pseudo-gradient parallèle au bord vertical, ceci assure que les sphères d'attachement d'anses sont confinées à l'intérieur de $W$, puis le même raisonnement que dans le cas sans bord vertical s'applique.

Par ailleurs, la classe $c$ peut se décomposer en classes $c_i \in H_1( W_i , \Z{2})$ : il suffit de choisir un représentant $C\subset W$ de dimension 1  qui n'intersecte pas les surfaces intermédiaires. Un représentant générique intersecte les surfaces en un nombre pair de points, qui peuvent être éliminés deux par deux.

2. Si l'on se donne deux telles fonctions de Morse, il est possible de les relier par un chemin de fonctions ayant un nombre fini de singularités type naîssance mort, ou d'interversions de points critiques, en gardant les mêmes valeurs sur le bord vertical. Le reste de l'argument est en tout point analogue à la preuve de \cite[Theorem 3.4]{connectedcerf}.

Notons que pour les mouvements $(iii), (iv)$, on a supposé que la classe d'homologie était nulle. Ceci est possible quitte à rajouter des cobordismes triviaux, et parce que l'on peut isoler l'homologie à l'extérieur d'une paire de naîssance/mort. 
\end{proof}

Nous obtenons ainsi un critère permettant de factoriser un foncteur $\textbf{F}: \Cobelem \rightarrow  \Symp $ par le foncteur de recollement $\Cobelem \rightarrow \Cob $.

\begin{cor}\label{factorisation} Soit un foncteur $\textbf{F}: \Cobelem \rightarrow  \Symp $ vérifiant

\begin{enumerate}
\item[$(i)$] $\textbf{F}(W,c) = \textbf{F}(W',c)$ dès lors que les cobordismes $(W,c)$ et $(W',c)$ sont difféo-équivalents.

\item[$(ii)$] 

$\textbf{F}(W) = \Delta_{F(S)}$, si $W$ est un cobordisme trivial sans classe d'homologie. 

\item[$(iii)$] Si $W\odot W'$ est une paire de naîssance/mort, la composition géométrique $\textbf{F}(W,0) \circ \textbf{F}(W',0)$ satisfait les hypothèses du théorème \ref{compogeom} et vaut la diagonale $\Delta_{F(S)}$.

\item[$(iv)$] Si $W_2' \odot W_1'$ est obtenue à partir de $W_1 \odot W_2$ par une interversion de points critiques, $\textbf{F}(W_1,0) \circ \textbf{F}(W_2,0) = \textbf{F}(W_2',0) \circ \textbf{F}(W_1',0)$ et ces compositions satisfont les hypothèses du théorème \ref{compogeom}.

\item[$(v)$] Si $c+c' =d+d'$, alors $\textbf{F}(W,c) \circ \textbf{F}(W',c') = \textbf{F}(W,d) \circ \textbf{F}(W',d')$, et la composée à droite ou à gauche de $\textbf{F}(S\times I,c)$ avec tout autre morphisme satisfait les hypothèses du théorème \ref{compogeom}. 

\end{enumerate}

Alors $\textbf{F}$ se factorise en un foncteur $\Cob \rightarrow \Symp $.

\end{cor}
\arnaque

\chapter{Construction de l'homologie Instanton-Symplectique d'une variété munie d'une classe}\label{chapHSI}
\chaptermark{Définition du foncteur}

Afin de construire le foncteur $\Cob \rightarrow \Symp $ qui nous intéresse, on commencera par construire un foncteur $\Cobelem \rightarrow \Symp $ (trois premiers paragraphes), puis on vérifiera qu'il se factorise par le foncteur de recollement précédent (quatrième paragraphe). Enfin, dans le dernier paragraphe nous définirons les groupes d'homologie "Instanton-Symplectique" $HSI(Y,c)$ associés à une 3-variété Y munie d'une classe $c\in H_1(Y;\Z{2})$.

\section{Espace des modules étendu}
Les espaces de modules qui apparaitront par la suite seront toujours associés au groupe de Lie $SU(2)$. On notera $\mathfrak{su(2)}$ son algèbre de Lie, identifiée aux matrices $2\times 2$ antihermitiennes de trace nulle, que l'on munira du produit scalaire usuel $<a,b> = \mathrm{Tr} (a b^*) = -\mathrm{Tr} (ab)$. Ainsi, on identifiera $\mathfrak{su(2)}$ et $\mathfrak{su(2)}^*$ à l'aide de ce produit scalaire.


Soit $(\Sigma,p)$ une surface à bord paramétré comme dans la définition \ref{chemincob}, on lui associe l'espace des modules étendu $ \Mg (\Sigma,  p)$ défini et étudié par Jeffrey dans \cite{jeffrey}. Rappelons sa définition :

\begin{defi}(Espace des modules associé à une surface, \cite[Def. 2.1]{jeffrey}) 
Soit l'espace des connexions plates :
\[ \mathscr{A}_F^\mathfrak{g}(\Sigma)  = \lbrace A \in \Omega^{1} (\Sigma )\otimes \mathfrak{su(2)}\ |\ F_A = 0,\  A_{|\nu \partial \Sigma} = \theta ds \rbrace, \]
où $\nu \partial \Sigma$ désigne un voisinage tubulaire du bord (non fixé), et $s$ désigne le paramètre de $\rr/\zz$, et le groupe
 
 \[ \Gc (\Sigma ) = \left\lbrace u \colon \Sigma \rightarrow SU(2)\ |\ u_{|\nu \partial \Sigma} = I \right\rbrace \] agit par transformations de jauge.
L'espace des modules étendu est défini comme le quotient \[  \Mg (\Sigma, p) = \mathscr{A}_F^\mathfrak{g}(\Sigma) /\Gc (\Sigma ),\]
\end{defi}

La proposition suivante fournit une description explicite de cet espace :
\begin{prop}\label{riemannhilbert} (\cite[Prop. 2.5]{jeffrey})
Soit $* \in \partial \Sigma$ un point base (on prendra usuellement $* = p(0)$). La valeur $\theta ds$ de la connexion  au voisinage du bord et l'holonomie fournissent une identification de $\Mg (\Sigma, p)$ avec \[ \left\lbrace  (\rho, \theta)\in Hom(\pi_1(\Sigma, *), SU(2)) \times \mathfrak{su(2)}\ |\ e^\theta = \rho(p) \right\rbrace .\]
En particulier, une présentation du groupe fondamental
\[ \pi_1 (\Sigma, *) = \langle \alpha_1 , \beta_1 , \cdots \alpha_h , \beta_h\rangle,\]
telle que  $p = \prod_{i=1}^{h}{[\alpha_i , \beta_i ]} $ induit un homéomorphisme
\begin{align*}
 &\Mg (\Sigma, p) \simeq \\ & \left\lbrace   (\theta , A_1 , B_1 , \cdots, A_h , B_h ) \in \mathfrak{su(2)} \times  SU(2)^{2h} ~|~ e^{2\pi \theta}  =  \prod_{i=1}^{h}{[A_i , B_i ]} \right\rbrace . \end{align*}
L'élément $\theta$ est tel que la connexion vaut $\theta ds$ au voisinage du bord, $A_i$ (resp. $B_i$ ) est l'holonomie de $A$ le long de la courbe $\alpha_i$ (resp. $\beta_i$ ).
\end{prop}

On note \[ \N(\Sigma,p) = \left\lbrace (\theta , A_1 , B_1 , \cdots A_h , B_h ) \in \Mg (\Sigma, p)  ~|~ |\theta | < \pi \sqrt{2} \right\rbrace  \]

Cette partie s'identifie à l'ouvert de $SU(2)^{2h}$ correspondant aux éléments $( A_1 , B_1 , \cdots , A_h , B_h ) $ tels que $\prod_{i=1}^{h}{[A_i , B_i ]} \neq -I $, et est donc lisse.

Rappelons que cet espace est muni de la 2-forme de Huebschmann-Jeffrey, définie de manière similaire à la forme d'Atiyah-Bott pour une surface sans bord : si $A$ est une connexion représentant un point lisse de $\Mg (\Sigma, p)$, l'espace tangent s'identifie au quotient :

\[ T_{[A]} \Mg (\Sigma, p) = \frac{ \left\lbrace  \alpha\in \Omega^1(\Sigma) \otimes \mathfrak{su(2)}~|~ \alpha_{\nu \partial \Sigma} = \eta ds,~d_A\alpha = 0 \right\rbrace }{ \left\lbrace d_A f ~|~ f\in \Omega^0(\Sigma) \otimes \mathfrak{su(2)},~  f_{\nu \partial \Sigma} = 0 \right\rbrace }. \]
Si $\alpha = \eta\otimes a$ et $\beta = \mu\otimes b$, avec $\eta,\mu \in \mathfrak{su(2)}$ et $a,b$ des 1-formes à valeurs réelles, on note  $<\alpha \wedge \beta>$ la 2-forme à valeurs réelles définie par \[<\alpha \wedge \beta> = <\eta, \mu>  a\wedge b.\]
La forme de Huebschmann-Jeffrey $\omega$ est alors définie par :
\[ \omega_{[A]}([\alpha],[\beta]) = \int_{\Sigma'} \langle \alpha\wedge\beta \rangle   .\]
Cette forme est symplectique sur l'ouvert $\N(\Sigma,p)$, voir \cite[Prop. 3.1]{jeffrey}.

\section{Compactification par découpage symplectique}

Dans ce paragraphe nous rappelons brièvement comment Manolescu et Woodward obtiennent un objet de $\Symp$ à partir de l'espace $\N(\Sigma,p)$ : il s'agit d'une compactification $\Nc(\Sigma,p)$ obtenue  par découpage symplectique. Nous renvoyons à \cite[Parag. 4.5]{MW} pour plus de détails.

L'application $[A] \mapsto \theta \in \mathfrak{su(2)}$ est le moment d'une action $SU(2)$-Hamil\-tonienne, ainsi $[A] \mapsto |\theta| \in \rr$ est le moment d'une action du cercle (sur le complémentaire de $ \lbrace \theta = 0 \rbrace $). Il est alors possible de considérer le découpage symplectique de Lerman  en une valeur $\lambda \in \rr$ : c'est la réduction symplectique 
\[ \Mg(\Sigma,p)_{\leq \lambda} = \left( \Mg(\Sigma,p) \times \cc \right)  \red U(1)  \]
de l'action du cercle ayant pour moment $\Phi([A],z) = | \theta| +\frac{1}{2} |z|^2 -\lambda$.

\begin{remark}L'action du cercle n'étant pas définie sur $\lbrace \theta = 0 \rbrace$, on considère en réalité la réduction $\left( ( \Mg(\Sigma,p) \setminus  \lbrace \theta = 0 \rbrace) \times \cc \right)  \red U(1)$, qui contient $\N(\Sigma,p) \setminus  \lbrace \theta = 0 \rbrace$ plongée symplectiquement, à laquelle on recolle ensuite $\lbrace \theta = 0 \rbrace$.
\end{remark}

Pour $\lambda = \pi \sqrt{2}$, 

\[ \Mg(\Sigma,p)_{\leq \pi \sqrt{2}} = \N(\Sigma,p) \cup R, \]

avec $R =  \lbrace |\theta| = \pi \sqrt{2} \rbrace / U(1) $.

C'est cet espace que l'on note $\Nc(\Sigma,p)$, et $\tilde{\omega}$ est la 2-forme induite de la réduction. Elle est monotone (\cite[Proposition 4.10]{MW}), mais dégénérée sur $R$ (\cite[Lemma 4.11]{MW}).

En revanche, si $ \lambda = \pi \sqrt{2} - \epsilon$ pour $\epsilon$ petit,  $\Mg(\Sigma,p)_{\leq \lambda }$ est encore difféomorphe à $\Nc(\Sigma,p)$. Soit  $\varphi_\epsilon\colon \Nc(\Sigma,p) \rightarrow \Mg(\Sigma,p)_{\leq \pi \sqrt{2} - \epsilon }$ un difféomorphisme à support contenu au voisinage de $R$, et $\omega_\epsilon$  la forme symplectique de $\Mg(\Sigma,p)_{\leq \pi \sqrt{2} - \epsilon }$, alors $\omega = \varphi_\epsilon ^* \omega_\epsilon$ est une forme symplectique mais non-monotone sur $\Nc(\Sigma,p)$.

Enfin, $\tilde{J}$ est une structure presque complexe "de référence" sur $\Nc(\Sigma,p)$, compatible avec $\omega$ et telle que $R$ soit une hypersurface complexe.

Rappelons également le résultat suivant concernant la structure du lieu de dégénérescence $R$ de $\tilde{\omega}$, qui sera utile pour contrôler les phénomènes de bubbling:

\begin{prop} (\cite[Prop. 3.7]{MW})\label{degener} L'hypersurface $R$ admet une fibration en sphères telle que le noyau de $\tilde{\omega}$ correspond à l'espace tangent des fibres. De plus, le nombre d'intersection d'une fibre avec $R$ dans $\Nc(\Sigma,p)$ est -2.

\end{prop} 

Il s'en suit en particulier que les courbes pseudo-holomorphe d'aire nulle pour $\tilde{\omega}$ seront des revêtements ramifiés de fibres de cette fibration.

Ainsi, pour résumer les propriétés :

\begin{prop}[\cite{MW}] Le 5-uplet ($\Nc(\Sigma,p)$, $\omega$, $\tilde{\omega}$,$R$,$\tilde{J}$) satisfait les hypothèses de la définition \ref{defcat} : c'est un objet de $\Symp$.
\end{prop}

\section{Correspondances Lagrangiennes} 

Soit $(W,\pi_{\Sigma_0},\pi_{\Sigma_1},p,c)$ comme dans la définition \ref{chemincob}, et $C\subset int(W)$ une sous-variété sans bord de dimension 1, dont la classe d'homologie dans $H_1(W,\Z{2})$ vaut $c$. Il découlera de la proposition \ref{reformul} que seule la classe $c$ interviendra (indépendamment du choix de $C$).

On commence par définir une correspondance $L(\pi_{\Sigma_0},\pi_{\Sigma_1},p, C)$ entre $\Mg (\Sigma_0, p_0)$ et  $\Mg (\Sigma_1, p_1)$, puis lorsque le cobordisme est élémentaire, on en déduit une correspondance Lagrangienne  entre les découpages symplectiques, qui sera notée $L^c(\pi_{\Sigma_0},\pi_{\Sigma_1},p, C)$, et qui vérifiera les hypothèses de la catégorie $\Symp$.

\begin{defi}
\begin{itemize}

\item[$(i)$] (Espace des modules associé à un cobordisme à bord vertical $(W, \pi_{\Sigma_0},\pi_{\Sigma_1},p, C)$). Soit l'espace des connexions suivant, où $s$ désigne la coordonnée en $\rr/\zz$ de $p(\rr/\zz \times [0,1])$, et  $\mu$ désigne un méridien arbitraire de $C$ :

\begin{align*}&\mathscr{A}_F^\mathfrak{g}(W,C) =  \\ 
& \left\lbrace A \in \Omega^{1} (W \setminus C) \otimes \mathfrak{su(2)}\ |\ F_A = 0, \mathrm{Hol}_{\mu} A = -I, A_{|\nu p(\rr/\zz \times [0,1])} = \theta ds\right\rbrace.
\end{align*}

Sur cet espace agit le groupe de jauge suivant :

\[\Gc (W\setminus C)  = \left\lbrace u \colon W \setminus C \rightarrow SU(2)\ |\ u_{|\nu p(\rr/\zz \times [0,1])} = I\right\rbrace,\]

On définit alors le quotient \[ \Mg (W, \pi_{\Sigma_0},\pi_{\Sigma_1},p, C) =  \mathscr{A}_F^\mathfrak{g}(W,C) /\Gc (W\setminus C ).\]

\item[$(ii)$] (Correspondance associée à un cobordisme à bord vertical.) On définit une correspondance 

\[ L(\pi_{\Sigma_0},\pi_{\Sigma_1},p, C) \subset \Mg (\Sigma_0, p_0)^- \times \Mg (\Sigma_1, p_1) \]
comme les couples de connexions se prolongeant de manière plate à $W\setminus C$, avec holonomie $-I$ autour de $C$ :

\[L(\pi_{\Sigma_0},\pi_{\Sigma_1},p, C) = \{([A_{|\Sigma_0} ], [A_{|\Sigma_1} ])\ |\ A \in \Mg(W, C, *)\}\]

\end{itemize}
\end{defi}

\begin{remark}Cette définition dépend du paramétrage $p$ de la partie verticale de $\partial W$, cette dépendance sera décrite dans l'exemple \ref{exemreparam}. 

\end{remark}

\subsection{Reformulation en termes de fibrés SO(3) non-triviaux}

Dans cette sous-section nous présentons un point de vue différent concernant la construction des correspondances Lagrangiennes. Il en résultera que ces Lagrangiennes ne dépendent de la sous-variété $C$ que via sa classe $c$ dans $H_1(W; \Z{2})$. Ce résultat pourrait être démontré directement en observant que l'espace des modules ne change pas lorsque l'on supprime un croisement et lorsque l'on ajoute deux entrelacs parallèles, mais le point de vue que nous allons donner a une importance conceptuelle, notamment en lien avec le triangle initial de Floer pour l'homologie des instantons.

Soit $P$ le $SO(3)$-fibré principal au-dessus de $W$ défini en recollant les fibrés triviaux au-dessus de $ W\setminus \nu C$ et $ \nu C$ le long du bord $\partial \nu C$ par une fonction de transition  $f\colon \partial \nu C \rightarrow SO(3)$ telle que l'image de tout méridien de $C$ définit l'élément non-trivial de $\pi_1(SO(3))$ : 

\[ P =  SO(3) \times (W\setminus \nu C) \cup_f SO(3) \times  \nu C .\]

On notera $\tau\colon P_{| W\setminus \nu C}\rightarrow ( W\setminus \nu C ) \times SO(3)$ la trivialisation de $P$ sur $W\setminus \nu C$.

\begin{lemma} La seconde classe de Stiefel-Whitney $w_2(P)\in H^2(W; \Z{2})$ est Poincaré-duale à l'image de $c$ dans $H_1(W, \partial W; \Z{2})$.
\end{lemma}

\begin{proof}

Commençons par rappeler la construction de $w_2(P)$ en homologie de \v{C}ech : fixons $\lbrace U_i\rbrace_i$ un recouvrement acyclique de $W$ compatible avec $C$ au sens suivant : au voisinage de $C$, le recouvrement est modelé par 4 ouverts $V_0, V_1, V_2, V_3$ comme représentés dans la figure \ref{recouvrement}. Chaque composante connexe de $\nu C$ est recouverte par 3 (ou plus) ouverts du type $V_0$, et $\partial \nu C$ est recouvert par les $V_1, V_2, V_3$.

\begin{figure}[!h]
    \centering
    \def\svgwidth{.65\textwidth}
    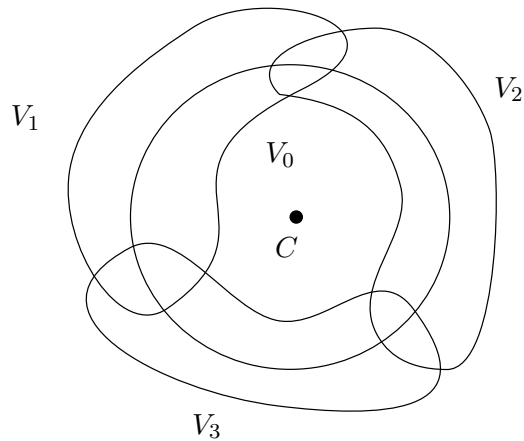
      \caption{Vue en coupe du recouvrement au voisinage de $C$.}
      \label{recouvrement}
\end{figure}

Le fibré $P$  est donné par des fonctions de transition $\alpha_{ij}\colon U_i \cap U_j \rightarrow SO(3)$. Ces fonctions vérifient $\alpha_{ij} \alpha_{ji} = I$ et $\alpha_{ij} \alpha_{jk} \alpha_{ki} = I$. Soient \[\widetilde{\alpha}_{ij}\colon U_i \cap U_j \rightarrow SU(2)\] des relèvements de $\alpha_{ij}$ à $SU(2)$, la seconde relation précédente devient \[\widetilde{\alpha}_{ij} \widetilde{\alpha}_{jk} \widetilde{\alpha}_{ki} \in \Z{2}\] et permet de définir un 2-cocycle en homologie de \v{C}ech :
\[ (c_{ijk} \colon U_i \cap U_j \cap U_k \rightarrow \Z{2})_{ijk} \in \check{\mathrm{C}}^2(W,\Z{2}), \]
et sa classe dans $\check{\mathrm{H}}^2(W,\Z{2})$ est alors $w_2(P)$, par définition.

D'après la construction du fibré $P$, les fonctions de transition peuvent être choisies comme étant, en notant encore $f$ un prolongement de $f$ à un voisinage de $\partial \nu C$,  $\alpha_{ij}(x) = f(x) $ si $U_i = V_0$ et $U_j \in \lbrace V_1, V_2, V_3 \rbrace$, et $\alpha_{ij} = I$ sinon. 

Par hypothèse, la fonction de transition $f$ ne se relève pas en une fonction de $\partial \nu  C$ vers $SU(2)$. En revanche il est possible de choisir des relèvements $\widetilde{f}_j \colon V_j\rightarrow SU(2)$ pour chaque $j = 1,2, 3$. On peut supposer que 
$\widetilde{f}_1 = \widetilde{f}_2 $ sur $V_1 \cap V_2 $, $\widetilde{f}_2 = \widetilde{f}_3 $ sur $V_2 \cap V_3 $, et $\widetilde{f}_3 = - \widetilde{f}_1 $ sur $V_3 \cap V_1 $. On pose alors  $\widetilde{\alpha}_{ij}(x) = \widetilde{f}_j(x) $ si $U_i = V_0$ et $U_j \in \lbrace V_1, V_2, V_3 \rbrace$, et $\alpha_{ij} = I$ sinon. Le cocycle $(c_{ijk}) _{ijk}$ prend alors les valeurs suivantes :

\[ c_{ijk} = \begin{cases} -I\text{ si } \lbrace U_i, U_j, U_k\rbrace = \lbrace V_0, V_1, V_3 \rbrace\\ I\text{ sinon,} \end{cases}  \]
où $V_0$, $V_1$ et $V_3$ désignent des ouverts du type précédent.

Soit à présent un cycle \[F = \sum_{ \lbrace i,j,k\rbrace\in I_F}{U_{i}\cap U_j \cap U_k } \in \check{\mathrm{C}}_2(W,\Z{2}), \] alors \[\langle w_2(P), [F]\rangle = \sum_{\lbrace i,j,k\rbrace\in I_F}{c_{ijk}} = [C] . [F].\]

\end{proof}

Soit $\A(W,P)$ l'espace des connexions plates sur $P$ de la forme $\theta ds$ au voisinage de $\partial ^{vert} W$, où l'on a  identifié les connexions à des 1-formes $\mathfrak{su(2)}$-valuées via la trivialisation $\tau$, et $s$ désigne le paramètre circulaire de $\partial ^{vert} W$. Cet espace admet une action du groupe $\G^0(W,P)$ des transformations de jauge triviales au voisinage de $\partial ^{vert} W$ et homotopes à l'identité (c'est-à-dire la composante connexe de l'identité du groupe des transformations de jauge triviales au voisinage de $\partial ^{vert} W$). Posons \[\Mg(W,P) =  \A(W,P)/\G^0(W,P)\] l'espace des orbites pour cette action. La trivialisation $\tau$ permet de définir une application \[  \Mg(W,P) \rightarrow  \N(\Sigma_0) \times \N(\Sigma_1)\] par restriction aux bords et tiré en arrière à $SU(2)\times ( \Sigma_0 \sqcup \Sigma_1 )$. On note $L(W,P) \subset \N(\Sigma_0) \times \N(\Sigma_1)$ son image.

\begin{remark} L'espace des modules $\Mg(W,P)$ ne dépend que du type d'isomorphisme de $P$, c'est-à-dire de la classe $c$, et la correspondance $L(W,P)$ ne dépend que de la restriction de $\tau$ à $\partial W$. Il s'en suit donc que $L(W,P)$ ne dépend de $C$ que via $c$.
\end{remark}

\begin{prop} \label{reformul}

L'espace des modules $\Mg(W,C)$ s'identifie canoniquement à $\Mg(W,P)$. Il s'en suit que $L(W,P) = L(W,C)$, ainsi d'après la remarque précédente $L(W,C)$ ne dépend que de la classe $c$.

\end{prop}

\begin{proof}

 Pour prouver cela nous allons construire deux applications qui sont inverses l'une de l'autre  :
\[ \begin{cases}
\Phi_1\colon \Mg(W,C) \rightarrow \Mg(W,P) \\ 
\Phi_2\colon \Mg(W,P) \rightarrow \Mg(W,C). 
\end{cases} \]

\textbf{1. L'application $\Phi_1$.} Soit $[A] \in \Mg(W,C)$ et $A \in [A]$ un représentant. La connexion $A$ induit une connexion plate $\widehat{A}$ sur $SO(3) \times (W\setminus\nu C)$ par passage au quotient, et si $\mu$ est un méridien de $C$, $\mathrm{Hol}_\mu \widehat{A} = I$. Ce fait, et la donnée d'un paramétrage $p\colon C\times \rr/\zz \rightarrow \partial \nu C$ permettent de définir une fonction de transition $f\colon \partial \nu C\rightarrow SO(3)$ par \[f(c,s) = \mathrm{Hol}_{\lbrace c\rbrace\times [0,s]}\widehat{A},\] avec $[0,s]\subset \rr/\zz$ un arc orienté quelconque allant de 0 à $s$.

Cette fonction de transition permet de recoller le fibré plat $(SO(3) \times (W\setminus\nu C), \widehat{A})$ avec le fibré horizontal $(SO(3) \times \nu C, A_{horiz})$. Notons $(Q, A_Q)$ le fibré plat ainsi obtenu. La fonction $f$ vérifie les mêmes hypothèses que celle choisie pour définir le fibré $P$,  les fibrés $Q$ et $P$ sont donc isomorphes. Soit $\varphi\colon Q\rightarrow P$ un isomorphisme, tel que $\tau \circ\varphi$ est l'identité sur $SO(3) \times (W\setminus\nu C)$.

Posons finalement $\Phi_1([A]) = [\varphi_* A_Q]\in \Mg(W,P)$. Cette classe est indépendante des choix que l'on a fait, modulo un élément de $\G^0(W,P)$.

\textbf{2. L'application $\Phi_2$.} Soit $[A] \in \Mg(W,P)$, et $A\in [A]$ un représentant. Le poussé en avant $\tau_*  A_{|W\setminus \nu C}$ définit une connexion sur $SO(3)\times W\setminus \nu C$, notons $\widetilde{A}$ la connexion sur $SU(2)\times W\setminus \nu C$ tirée en arrière par l'application quotient. Cette connexion vérifie $\mathrm{Hol}_\mu \widetilde{A} = -I$ pour tout méridien $\mu$ de $C$, en effet dans la trivialisation au-dessus de $\nu C$, le lacet $\gamma\colon s\mapsto \mathrm{Hol}_{[0,s]} A$ est contractile dans $SO(3)$, car $\mu$ borde un disque. Il s'en suit que le lacet $\widetilde{\gamma} \colon s\mapsto \mathrm{Hol}_{[0,s]} A$ défini dans la trivialisation au-dessus de $W\setminus \nu C$ ne l'est pas, car $\gamma$ et $\widetilde{\gamma}$ diffèrent par la fonction de transition $f$. La connexion $\widetilde{A}$ définit donc un élément $\Phi_2([A])$ de  $ \Mg(W,C)$, indépendant des choix modulo l'action de $\Gc(W,C)$.

Ces deux applications sont inverses l'une de l'autre par construction, et identifient ainsi $\Mg(W,C)$ et $\Mg(W,P)$.

\end{proof}

\subsection{Correspondance associée à un cobordisme élémentaire}
L'espace des modules $\Mg (W, \pi_{\Sigma_0},\pi_{\Sigma_1},p, c)$ associé à un cobordisme vertical quelconque peut ne pas être lisse, et l'application induite par l'inclusion $\Mg (W, \pi_{\Sigma_0},\pi_{\Sigma_1},p, c) \rightarrow \Mg (\Sigma_0)  \times \Mg (\Sigma_1) $ peut ne pas être un plongement, ainsi la correspondance $L(W, \pi_{\Sigma_0} , \pi_{\Sigma_1} ,p ,  c)$ peut ne pas être une sous variété Lagrangienne. Nous allons voir que ces problèmes n'apparaissent pas pour des cobordismes élémentaires. Nous décrivons à présent les correspondances associés à de tels cobordismes, puis nous prouvons qu'elles sont Lagrangiennes dans la proposition \ref{correspcobelem}.

\begin{exam}[Cobordisme trivial]\label{exemcobtriv}
Soit $(\Sigma,p)$ une surface à bord paramétré, $W = \Sigma \times [0,1]$, muni des plongements $\pi_i(x) = (x,i)$ et $p(s,t) = (p(s) , t)$.

Si $C = \emptyset$ : $L(W,C)$ est la diagonale $\Delta_{\Mg(\Sigma,x,*)}$.

Si $C \neq \emptyset$, $L(\Sigma \times [0,1] , C)$ est le graphe du difféomorphisme ayant pour expression au niveau des  holonomies : 
\[\begin{cases}
A_i \mapsto (-1)^{\alpha_i . C'} A_i\\
B_i \mapsto (-1)^{\beta_i . C'} B_i,
\end{cases}\]
 où $\pi_1 (\Sigma , p(0)) = \langle\alpha_1 , \cdots , \beta_h\rangle$, $C'$ est la projection de $C$ sur $\Sigma$ et $\alpha_i . C'$, $\beta_i . C'$ désignent les nombres d'intersection dans $\Sigma$ modulo 2.

En particulier, si $a_i = [\alpha_i]\in H_1(\Sigma \times [0,1], \Z{2})$ et $b_i = [\beta_i]$,  $L(\Sigma \times [0,1] , a_i, *)$ correspond au difféomorphisme qui envoie $B_i$ sur $-B_i$ et ne change pas les autres holonomies, et $L(\Sigma \times [0,1] , b_i, *)$ correspond au difféomorphisme qui envoie $A_i$ sur $-A_i$ et ne change pas les autres holonomies.

\end{exam}

\begin{proof} Choisissons l'entrelacs $C$ comme une courbe simple contenue dans la surface $\Sigma\times \lbrace \frac{1}{2}\rbrace$, de sorte que le complémentaire $W\setminus C$ se rétracte sur la réunion de $\Sigma\times \lbrace 0\rbrace$ et d'un tore entourant $C$ et touchant $\Sigma\times \lbrace 0\rbrace$ en $C\times \lbrace 0\rbrace$. D'après le théorème de Seifert-Van Kampen, $\pi_1(W\setminus C, *) \simeq (\zz \lambda \oplus \zz \mu ) * F_{2g-1}$, où $\lambda$  et $\mu$ désignent une longitude et un méridien de $C$. Ainsi, les représentations de $\pi_1(W\setminus C, *) $ envoyant $\mu$ sur $-I$ sont en bijection avec les représentations de $(\zz \lambda  ) * F_{2g-1} \simeq \pi_1(\Sigma, *)$, car $-I$ est dans le centre de $SU(2)$. Il s'en suit que $\Mg(W,C) \simeq \Mg(\Sigma \times \lbrace 0\rbrace,p)$. 

Par ailleurs, étudions la restriction $\Mg(W,C) \rightarrow \Mg(\Sigma \times \lbrace 1\rbrace,p)$. Si $\gamma$ est un lacet basé dans $\Sigma $, le carré $\gamma\times [0,1]$ rencontre $C$ $\gamma \cdot C'$ fois,  l'holonomie d'une connexion $A$ autour du bord vaut donc $(-1)^{\gamma \cdot C'}$. D'autre part elle vaut $\mathrm{Hol}_{\gamma \times \lbrace 1\rbrace} A (\mathrm{Hol}_{\gamma \times \lbrace 0\rbrace}A)^{-1}$.

\end{proof}

\begin{exam}[Reparamétrage du cylindre vertical]\label{exemreparam}

Supposons $W$ et les plongements $\pi_0$, $\pi_1$ comme dans l'exemple précédent, mais $p(s,t) = (p(s) +\psi (t) , t)$, pour une fonction $\psi\colon [0,1] \rightarrow \rr$. Alors $L(W, \pi_{\Sigma_0} , \pi_{\Sigma_1} , C, p)$ est le graphe du difféomorphisme suivant :

\[ (\theta, A_1, B_1, \cdots) \mapsto (\theta, Ad_{e^{\alpha \theta}}A_1, Ad_{e^{\alpha \theta}}B_1, \cdots),\]
avec $\alpha = \psi(1) - \psi(0)$. (Cela correspond à faire une rotation d'angle $\alpha$ le long du bord de $\Sigma$. ) 

\end{exam}

\begin{proof} On applique le même raisonnement que pour l'exemple précédent, mais cette fois l'holonomie le long des deux autres bords du carré vaut $e^{\alpha \theta}$ et $e^{-\alpha \theta}$.

\end{proof}

\begin{exam}[Difféomorphisme d'une surface]\label{exemdiffeo}

Soit $\varphi$ un difféomorphisme de $(\Sigma,p)$ valant l'identité sur le bord, $W = \Sigma \times  [0, 1]$, $\pi_0 =  id_\Sigma \times  \lbrace 0\rbrace$, $\pi_1 = \varphi  \times  \lbrace 1\rbrace$, et $p'(s,t) = (p(s),t)$. Si $\pi_1 (\Sigma' , *) = \langle\alpha_1 , \cdots , \beta_h\rangle$ est le groupe libre à $2h$ générateurs, soit  $ u_i (\alpha_1 , \cdots , \beta_h )$ le mot en $\alpha_1 , \cdots , \beta_h$ correspondant à $\varphi_* \alpha_i$, et $v_i (\alpha_1 , \cdots , \beta_h )$ le mot correspondant à $\varphi_* \beta_i$. Alors  $L(W, \pi_{\Sigma_0} , \pi_{\Sigma_1} ,  p', 0)$ est le graphe du difféomorphisme :

\[ (\theta, A_1, B_1, \cdots) \mapsto (\theta, u_1(A_1, B_1, \cdots), v_1(A_1, B_1, \cdots), \cdots).\]

En particulier, le twist de Dehn autour d'une courbe librement homotope à $\beta_1$ est le graphe du difféomorphisme :

\[ (\theta, A_1, B_1, \cdots) \mapsto (\theta, A_1 B_1,  B_1,  \cdots).\]

\end{exam}

\begin{proof}Découle de l'exemple \ref{exemcobtriv} et de la formule donnant l'holonomie le long d'un produit de lacets.

\end{proof}

L'exemple suivant illustre la nécessité de considérer la catégorie $\Cob$, et non la catégorie des cobordismes (sans bord verticaux). En effet si l'on referme le cobordisme suivant en collant un tube le long du bord vertical, on obtient un cobordisme trivial du tore vers le tore, identique à celui que l'on aurait obtenu à partir du cobordisme de l'exemple \ref{exemcobtriv}, mais les correspondances Lagrangiennes obtenues ne sont pas les mêmes.

\begin{exam}[Un changement de "chemin base"]\label{exemchgtchemin}
Soit $\Sigma$ le 2-tore privé d'un petit disque $D$, $*$ un point base sur le bord,  $\alpha_1$ et $\beta_1$ des courbes simples formant une base de son groupe fondamental, et $W = (T^2\times[0,1]) \setminus S$, où $S$ est un voisinage tubulaire du chemin $(\alpha_1(t),t)$, (on paramètre le bord vertical sans tourner) voir figure \ref{exemchgtcheminfig}. $L(W,p, 0)$ est le graphe de :

\[ (\theta, A_1, B_1) \mapsto (\theta, A_1 , A_1^{-1} B_1 A_1).\]

\end{exam}

\begin{figure}[!h]
    \centering
    \def\svgwidth{.65\textwidth}
    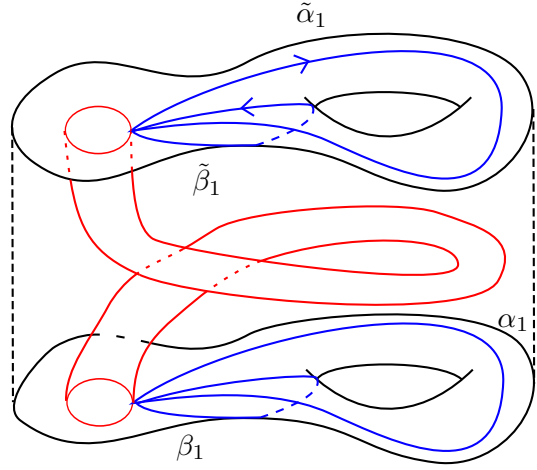
      \caption{Un changement de "chemin base".}
      \label{exemchgtcheminfig}
\end{figure}

\begin{proof} Identifions $\alpha_1$ et $\beta_1$  aux courbes correspondantes dans $\Sigma \times\lbrace 0\rbrace$, et notons $\tilde{\alpha}_1$ et $\tilde{\beta}_1$ les courbes correspondantes dans $\Sigma \times\lbrace 1\rbrace$. On note $*$ et $\tilde{*}$ les points bases correspondants et $\gamma$ l'arc vertical allant de $*$ vers $\tilde{*}$. L'assertion vient du fait que  $\tilde{\alpha}_1$ (resp.$\tilde{\beta}_1$ ) est homotope à $\gamma \alpha_1 \gamma^{-1}$ (resp. $\gamma \alpha_1^{-1} \beta_1 \alpha_1 \gamma^{-1}$).

\end{proof}
%


\begin{exam}[Ajout d'une 2-anse]\label{exemanse}

Soit $s\subset int(\Sigma)$ une courbe simple librement homotope à $\beta_{1}$, et $W\colon \Sigma \rightarrow S$ le cobordisme correspondant à l'attachement d'une 2-anse le long de $s$,  alors $\pi_1(S) = \langle \alpha_2, \beta_2, \cdots \rangle$, et

\[ L(W, p, 0) = \lbrace (\theta , A_1 , I , A_2, B_2,  \cdots A_h , B_h ), (\theta , A_2 , B_2 , \cdots A_h , B_h )  \rbrace, \]
où $A_1 \in SU(2)$ et $(\theta , A_2 , B_2 , \cdots A_h , B_h ) \in \Mg(S)$.

\end{exam}

\begin{proof} Le cobordisme $W$ se rétracte sur le bouquet de $S$ et du cercle $\alpha_1$ correspondant à la co-âme de l'anse. Il s'en suit que $\Mg(W) \simeq \Mg(S) \times SU(2)$.

Par ailleurs, sous cette identification, l'application $\Mg(W) \rightarrow \Mg(S)$ est la projection sur le premier facteur, et $\Mg(W) \rightarrow \Mg(\Sigma)$ envoie un couple $(A_1, [A] )$ sur la connexion telle que $\mathrm{Hol}_{\beta_{1}} = I$, $\mathrm{Hol}_{\alpha_{1}} = A_1$, et  dont les autres holonomies sont identiques à celles de $A$.

\end{proof}

\begin{prop}\label{correspcobelem}S'il existe, comme dans la définition \ref{chemincob}, une fonction de Morse $f$ sur $W$, constante sur les bords $\Sigma_0$ et $\Sigma_1$, et avec au plus un point critique, $L(W, \pi_{\Sigma_0} , \pi_{\Sigma_1} , p, c)$ est une correspondance Lagrangienne en restriction à la partie symplectique des espaces de modules.
\end{prop}

\begin{proof}

Un cobordisme élémentaire correspond soit à un cobordisme trivial, soit à l'ajout d'une 1-anse ou d'une 2-anse, les deux derniers cas étant symétriques. Il est clair, d'après les exemples \ref{exemcobtriv} et \ref{exemanse}, que dans chaque cas, $L(W, \pi_\Sigma , \pi_S , C, p)$ est lisse et de dimension maximale.

Montrons l'isotropie pour la forme symplectique. Soit $[A]\in \Mg (W, C, p)$, choisissons un représentant $A\in [A]$ de la forme $\eta_0 ds$ au voisinage de $C$, où $s\in \rr/\zz$ est le paramètre d'un méridien, et $\eta_0\in \mathfrak{su(2)}$ un élément fixé tel que $\exp(\eta_0) = -I$.

Soient $\alpha$, $\beta$ des 1-formes $\mathfrak{su(2)}$-valuées représentant des vecteurs tangents de $T_{[A]}\Mg (W, C, p)$ : c'est-à-dire vérifiant $d_A\alpha = d_A\beta = 0$, et de la forme $\theta ds$ au voisinage de $\partial^{vert} W$. Toute connexion plate proche de $A$ pouvant, à transformation de jauge près, s'écrire sous la forme $\eta_0 ds$ au voisinage de $C$, on peut de plus supposer $\alpha$ et $\beta$ nulles au voisinage $C$ et en particulier prolongeables à $W$ en entier.

Si l'on note $\tilde{A} \in L(W, \pi_\Sigma , \pi_S , C, p)$ l'image de $[A]$ par le plongement

\[ \Mg (W, C, p) \rightarrow \Mg (\Sigma_0, p_0) \times \Mg (\Sigma_1, p_1), \]

 et $\tilde{\alpha},\tilde{\beta} \in T_{\tilde{A}}L(W, \pi_\Sigma , \pi_S , C, p)$ les vecteurs tangents correspondants,

\[ \omega_{\tilde{A}}(\tilde{\alpha},\tilde{\beta}  ) = \int_{\Sigma_1}{\langle \alpha \wedge \beta \rangle } -  \int_{\Sigma_0}{\langle \alpha \wedge \beta \rangle } \]
 La formule de Stokes donne alors : 

\begin{align*}
 0 &= \int_{W}{d \langle \alpha \wedge \beta \rangle }  \\ 
 &= \int_{\Sigma_1}{\langle \alpha \wedge \beta \rangle } -  \int_{\Sigma_0}{\langle \alpha \wedge \beta \rangle } +  \int_{\partial^{vert} W}{\langle \alpha \wedge \beta \rangle } \\
\end{align*}
Et le dernier terme est nul car sur la partie verticale, $\alpha$ et $\beta$ sont proportionelles à $ds$. 
\end{proof}

Il découle de la proposition \ref{correspcobelem} que les difféomorphismes des exemples précédents dont les correspondances sont les graphes sont des symplectomorphismes. Seul le dernier type de  correspondances (exemple \ref{exemanse}) ne provient pas d'un symplectomorphisme, mais d'une sous-variété coisotrope fibrée. L'énoncé suivant, que l'on peut trouver dans \cite[Exemple 6.3]{MW} pour le cas d'une "sous-variété coisotrope sphériquement fibrée", fournit un critère utile pour découper des correspondances Lagrangiennes. Il contient tous les exemples précédents, en effet pour un symplectomorphisme il suffit de considérer $C = M_0$ et  $\varphi = \pi$.

\begin{remark}En toute rigueur les 2-formes peuvent être dégénérées, mais "action Hamiltonienne" continue d'avoir un sens dès que l'équation " $\iota _{X_\xi} \omega =  d\langle H,\xi\rangle$" est encore satisfaite. L'énoncé que l'on va donner est toujours valide dans ce cas.
\end{remark}

\begin{prop}
\label{decoupcoiso}Soit $M_0$ une variété symplectique munie d'une action Hamiltonienne de $U(1)$ de moment $\varphi_0\colon M_0\rightarrow \rr$ ainsi que d'une sous-variété coisotrope $C\subset M_0$ qui admet une fibration $\pi\colon C\rightarrow M_1$ au-dessus d'une variété symplectique $M_1$ telle que l'image $L =(\iota \times \pi) (C) \subset M_0^- \times M_1$ soit une correspondance Lagrangienne.

Soit $\lambda\in \rr$ une valeur régulière de $\varphi_0$  telle que l'action de $U(1)$ sur $\varphi_0^{-1}(\lambda)$ soit libre. On peut alors former le découpage de Lerman  $M_{0,\leq \lambda} = M_{0,< \lambda}\cup R_0$.

On suppose de plus que $C$ est $U(1)$-équivariante, et intersecte $\varphi_0^{-1}(\lambda)$ transversalement. L'action de $U(1)$ passe alors au quotient en une action Hamiltonienne de moment $\varphi_1 \colon M_1 \rightarrow \rr$,  pour laquelle $\lambda$ est une valeur régulière. On note $M_{1,\leq \lambda} = M_{1,< \lambda}\cup R_1$ le découpage de Lerman.

Alors, l'adhérence $L^c$ de $L\cap  (M_{0,< \lambda}^- \times M_{1,< \lambda} )$ dans $M_{0,\leq \lambda}^- \times M_{1,\leq \lambda}$ définit une correspondance Lagrangienne $(R_0, R_1)$-compatible.

Si de plus $M_{0,\leq \lambda}$ et $M_{1,\leq \lambda}$ sont des objets de $\Symp$, et si $L\cap  (M_{0,< \lambda}^- \times M_{1,< \lambda} )$ est simplement connexe et spin, alors $L^c$ est un morphisme de  $\Symp$ : tout disque pseudo-holomorphe $(u_0,u_1)\colon (D^2,\partial D^2) \rightarrow (M_{0,\leq \lambda}^- \times M_{1,\leq \lambda}, L^c)$ d'aire nulle a un nombre d'intersection avec $(R_0, R_1)$ strictement plus petit que -2.

\end{prop}

\begin{proof}

Notons $\Phi_i \colon M_i \times \cc \rightarrow \rr$ les moments de l'action de $U(1)$, définis par \[\Phi_i(m,z) = \varphi_i (m) + \frac{1}{2}|z|^2 - \lambda,\] qui donneront  les découpages $M_{i,\leq \lambda} = \Phi_i^{-1}(0) / U(1)$. Posons également 
$Q_i = \varphi_i^{-1}(\lambda) $, de manière à avoir $R_i = Q_i /U(1)$. Enfin, posons \[\widetilde{L} = (L\times \cc^2 ) \cap ( \Phi_0^{-1}(0) \times \Phi_1^{-1}(0)) \subset M_0 \times \cc \times M_1 \times \cc .\]

On va montrer que $ \widetilde{L} / U(1)^2 \subset M_{0, \leq \lambda} \times M_{1, \leq \lambda}$ est une correspondance Lagrangienne lisse, et compatible avec les hypersurfaces. Cette correspondance contenant $L\cap  (M_{0,< \lambda}^- \times M_{1,< \lambda} )$ comme un ouvert dense, il s'en suivra que $L^c = \widetilde{L} / U(1)^2$.

D'une part, $\Phi_0^{-1}(0)$ et $ \Phi_1^{-1}(0)$ sont lisses car $\lambda$ est valeur régulière de $\varphi_0$ et $\varphi_1$. L'intersection $(L\times \cc^2 ) \cap ( \Phi_0^{-1}(0) \times \Phi_1^{-1}(0))$ est transverse dans $M_0 \times \cc \times M_1 \times \cc$, en effet $(\lbrace 0\rbrace \times \cc^2 ) \cap ( \Phi_0^{-1}(0) \times \Phi_1^{-1}(0)) = \lbrace 0\rbrace$.  Enfin, l'action de $U(1)^2$ sur $\tilde{L}$ est libre, car l'action de $U(1)$ est libre sur $\cc\setminus 0$ et sur $\varphi_0^{-1}(\lambda)$, par hypothèse. Il s'en suit que $\tilde{L}/U(1)^2$ est lisse.

Montrons maintenant la compatibilité avec les diviseurs. D'une part, si $(m_0,m_1) \in L$, $\varphi_0 (m_0) = \varphi_1 (m_1)$. Il s'en suit que 
\[L\cap (M_0 \times Q_1 ) = L\cap (Q_0 \times M_1) = L\cap (Q_0 \times Q_1),\] 
puis 
\[\widetilde{L}\cap (M_0 \times Q_1 \times \cc^2) = L\cap (Q_0 \times M_1\times \cc^2) = L\cap (Q_0 \times Q_1\times \cc^2),\] 
et enfin : 
\[L^c\cap (M_{0, \leq \lambda} \times R_1  ) = L\cap (R_0 \times M_{1, \leq \lambda}) = L\cap (R_0 \times R_1).\]

D'après le théorème des fonctions implicites, la propriété de graphe local est équivalente à  $\forall x \in L^c\cap (R_0 \times R_1 ),$ \[T_x L^c\cap T_x (M_0 \times R_1 ) =T_x  L^c\cap T_x (R_0 \times M_1) = T_x ( L\cap (R_0 \times R_1) ).\]
Si $x \in L\cap (Q_0 \times Q_1 )$ et $(v_0,v_1) \in T_x  L$, $v_1 = d\pi_{x_0}.v_0$. Il s'en suit :
\[ T_x L\cap T_x (M_0 \times Q_1 ) =T_x  L\cap T_x (Q_0 \times M_1) = T_x ( L\cap (Q_0 \times Q_1) ),\]
ce qui entraîne la propriété annoncée.

Enfin, concernant la propriété sur les disques, si $(u_0,u_1)$ est un tel disque, alors $\pi\circ u_0 \cup u_1\colon D^2\cup_{\partial D^2} D^2 \rightarrow M_{1, \leq \lambda}$ définit une sphère pseudo-holomorphe d'aire nulle, qui intersecte $R_0$ en un multiple strictement positif de -2.

\end{proof}

Notons que si $(W,p,c)$ est un cobordisme élémentaire,  \[ L(W,p,c) \cap \left( \N (\Sigma_0,p_0)^- \times \N (\Sigma_1,p_1) \right)\] s'identifie soit à $\N (\Sigma_i,p_i) \times SU(2)$, avec $i=0$ ou 1, soit à $\N (\Sigma_0,p_0)$. Dans les deux cas c'est un ouvert d'un produit de copies de $SU(2)$, ainsi sa seconde classe de Stiefel-Whitney s'annule.

Définissons à présent les correspondances Lagrangiennes entre les espaces découpés:

\begin{defi}
Si $(W,p,c)$ est un cobordisme à bord vertical élémentaire de $(\Sigma_0,p_0)$ vers $(\Sigma_1,p_1)$,
la correspondance $L(W,p,c)$ vérifie les hypothèses de la proposition \ref{decoupcoiso}. On définit ainsi $L^c(W,p,c)\subset \Nc(\Sigma_0,p_0)^- \times \Nc(\Sigma_1,p_1)$ comme l'adhérence de 
\[ L(W,p,c) \cap \left( \N (\Sigma_0,p_0)^- \times \N (\Sigma_1,p_1) \right), \]
qui est donc un morphisme de $\Symp$.
\end{defi}



\section{Invariance par mouvements de Cerf}

Le fait suivant est vrai pour tous cobordismes à bord verticaux, élémentaires ou non :

\begin{prop}[Formule de composition]\label{compocob} Soient $\Sigma,S,T$ trois surfaces à bord paramétré, et  $(W_1, c_1)$, $(W_2, c_2 )$ deux cobordismes à bord verticaux, allant respectivement de $\Sigma$ vers $S$, et de $S$ vers $T$. Alors,
\[L(W_1 \cup_S W_2, c_1 + c_2) = L(W_1, c_1) \circ  L(W_2, c_2 ).\]
\end{prop}

\begin{proof}L'inclusion de $L(W_1 \cup_S W_2, c_1 + c_2)$ dans la composée est évidente. L'inclusion réciproque vient du fait que, si $C_1$ et $C_2$ sont des sous-variétés représentant les classes $c_1$ et $c_2$,  deux connections plates sur $W_1\setminus C_1$ et $W_2\setminus C_2$ qui coïncident sur $S$ se recollent en une connexion plate sur $W_1\setminus C_1 \cup W_2\setminus C_2 $.

\end{proof} 

\begin{remark}Cette composition géométrique n'est pas plongée en général.
\end{remark}

\begin{theo}\label{hyptechniques} Le foncteur de $\Cobelem$ vers $\Symp$ suivant se factorise ainsi en un foncteur de $\Cob$ vers $\Symp$.
\[ \left\lbrace \begin{aligned}
(\Sigma,p) & \mapsto \Nc(\Sigma,p) \\
(W,f,p,c) & \mapsto L^c(W,f,p,c).
\end{aligned}
\right. \]
On notera $\underline{L}(W,p,c)$ l'image d'un cobordisme par ce foncteur.
\end{theo}

\begin{proof}Il suffit de vérifier que le foncteur satisfait aux hypothèses de la proposition \ref{factorisation}. Les hypothèses $(i)$ et $(ii)$ sont clairement vérifiées, et le point $(iii)$ découle de \cite[Lemma 6.11]{MW}. Il reste à vérifier les hypothèses $(iv)$ et $(v)$.

Vérifions l'hypothèse $(iv)$ : soit $(\Sigma_0,p_0)$  une surface à bord paramétré de genre $g\geq 2$, et $s_1,\ s_2$ deux cercles d'attachement disjoints, non-séparants de $\Sigma_0$. Soit \[\alpha_1, \cdots , \alpha_g, \beta_1, \cdots ,\beta_g\] un système de générateurs de $\pi_1(\Sigma_0, p_0(0))$ tel que $\partial \Sigma_0  $ est le produit des commutateurs des $\alpha_i$ et $\beta_i$, et tel que $s_i$ est librement homotope à $\alpha_i$ ($i=1,\ 2$).

Soit $W_1$ le cobordisme entre $\Sigma_0$ et $\Sigma_1$ correspondant à l'attachement d'une 2-anse le long de $s_1$ ($\Sigma_1$ est de genre $g-1$), $W_2$ le cobordisme entre $\Sigma_1$ et $\Sigma_2$ correspondant à l'attachement d'une 2-anse le long de $s_2$ ($\Sigma_2$ est de genre $g-2$). 

Soit $W_1'$ le cobordisme entre $\Sigma_0$ et $\widetilde{\Sigma_1}$ correspondant à l'attachement d'une 2-anse le long de $s_2$ ($\widetilde{\Sigma_1}$ est de genre $g-1$), $W_2'$ le cobordisme entre $\widetilde{\Sigma_1}$ et $\Sigma_2$ correspondant à l'attachement d'une 2-anse le long de $s_1$.

\[\xymatrix{ & \Sigma_1 \ar[rd]^{W_2} &   \\  \Sigma_0 \ar[ru]^{W_1} \ar[rd]_{W_1 '}    &  & \Sigma_2  \\  & \widetilde{\Sigma_1} \ar[ru]_{W_2'} &  }\]

Notons $C_1 = \lbrace A_1 = I \rbrace\subset \N(\Sigma_0 )$ la sous-variété coisotrope sphériquement fibrée correspondant à $W_1$, et $C_2 = \lbrace A_2 = I \rbrace\subset \N(\Sigma_0 )$ la sous-variété coisotrope sphériquement fibrée correspondant à $W_1'$.

Les sous-variétés $C_1$ et $C_2$ s'intersectent transversalement dans $\N(\Sigma_0 )$, les compositions $L(W_1) \circ L(W_2)$ et $L(W_1') \circ L(W_2')$ sont donc plongées, et elles coïncident car correspondent à la sous-variété coisotrope $C_1 \cap C_2 \subset \N(\Sigma_0 )$, fibrée au-dessus de $\N(\Sigma_2 )$, et simplement connexe car difféomorphe à $SU(2)^2\times \N(\Sigma_2 )$. Puis, d'après la proposition \ref{decoupcoiso}, son adhérence dans $\Nc(\Sigma_0 )$ définit un morphisme de $\Symp$.

Vérifions l'assertion sur les cylindres "quilted" en nous inspirant du raisonnement de la preuve de \cite[Lemma 6.11]{MW} pour la composée $L(W_1) \circ L(W_2)$  (l'assertion concernant $L(W_1') \circ L(W_2')$ se traite de manière analogue).

Montrons que tout cylindre matelassé intersecte le triplet $(R_0 , R_1 , R_2)$ en un multiple strictement négatif de $2$. Soit $ u = (u_0,u_1,u_2)$ un quilt pseudo-holomorphe d'indice de Maslov nul comme dans la figure \ref{cylindre}, tel que $u_i$ prend ses valeurs dans $\Nc(\Sigma_i)$ et les condition aux coutures sont données par $L^c(W_1)$ , $L^c(W_2)$ et  $L^c(W_1) \circ L^c(W_2)$.

 Par monotonie, l'aire des disques $u_i$ pour les formes monotones $\tilde{\omega_i}$ est nulle, ce qui force $u_i$ a être contenu dans une fibre du lieu de dégénérescence $R_i$. En effet, rappelons que $R_i$ admet une fibration en sphères $S^2$ dont le fibré vertical correspond exactement au lieu d'annulation de $\tilde{\omega_i}$, voir la proposition \ref{degener}. Par ailleurs la fibre contenant $u_0$ touche la sous-variété coisotrope $ C_1 \cap C_2$ et est donc incluse dans cette dernière, elle se projette donc sur une fibre de $R_2$. Il en est de même pour $u_1$ : elle est contenue dans $\lbrace A_2 = I \rbrace \subset \Nc(\Sigma_1 )$ et se projette sur une fibre de $R_2$.
 
Ainsi, $u_2$ et les images de $u_0$ et $u_1$ par les projections sur $\Nc(\Sigma_2 )$ se recollent en une sphère pseudo-holomorphe de $R_2$, et cette sphère intersecte $R_2$ en un multiple de -2, mais ce nombre d'intersection est exactement $u.(R_0 , R_1, R_2 )$.

Vérifions à présent l'hypothèse $(v)$. Notons que si $c_i + c_{i+1} = d_i + d_{i+1}$, alors d'après la proposition \ref{compocob}, $L(W_i,c_i) \circ L(W_{i+1},c_{i+1})$ et $L(W_i,d_i) \circ L(W_{i+1}, d_{i+1})$ coïncident avec  $L(W_i \cup W_{i+1}, c_i + c_{i+1}) $. Enfin, la correspondance associée à un cobordisme trivial $(\Sigma\times [0,1],c)$ est le graphe d'un symplectomorphisme, sa  composition à droite ou à gauche  avec toute autre
correspondance satisfait les hypothèses du théorème \ref{compogeom}. 

\end{proof}

\section{Homologie Instanton-Symplectique d'une vari\-été munie d'une classe d'homologie}\sectionmark{Homologie Instanton-Symplectique}\label{defHSI}

Soit $Y$ une 3-variété orientée, compacte, sans bord, $c\in H_1(Y; \Z{2})$ et $z\in Y$. Soit $W$ la variété à bord obtenue par éclatement réel orienté de $Y$ en $z$, $W = (Y\setminus z)\cup S^2$, et $p\colon \rr/\zz \times [0,1] \rightarrow S^2$ un plongement orienté. $(W,p,c)$ est donc un morphisme dans la catégorie $\Cob$ du disque vers le disque. L'ensemble des points d'intersections généralisés $\I(\underline{L}(W,p,c))$ est contenu dans  le produit des niveaux zéros des moments "$\theta_i = 0$", et donc dans $\mathrm{int}\left\lbrace \omega_i = \tilde{\omega}_i \right\rbrace$. On peut alors considérer leur homologie matelassée. Il résulte alors du théorème \ref{hyptechniques} :
\begin{cor}Le groupe abélien $HF(\underline{L}(W,p,c))$, à isomorphisme près, ne dépend que du type topologique de $Y$, du point $z$, et de la classe $c$. On le note $HSI(Y,c,z)$.
\end{cor}
\arnaque

Ainsi, $\bigcup_z HSI(Y,c,z)$ peut \^etre vu comme un fibré au-dessus de $Y$, (et en particulier au-dessus d'un scindement comme dans \cite[Parag. 5.3]{MW}). On notera parfois $HSI(Y,c)$ au lieu de $HSI(Y,c,z)$.

\begin{remark}\label{foncteurainfini}Le foncteur de $\Cob$ vers $\Symp$ que l'on a construit dans ce chapitre nous a permis de définir l'homologie HSI en appliquant l'homologie de Floer matelassée. Néanmoins, un tel foncteur contient potentiellement beaucoup plus d'information, et il est en principe possible d'extraire d'autre types d'invariants, prenant des formes algébriques plus sophistiquées. Par exemple, dans \cite{WWfunctoriality} \WW  associent à une correspondance Lagrangienne $L\subset M_0^- \times M_1$ un foncteur entre deux catégories  $Don^\#(M_0)$ et $Don^\#(M_1)$ appelées "catégories de Donaldson étendues". On peut espérer que leur construction fournisse des invariants pour des 3-variétés à bords munies de classes d'isotopies de chemins reliant les bords, similaires aux invariants apparaissant dans les travaux récents de Fukaya \cite{Fukayaboundary}. De tels invariants motiveraient la construction et l'étude des catégories correspondantes pour les espaces des modules étendus $\Nc(\Sigma)$. Des versions $A_\infty$ entre des catégories de Fukaya dérivées devraient également exister.
\end{remark}

\chapter{Premières propriétés}\label{chappremieresprops}
\section{Calcul à partir d'un scindement de Heegaard}

Soit $Y = H_0 \cup_{\Sigma} H_1$ un scindement de Heegaard donné de $Y$, de genre $g$,  $z\in \Sigma$ un point, et $c \in H_1(Y;\Z{2})$ une classe d'homologie, que l'on peut décomposer en la somme de deux classes $c = c_0 + c_1$, avec $c_0 \in H_1(H_0;\Z{2})$ et $c_1 \in H_1(H_1;\Z{2})$.

\begin{remark}les applications $H_1(H_i;\Z{2}) \rightarrow H_1(Y;\Z{2})$ induites par les inclusions étant surjectives, on peut toujours supposer que $c_0 = 0$ ou $c_1 = 0$.
\end{remark}

Notons $W$, $\Sigma'$, $ H_0'$ et $H_1'$ les éclatements respecifs de $Y$, $\Sigma$, $ H_0$ et $H_1$ au point $z$, de manière à avoir un scindement éclaté $ W = H_0' \cup_{\Sigma'} H_1'$.

On se donne un paramétrage $p\colon \rr/\zz \times [0,1] \rightarrow \partial W$ tel que $p(\rr/\zz \times \frac{1}{2}) = \partial \Sigma'$, on note $p_0$ (resp. $p_1$) la restriction de $p$ à $\rr/\zz \times [0,\frac{1}{2}]$ (resp. $\rr/\zz \times [\frac{1}{2},1]$). Ainsi, dans la catégorie $\Cob$, $(H_0, p_0, c_0)\in Hom(D^2, \Sigma')$, et $(H_1, p_1, c_1)\in Hom(\Sigma', D^2)$. Soient $f_0, f_1$ des fonctions de Morse sur $H_0$ et $H_1$ respectivement, adaptées aux paramétrages $p_0$ et $p_1$ (de sorte qu'ils soient verticaux), et ayant exactement $g$ points critiques chacune (d'indices 1 pour $f_0$ et d'indices 2 pour $f_1$). Elles décomposent ainsi $H_0$ et $H_1$ en $g$ cobordismes élémentaires : $H_0 = H_0^1 \odot H_0^2 \odot \cdots H_0^g$, $H_1 = H_1^1 \odot H_1^2 \odot \cdots H_1^g$.

\begin{lemma}
Pour tout $i$ entre  $2$ et $g$, la composition $L(H_0^1 \cup  \cdots \cup H_0^{i-1}  ) \circ L(H_0^i)$ est plongée, vérifie les hypothèses du théorème \ref{compogeom}, et vaut $L(H_0^1 \cup  \cdots \cup H_0^{i}  )$. 
\end{lemma}

\begin{proof}Soit $\alpha_1,\cdots \alpha_i, \beta_1,\cdots \beta_i$ un système de générateur du groupe fondamental de la composante de bord de genre $i$ de $H_0^i$ tel que $H_0^i$ correspond à l'attachement d'une 2-anse le long de $\beta_i$, et tel que les courbes $\alpha_1,\cdots \alpha_{i-1}, \beta_1,\cdots \beta_{i-1}$ induisent un système de générateurs du bord de genre $i-1$. Sous les identifications holonomiques suivantes des espaces de modules

\begin{align*}\N(\Sigma_0^{i-1}) &= \lbrace  (A_1,B_1, \cdots,  A_{i-1}, B_{i-1})~|~ [A_1,B_1] \cdots  [A_{i-1}, B_{i-1}] \neq -I \rbrace  \\
 \N(\Sigma_0^{i}) &= \lbrace  (A_1,B_1, \cdots,  A_{i}, B_{i})~|~ [A_1,B_1] \cdots  [A_{i}, B_{i}] \neq -I \rbrace ,
\end{align*}
les correspondances sont données par :
\begin{align*}
 & L(H_0^1 \cup  \cdots \cup H_0^{i-1}  ) =\lbrace (A_1, \epsilon_1 I, A_2, \epsilon_2 I, \cdots ) \rbrace, \text{où } \epsilon_i = \pm 1   \\
 & L(H_0^{i}  ) = \lbrace (A_1,B_1, \cdots,  A_{i-1}, B_{i-1}),(A_1,B_1, \cdots,  A_{i-1}, B_{i-1}, A_i, \epsilon_i I)\rbrace .
\end{align*}
L'intersection $ ( L(H_0^1 \cup  \cdots \cup H_0^{i-1}  ) \times \N(\Sigma_0^{i}) )\cap L(H_0^{i}  )$ est donc transverse dans $ \N(\Sigma_0^{i-1}) \times \N(\Sigma_0^{i}) $, et correspond à 
\[ \left\lbrace (A_1, \epsilon_1 I, \cdots,  A_{i-1},  \epsilon_{i-1} I),(A_1,\epsilon_1 I, \cdots,  A_{i-1}, \epsilon_{i-1} I, A_i, \epsilon_i I)\right\rbrace \simeq SU(2)^i .\] 
La projection sur $\N(\Sigma_0^{i})$ induit ainsi un plongement sur $ L(H_0^1 \cup  \cdots \cup H_0^{i}  )$, qui est bien simplement connexe, et compatible (car disjointe) avec l'hypersurface $R_i$.

Par ailleurs, l'hypothèse sur les disques pseudo-holomorphe d'aire nulle est automatique car $ L(H_0^1 \cup  \cdots \cup H_0^{i}  )$ est disjointe de $R_i$, et celle sur les cylindres se vérifie de manière analogue à l'énoncé correspondant dans la démonstration du lemme \ref{hyptechniques} : étant donné que l'un des trois morceaux est envoyé sur un point, il peut être retiré du cylindre, et le cylindre matelassé correspond alors à un disque matelassé comme dans la figure~\ref{disque}, à conditions au bord dans $  L(H_0^1 \cup  \cdots \cup H_0^{i-1}  )$ et $  L(H_0^1 \cup  \cdots \cup H_0^{i}  )$,  et à conditions aux coutures dans $L(H_0^{i}  )$. Un tel disque matelassé se projette sur un disque de $\N(\Sigma_0^{i-1})$ d'aire nulle, et à bord dans $  L(H_0^1 \cup  \cdots \cup H_0^{i-1}  )$, qui ne peut exister car cette dernière Lagrangienne est disjointe du lieu de dégénérescence $R_{i-1}$ de la forme $\tilde{\omega}_{i-1}$.

\begin{figure}[!h]
    \centering
    \def\svgwidth{.65\textwidth}
    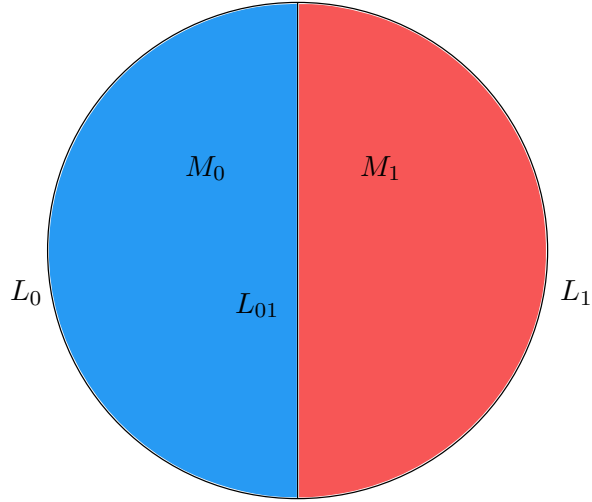
      \caption{Un disque matelassé.}
      \label{disque}
\end{figure}

\end{proof}

Ainsi, la correspondance Lagrangienne généralisée $\underline{L}(H_0, p_0,c_0)$ est équivalente dans $\Symp$ à la Lagrangienne $L(H_0, p_0,c_0)$, de même pour $H_1$. En vertu du théorème \ref{compogeom}, on obtient alors :
\begin{prop}\label{calculscindement} Dans ces conditions, $HSI(Y,c,z) \simeq HF(L_0,L_1; R)$, où $L_i = L(H_i,c_i,p_i) \subset \Nc(\Sigma,p)$. En particulier, si $c=0$, on retrouve les groupes $HSI(Y,z)$ définis par Manolescu et Woodward.
\end{prop}
\arnaque

\section{Renversement d'orientation}

Soit $\underline{L}\in Hom_{\Symp}(pt,pt)$, on peut définir la cohomologie $HF^*(\underline{L})$, c'est-à-dire l'homologie du complexe dual de $CF_*(\underline{L})$.

On rappelle que l'on note $\underline{L}^T$ la correspondance Lagrangienne généralisée obtenue en renversant les flèches. Si $x,y\in \I(\underline{L})$ sont deux points d'intersection généralisés,  une trajectoire matelassé $\underline{u}$ à conditions aux coutures dans $\underline{L}$ et allant de $x$ à $y$ peut être vue comme une trajectoire matelassée à conditions aux coutures dans $\underline{L}^T$ et allant de $y$ à $x$. Il vient donc : 
\begin{prop}
$HF_*(\underline{L}^T) \simeq HF^*(\underline{L})$
\end{prop}
\arnaque

Si $Y,c$ est une 3-variété munie d'une classe, et $z$ un point base,  $W$ l'éclatement, et $\overline{W}$ muni de l'orientation opposée, alors $\underline{L}(\overline{W},c) = \underline{L}(W,p,c)^T$. Ainsi, si l'on note $HSI_*$ ce que l'on a noté $HSI$ jusqu'ici, et $HSI^*$ la cohomologie :

\begin{prop}$HSI_*(\overline{Y},c,z) \simeq HSI^*(Y,c,z)$
\end{prop}
\arnaque
  
\section{Somme connexe}

Rappelons la formule de Künneth pour l'homologie de Floer matelassée, voir par exemple  \cite[Theorem 5.2.6]{WWqfc}  dans le cadre monotone non-relatif, dont la preuve se généralise de manière identique au cas qui nous intéresse :

\begin{prop}(Formule de Kunneth, \cite[Theorem 5.2.6]{WWqfc})
\label{kunnethfloer}
 Soient $\underline{L}$ et $\underline{L'}$ deux chaînes de correspondances Lagrangiennes entre $pt$ et $pt$, alors \[HF (\underline{L},  \underline{L'} ) \simeq HF (\underline{L}) \otimes HF (\underline{L'})\oplus \mathrm{Tor}(HF (\underline{L}) , HF (\underline{L'}))[-1],\] où $\mathrm{Tor}$ désigne le foncteur $\mathrm{Tor}$, et $[-1]$ désigne un décalage des degrés de $-1$.
\end{prop}
\arnaque

\begin{prop}\label{sommecnx} (Somme connexe)
\begin{align*}
HSI(Y \# Y' , c+ c') \simeq & HSI(Y, c) \otimes HSI(Y' , c' )  \\ & \oplus \mathrm{Tor}(HSI(Y, c) , HSI(Y' , c' ))[-1].
\end{align*}

\end{prop}

\begin{proof} Soient $\underline{L}$ et $\underline{L'}$ des chaînes de correspondances Lagrangiennes associées à $(Y, c)$, $(Y' , c' )$, qui sont toutes deux des morphismes de $\Symp$ du point vers le point. Alors $\underline{L}, \underline{L'}$ est une correspondance Lagrangienne généralisée associée à $(Y, c)\# (Y' , c' )$. Le résultat découle alors de la proposition précédente. 
\end{proof}

\section{Caractéristique d'Euler}
Les groupes $HSI(Y,c,z)$ étant relativement $\Z{8}$-gradué, leur caractéristique d'Euler $\chi(HSI(Y,c,z))$ est définie à un signe près.

\begin{prop}\label{eulercara} Si $b_1(Y) = 0$, $\left|  \chi(HSI(Y,c,z)) \right| =  | H_1(Y;\zz)|$, sinon, $\chi(HSI(Y,c,z)) = 0$.

\end{prop}

\begin{proof}Lorsque $c=0$, cela est établi par Manolescu et Woodward, \cite[Parag. 7.1]{MW} : pour $c_0 =  c_1 = 0$, la caractéristique est donnée par le nombre d'intersection $[L(H_0,c_0 )] . [L(H_1,c_1 )]$  des deux Lagrangiennes dans l'espace des modules du scindement, et ce nombre est calculé dans \cite[Prop. III.1.1, (a),(b)]{cassonpup}. Si $c\neq 0$, le nombre d'intersection est inchangé, en effet $L(H_i,c_i)$ peut être envoyée sur $L(H_i, 0)$ par une isotopie (non-Hamiltonienne) de $SU(2)^{2h}$ de la façon suivante : une présentation du groupe fondamental étant fixée, $L(H_i,c_i)$ est définie par des équations  
\[ \left\lbrace (A_1,B_1,\cdots)\ |\ A_1 = \epsilon_1 I, A_2 = \epsilon_2 I, \cdots \right\rbrace, \]
où $\epsilon_i = \pm 1$. Il suffit de prendre un chemin dans $SU(2)$ reliant $I$ et $-I$ pour ramener les $\epsilon_i$ à $+I$.

\end{proof}

\section{Variétés de genre de Heegaard 1}

Manolescu et Woodward ont calculé les groupes d'homologie HSI pour des variétés de genre de Heegaard 1 lorsque la classe $c$ est nulle. Nous complétons leurs calculs pour toutes les classes.

\begin{prop}\label{genreun}\begin{itemize}

\item[$(i)$] Pour $Y=S^2 \times S^1$ et $c\in H_1(Y; \Z{2})$,

\[ HSI(Y,c) = \begin{cases} \zz[0] \oplus \zz[3]\text{ si  }c=0, \\ \lbrace 0\rbrace \text{ sinon.} \end{cases} \]

\item[$(ii)$] $HSI(L(p,q),c)$ est de rang $p$ pour toute classe $c$. De plus, l'homologie est concentrée en degré pair.
\end{itemize}
\end{prop}

\begin{proof}
$(i)$ : Pour $c=0$, $HSI(S^2 \times S^1)$ a été calculée par Manolescu et Woodward. Pour $c\neq 0$, avec $\Sigma$ le scindement de genre 1 et $A$, $B$ les holonomies le long d'une base du groupe fondamental de $\Sigma$ dont la première courbe borde un disque dans les deux corps à anses, les deux Lagrangiennes $\lbrace A=I\rbrace$ et $\lbrace A= -I\rbrace$ sont disjointes.

Concernant $L(p,q)$, on peut choisir un scindement et un système de coordonnées tels que les deux Lagrangiennes aient pour équation : $L_0 = \lbrace B = I \rbrace$ et $L_1 = \lbrace A^p B^{-q}  = \pm I\rbrace$. Elles s'intersectent de manière "clean" en une réunion de copies de $S^2$ et, selon la parité de $p$, d'un ou deux points. On peut déplacer l'une des deux Lagrangiennes par une isotopie Hamiltonienne de sorte que l'intersection soit transverse, et que chaque copie de $S^2$ donne lieu à deux points. Il y a donc $p$ points d'intersection, par ailleurs on sait que ce nombre correspond au nombre d'intersection des deux Lagrangiennes. Ainsi le complexe est lacunaire, et la différentielle est nulle.

\end{proof}

\chapter{Chirurgie}\label{chapchir}

Dans cette partie, nous étudions l'effet d'une chirurgie de Dehn entière sur l'homologie HSI. 

\begin{defi} Une \emph{triade de chirurgie} est un triplet de 3-variétés $Y_\alpha$, $Y_\beta$ et $Y_\gamma$ obtenues à partir d'une 3-variété compacte, orientée, avec un bord de genre 1, en recollant un tore solide le long du bord, de façon à envoyer le méridien sur trois courbes simples $\alpha$, $\beta$ et $\gamma$ respectivement, telles que $\alpha. \beta = \beta. \gamma =\gamma. \alpha = -1$.
\end{defi}

Rappelons que dans les théories similaires ( Seiberg-Witten, \cite{KMtri} ; Heegaard-Floer, \cite{OStri}) les invariants d'une triade de chirurgie sont reliés par une suite exacte longue du type :

\[\cdots\rightarrow HF^+(Y_\alpha ) \rightarrow HF^+(Y_\beta) \rightarrow HF^+(Y_\gamma)\rightarrow \cdots .\]
Nous obtenons une suite exacte légèrement différente, mais néanmoins similaire au triangle de Floer pour les instantons, cf \cite{Floertri}, ainsi qu'à ceux obtenus par Wehrheim et Woodward dans \cite{WWtriangle} pour leur "Floer Field theory" \cite{WWfft} : la classe d'homologie d'un des trois groupes est modifiée par l'âme de la chirurgie :

\begin{theo}[Suite exacte de chirurgie]\label{trianglechir} Soit $(Y_\alpha , Y_\beta, Y_\gamma)$ une triade de chirurgie obtenue à partir de $Y$ comme dans la définition précédente, $c\in H_1(Y;\Z{2})$, et pour $\delta \in \lbrace \alpha,\beta,\gamma\rbrace$, $c_\delta \in H_1(Y_\delta;\Z{2})$ la classe induite à partir de $c$ par les inclusions. Soit également  $k_\alpha \in H_1(Y_\alpha;\Z{2})$ la classe correspondant à l'\^ame du tore solide. Alors, il existe une suite exacte longue :
\[ \cdots\rightarrow HSI(Y_\alpha ,c_\alpha+ k_\alpha) \rightarrow HSI(Y_\beta,c_\beta) \rightarrow HSI(Y_\gamma,c_\gamma)\rightarrow \cdots . \]

\end{theo}

\begin{remark}Par symétrie cyclique des trois courbes, la modification $k_\alpha$ peut également être mise sur $Y_\beta$ ou $Y_\gamma$. Il est également possible de démontrer une suite exacte plus symétrique :
\[ \cdots\rightarrow HSI(Y_\alpha ,c_\alpha+ k_\alpha) \rightarrow HSI(Y_\beta,c_\beta+ k_\beta) \rightarrow HSI(Y_\gamma,c_\gamma+ k_\gamma)\rightarrow \cdots,\]
avec $k_\beta$ et $k_\gamma$ des classes similaires à $k_\alpha$.
\end{remark}

Afin de démontrer ce théorème, nous remarquons qu'un twist de Dehn du tore troué $T'$ autour d'une courbe simple non-séparante induit un symplectomorphisme au niveau de l'espace des modules $\Nc ({T} ')$. Ce symplectomorphisme peut s'exprimer comme le flot d'un Hamiltonien en dehors de la sphère Lagrangienne correspondante aux connexions dont l'holonomie le long de la courbe $\gamma$ vaut $-I$. Ce n'est pas à priori un twist de Dehn généralisé, néanmoins il est  possible de construire un twist de Dehn généralisé à partir de ce symplectomorphisme qui permet d'obtenir la suite exacte annoncée,  en appliquant un analogue de la suite exacte de Seidel (théorème \ref{quilttri} ) pour l'homologie matelassée.

\section{Twists de Dehn généralisés et homologie matelassée}\sectionmark{Twists de Dehn généralisés}
Toutes les variétés symplectiques, Lagrangiennes, et correspondances Lagrangiennes apparaissant dans la suite satisferont, sauf mention contraire, aux hypothèse de la catégorie $\Symp$. Soient $M_0$, $M_1$, ..., $M_k$ des objets de $\Symp$,  

\[\underline{L} = \left(  \xymatrix{   M_0 \ar[r]^{L_{01}} &  M_1 \ar[r]^{L_{12}} &  M_2 \ar[r]^{L_{23}} & \cdots \ar[r]^{L_{(k-1)k}} & M_k \ar[r]^{L_{k}} & pt} \right), \]
une correspondance Lagrangienne généralisée, $S \subset M_0$ une sphère Lagrangienne disjointe de l'hypersurface $R_0$, et $\tau_S\in Symp(M_0)$ un twist de Dehn généralisé autour de $S$, comme défini dans la section \ref{rappeltwists} (ou \cite[Section 1.2]{Seidel}). Le but de cette section est de démontrer le théorème suivant :

\begin{theo}
\label{quilttri}Soient $L_0\subset M_0$ une Lagrangienne,  $S$ et $\underline{L}$ comme précédemment. On suppose de plus que $\dim S >2$. Il existe une suite exacte longue :

\[ \ldots \rightarrow HF(\tau_S L_0, \underline{L}) \rightarrow HF(L_0, \underline{L}) \rightarrow HF(L_0,S^T, S, \underline{L})\rightarrow \cdots \]

\end{theo}

\begin{remark}L'hypothèse $\dim S >2$ intervenant dans ce théorème garantie la monotonie d'une fibration de Lefschetz. Un énoncé similaire est probablement vrai, néanmonis nous nous limiterons à ces dimensions, car dans nos applications les sphères seront de dimension 3.
\end{remark}

\begin{remark}
\label{remtriww}
 Si l'on pose $M = M_0^- \times M_1\times M_2^-\times ... \times M_k^\pm $, $\widetilde{L}_0 = L_0 \times L_{12}\times \cdots$,  et $\widetilde{L}_1 = L_{01}\times \cdots $, alors $C = S\times M_1 \times M_2 \times \cdots \times M_k \subset M$ est une sous-variété coisotrope sphériquement fibrée au-dessus de $B = M_1 \times M_2^- \times \cdots \times M_k^\pm$, $\tau_C = \tau_S \times id_B$ est un twist de Dehn fibré, et le triangle précédent peut s'écrire sous la forme :

\[ HF(\tau_C \widetilde{L}_0, \widetilde{L}_1) \rightarrow HF(\widetilde{L}_0,\widetilde{L}_1) \rightarrow HF(\widetilde{L}_0\times C^T, C\times \widetilde{L}_1)\rightarrow \cdots, \]
qui est essentiellement le triangle de \cite[Theorem 1.3]{WWtriangle}. Nous allons le démontrer dans le cadre dans lequel nous l'utiliserons (la catégorie $\Symp$) qui n'est pas le même que celui de Wehrheim et Woodward. La preuve est analogue, il s'agit essentiellement de vérifier qu'il ne se produit pas de bubbling sur les diviseurs à chaque fois qu'intervient un raisonnement de dégénérescence d'espaces de modules.
\end{remark}

\subsection{Rappels sur les twists de Dehn généralisés}\label{rappeltwists}

Nous renvoyons à \cite[Section 1]{Seidel} pour plus de détails.

\paragraph{Twist de Dehn dans $T^*S^n$}

On considère le fibré cotangent $T = T^*S^n$ munie de sa forme symplectique standard $\omega = dp \wedge dq$. Si l'on munit $S^n$ de la métrique ronde, $T$ s'identifie à \[\left\lbrace (u,v) \in \rr^{n+1} \times \rr^{n+1}~|~ |v| = 1, \langle u.v \rangle = 0  \right\rbrace. \] On note $T(\lambda) = \lbrace (u,v) \in T ~|~ |u| \leq \lambda \rbrace$, en particulier $T(0)$ désigne la section nulle.

La fonction $\mu (u,v) = |u|$ engendre une action du cercle sur le complémentaire de la section nulle, son flot au temps $t$ est donné par :

\[ \sigma_t (u,v) = (cos(t) u -sin(t) |u|v, cos(t) v + sin(t) \frac{u}{|u|}),  \]
et le flot au temps $\pi$ se prolonge à la section nulle par l'application antipodale, que l'on notera $\mathbb{A}$.

Soit $\lambda > 0$, et $R\colon \rr\rightarrow \rr$ une fonction lisse, nulle pour $t\geq \lambda$, et telle que $R(-t) = R(t) - t$. On considère l'Hamiltonien $H = R\circ \mu $  sur $T(\lambda) \setminus T(0)$ : son flot au temps $2\pi$ est donné par $\varphi^H_{2\pi} (u,v) =  \sigma_t (u,v)$, avec $t = R'(|u|)$, et il se recolle de manière lisse à la section nulle par l'application antipodale. Le symplectomorphisme obtenu $\tau$ est un twist de Dehn "modèle", de fonction angle   $R'(\mu(u,v))$.

\begin{defi}\label{concave}On dira qu'un twist de Dehn modèle est concave si la fonction $R$ intervenant dans la définition du twist $\tau$ est strictement concave et décroissante, c'est-à dire vérifie  $R'(t)\geq 0$ et $R''(t)< 0$ pour tout $t\geq 0$.
\end{defi}

Seidel démontre le résultat suivant dans un cadre un peu plus général : il autorise les fonctions angle à osciller sensiblement, de manière "$\delta$-wobbly", avec $0\leq \delta <\frac{1}{2}$.  L'énoncé suivant, correspondant à $\delta = 0$, nous suffira.

\begin{lemma} \label{intertwistbis} (\cite[Lemma 1.9]{Seidel}) Supposons que le twist $\tau$ est concave. Soient $F_0 = T(\lambda)_{y_0}$ et $F_1 = T(\lambda)_{y_1}$ des fibres au-dessus de deux points $y_0,y_1 \in S^n$  . Alors $\tau(F_0)$ et $F_1$ s'intersectent transversalement en un unique point $y$. De plus, ce point vérifie \[2\pi R'(y) = d(y_0,y_1),\] où $d$ désigne la distance standard sur $S^n$.
\end{lemma}

\paragraph{Twist de Dehn autour d'une sphère Lagrangienne}

Si $S\subset M$ est une sphère Lagrangienne, elle admet un voisinage de Weinstein, c'est-à-dire un plongement symplectique $\iota\colon T(\lambda) \rightarrow M$ pour un $\lambda >0$, avec $\iota(T(0)) = S$. Ainsi, un twist de Dehn modèle de $ T(\lambda)$ définit un symplectomorphisme de $M$, noté $\tau_S$, à support dans $\iota (T(\lambda))$. 

On dit qu'un symplectomorphisme de $M$ est un twist de Dehn généralisé autour de $S$ s'il est isotope à un tel twist de Dehn modèle.

\begin{remark}Si deux twists de Dehn modèles de $T(\lambda)$ diffèrent toujours par une isotopie Hamiltonienne de $T(\lambda)$, un twist de Dehn autour de $S$ dépend du paramétrage de $S$, voir \cite{rizell}.
\end{remark}

\subsection{Homologie à coefficients dans l'anneau du groupe $\rr$}

L'argument principal de la preuve du théorème de Seidel repose sur le fait que les complexes de Floer possèdent une graduation sur $\rr$ donnée par l'action symplectique, car les variétés symplectiques et les Lagrangiennes qu'il considère sont exactes. Les "termes dominants" des morphismes intervenant dans la suite exacte par rapport à la filtration induite par cette graduation correspondent à des courbes pseudo-holomorphes de petite énergie. Il suffit alors de vérifier que ces derniers induisent une suite exacte. 

Lorsque les variétés symplectiques et les Lagrangiennes ne sont plus exactes mais seulement monotones, l'action symplectique n'est plus définie que modulo $M = \kappa N$, avec $\kappa$ la constante de monotonie et $N$ le nombre de Maslov minimal, c'est-à-dire l'aire de la plus petite sphère pseudo-holomorphe (voir section \ref{secbasseénergie}). L'approche de Wehrheim et Woodward consiste alors à encoder cette énergie dans la puissance d'un paramètre formel $q$, via l'anneau du groupe $\rr$ :

\[ \Lambda =  \left\lbrace \left. \sum_{k=1}^{n}{a_k q^{\lambda_k}} ~\right|~ n\geq 1,~a_k \in \zz,~ \lambda_k \in \rr  \right\rbrace .\]

Le complexe de Floer à coefficients dans cet anneau est alors  le $\Lambda$-module libre $CF (\underline{L};\Lambda):= CF(\underline{L})\otimes_\zz \Lambda$, muni de la différentielle $\partial_\Lambda$ définie par :\

\[ \partial_\Lambda x_- = \sum_{x_+}{\sum_{\underline{u}\in \mathcal{M}(x_-,x_+)}{o(\underline{u}) q^{A(\underline{u})}} x_+}, \]
où $x_+,x_-\in \I(\underline{L})$ sont des points d'intersection généralisés, $\mathcal{M}(x_-,x_+)$ représente l'espace des modules des trajectoires de Floer généralisées d'indice 1 et d'intersection nulle avec $\underline{R}$ (modulo translation), $o(\underline{u}) = \pm 1$ est l'orientation du point $\underline{u}$ dans l'espace des modules construite dans \cite{WWorient} à partir de l'unique structure spin relative sur $\underline{L}$, et $A(\underline{u})$ est l'aire symplectique pour les formes monotones $\tilde{\omega}_i$.

L'homologie de $(CF(\underline{L};\Lambda),\partial_\Lambda)$ est alors le $\Lambda$-module noté $HF(\underline{L};\Lambda)$. Généralement, cette homologie peut être différente de l'homologie à coefficients dans $\zz$, nous verrons cependant dans la section \ref{preuvedutriangle} que la monotonie de $\underline{L}$ entraîne $ CF(\underline{L};\Lambda) \simeq CF(\underline{L};\zz) \otimes_\zz \Lambda$, et $ HF(\underline{L}) \simeq HF(\underline{L};\Lambda) / ( q-1)$.

\subsection{Suite exacte courte au niveau des complexes de chaînes}

La proposition suivante résulte du lemme \ref{intertwistbis} :

\begin{prop}\label{inter}
Soit $\iota\colon T(\lambda) \rightarrow M_0$ un plongement symplectique, $\tau_S$ un twist de Dehn modèle concave associé à $\iota$. On suppose :
\begin{itemize}

\item[i)] que $\I(L_0, \underline{L})$ est disjoint de $\iota\left(  T(\lambda) \right)$, 

\item[ii)] que  $L_0 \cap \iota( T(\lambda))$ est une réunion de fibres :

\[  \iota^{-1}( L_0) = \bigcup_{y\in \iota^{-1}( L_0\cap S)}{T(\lambda)_y} \subset T(\lambda)\]

\item[iii)] que  $L_{01}$ et $S\times M_1$ s'intersectent transversalement dans $M_0\times M_1$, et que, en notant $\pi \colon  \iota\left(  T(\lambda) \right) \rightarrow S$ la projection,

\[ L_{01}\cap \left(  \iota\left(  T(\lambda)  \right)  \times M_1 \right) = (\pi\times id_{M_1})^{-1}(L_{01} \cap \left( S\times M_1 \right) ). \]  

\end{itemize}

Alors, il existe deux injections naturelles \[i_1\colon \I(\tau_S L_0,S^T, S, \underline{L}) \rightarrow \I(\tau_S L_0, \underline{L})\] et \[i_2\colon \I( L_0, \underline{L}) \rightarrow \I(\tau_S L_0, \underline{L})\] telles que

\[ \I(\tau_S L_0, \underline{L}) = i_2 \left( \I(L_0, \underline{L}) \right)  \sqcup  i_1 \left( \I(\tau_S L_0,S^T, S, \underline{L}) \right) . \]

\end{prop}
\begin{proof}

Notons $\nu S = \iota\left(  T(\lambda) \right)$,

\begin{align*}
\I(\tau_S L_0, \underline{L}) = & \I(\tau_S L_0, \underline{L}) \cap \left( M_0\setminus \nu S\right) \times M_1 \times \cdots \times M_k \\
  &\sqcup \I(\tau_S L_0, \underline{L}) \cap \nu S \times M_1 \times \cdots \times M_k .
\end{align*}
 
D'après $i)$ et le fait que $\tau_S$ est à support dans $\nu S$, 

\[ \I(\tau_S L_0, \underline{L}) \cap \left( M_0\setminus \nu S\right) \times M_1 \times \cdots \times M_k = \I(L_0, \underline{L}). \]
L'application $i_2$ peut donc être choisie comme étant l'identité. D'après $ii)$,  il vient :

\[\I(\tau_S L_0, \underline{L}) \cap \nu S \times M_1 \times \cdots \times M_k = \bigcup_{x_0 \in L_0\cap S} {\I( \tau_S \left( T(\lambda)_{x_0}\right) , \underline{L} ) }. \]

Soit $\underline{y} \in \I(S, \underline{L})$, par hypothèse $T(\lambda)_{y_0} \times  \left\lbrace  y_1 \right\rbrace \subset L_{01}$, et par le lemme~\ref{intertwistbis}, $\tau_S \left( T(\lambda)_{x_0}\right)$ et $T(\lambda)_{y_0}$ s'intersectent en exactement un point $z$. On définit alors $i_1$ en posant $i_1(x_0, y_0, y_1, \cdots) = (z, y_1, \cdots)$. Cette application réalise bien une bijection entre $\I(\tau_S L_0,S^T, S, \underline{L})$ et \[\bigcup_{x_0 \in L_0\cap S} {\I( \tau_S \left( T(\lambda)_{x_0}\right) , \underline{L}) }, \]
en effet son inverse est donné par l'application $(z, y_1, \cdots) \mapsto (x_0, y_0, y_1, \cdots),
$ où $ x_0 = \pi (z) $, et $ y_0 = \pi ( \tau_S^{-1} (z))$.

\end{proof}

\begin{remark}\label{justifinter}Quitte à déplacer les Lagrangiennes par des isotopies Hamiltoniennes et à choisir un $\lambda$ suffisamment petit, il est toujours possible de se ramener aux hypothèses de la proposition \ref{inter}. En effet toutes les intersections peuvent être rendues transverses, il est ensuite possible de choisir le plongement $\iota$ de sorte à avoir $ii)$ et $iii)$.

\end{remark}

D'où une décomposition en somme directe des $\Lambda$-modules : 

\[ CF (\tau_S L_0, \underline{L};\Lambda) = CF (\underline{L};\Lambda) \oplus CF (\tau_S L_0,S^T, S, \underline{L};\Lambda), \]
et une suite exacte courte (de $\Lambda$-modules et non de complexes de chaînes) :
\begin{equation}\label{suitecourte}
0 \rightarrow CF (\tau_S L_0,S^T, S, \underline{L};\Lambda ) \rightarrow CF (\tau_S L_0, \underline{L} ;\Lambda) \rightarrow CF (L_0, \underline{L};\Lambda) \rightarrow   0 .
\end{equation}

\begin{remark} La sphère $S$ étant invariante par le twist, on a les isomorphismes de complexes suivants :
 \begin{align*} CF (\tau_S L_0,S^T, S, \underline{L};\Lambda) &\simeq CF (\tau_S L_0,\tau_S S^T, S, \underline{L};\Lambda) \\ 
  &\simeq CF (L_0,S^T, S, \underline{L};\Lambda).  
\end{align*}

\end{remark}

\subsection{Fibrations de Lefschetz matelassées}

La stratégie pour démontrer la suite exacte longue consiste à approcher les flèches de la suite exacte courte par des morphismes de complexes. Afin de commuter avec les différentielles, ces morphismes seront construits en comptant des quilts pseudo-holomorphes, plus précisément des sections pseudo-holomorphes de fibrations de Lefschetz matelassées. Rappelons les définitions de ces objets, tirées de \cite{WWquilts} et que nous adaptons au cadre de la catégorie $\Symp$. 

\begin{defi}\label{fiblefschetz}
Soit $S$ une surface de Riemann compacte, avec ou sans bord. Une \emph{fibration de Lefschetz au-dessus de  $S$}, dans le contexte de la catégorie $\Symp$, est la donnée de $(E, \pi, \omega, \tilde{\omega}, R, \tilde{J})$, avec :

\begin{itemize}
\item $E$ est une variété orientable compacte, de dimension $2n +2$,
\item $\pi\colon E\rightarrow S$ est une application différentiable surjective, telle que $\partial E = \pi^{-1} (\partial S)$, et submersive sauf en un nombre fini de points critiques $E^{crit}$, disjoint de $\partial E$,
\item  $\tilde{J}$  est une structure presque complexe sur $E$, intégrable au voisinage de $E^{crit}$, telle que la différentielle de $\pi$ est $\cc$-linéaire, et qu'au voisinage de chaque point critique, dans des cartes holomorphes, $\pi$ a pour expression :

\[\pi(z_0, \cdots, z_n) = \sum_{i}{z_i^2},\]
\item $\omega$ et $\tilde{\omega}$ sont deux 2-formes fermées sur $E$ et non-dégénérées au voisinage des points critiques,

\item $R$ est une hypersurface presque complexe pour $\tilde{J}$, disjointe de $E^{crit}$, transverse aux fibres de $\pi$, et telle que pour toute fibre régulière $F$ de $\pi$, $(F, \omega_{|F}, \tilde{\omega}_{|F}, R \cap F, \tilde{J}_{|F})$ est un objet de $\Symp$.
\end{itemize}

 \end{defi}
 
\begin{remark}
On a supposé $\tilde{\omega}$ monotone seulement le long des fibres, mais d'après \cite[Prop. 4.6]{WWtriangle}, dès que $n\geq 2$, la forme $\tilde{\omega}$ est  monotone sur $E$.

\end{remark}

\begin{defi} Une \emph{surface matelassée avec bouts de type bandes} est la donnée :
\begin{enumerate}

\item D'une surface matelassée compacte $\underline{S}$

\item D'un ensemble fini de points marqués , entrants et sortants \[\mathcal{E} = \mathcal{E}_- \sqcup \mathcal{E}_+ \subset \partial \underline{S}.\]

\item De "bouts de type bandes" associés à chaque point marqué $e \in \mathcal{E}$, c'est-à-dire des applications matelassées holomorphes

\[ \epsilon_e \colon  \begin{cases}
 [0, + \infty ) \times [0,N_e] \rightarrow \underline{S}  &\text{si $e \in \mathcal{E}_+$ est une sortie}\\
 ( -\infty ,0] \times [0,N_e] \rightarrow \underline{S} &\text{si $e \in \mathcal{E}_-$ est une entrée}
\end{cases}  \]

ayant pour limite $e$ en $\pm \infty$, et dont l'adhérence de l'image est un voisinage de $e$ dans $\underline{S}$. Si $N_e$ représente le nombre de morceaux $S_1, S_2, \cdots$, $S_{N_e}$ touchant $e$, $[0, \pm \infty ) \times [0,N_e]$ est vue comme une surface matelassée avec $N_e$ bandes de largeur $1$ cousues parallèlement. L'application $ \epsilon_e$ correspond à la donnée d'applications compatibles aux coutures : 
\[ \epsilon_{k,e} \colon [0, \pm \infty ) \times [k-1,k] \rightarrow S_k \]

\end{enumerate}

\end{defi}

\begin{defi} \label{quiltfiblefschetz}
 Soit $\underline{S}$ une surface matelassée avec des bouts de type bandes, une \emph{fibration de Lefschetz matelassée au-dessus de $\underline{S}$, avec conditions aux bords et aux coutures}, est la donnée :
\begin{enumerate}

\item Pour chaque morceau $S_k$, d'une fibration de Lefschetz  $\pi_k \colon E_k \rightarrow S_k$  comme dans la définition \ref{fiblefschetz}.

\item D'un ensemble de conditions aux bords/coutures  Lagrangiennes, noté $\underline{F}$, consistant en :

\begin{itemize}

\item[$(a)$] Pour une couture $\sigma = \lbrace I_{k_0,b_0},I_{k_1,b_1} \rbrace \in \mathcal{S}$, une sous-variété 
\[ F_{\sigma} \subset {E_{k_0}}_{|I_{k_0,b_0}} \times_{|I_{k_0,b_0}} \varphi_{\sigma}^* {E_{k_1}}_{|I_{k_1,b_1}}, \] isotrope pour les formes $\tilde{\omega_i}$, transverse aux fibres, et telle que l'intersection avec toute fibre est une correspondance Lagrangienne vérifiant les hypothèses de $\Symp$. On rappelle que $\varphi_\sigma \colon I_{k,b} \rightarrow I_{k',b'}$ désigne le difféomorphisme analytique réel servant à identifier les coutures.

\item[$(b)$] Pour un bord $I_{k,b}\notin \bigcup_{\sigma\in \mathcal{S}}\sigma$, une sous-variété $F_{k,b} \subset {E_k} |_{I_{k,b}}$, transverse aux fibres, et telle que son intersection avec toute fibre est une Lagrangienne vérifiant les hypothèses de $\Symp$.

\end{itemize}
\item De trivialisations au-dessus des bouts $\epsilon_{k,e}$ :

\[{\epsilon_{k,e}}  ^*{(E_k)}\simeq (E_k)_e \times [0, \pm \infty ) \times [k-1,k].  \]

telles que les conditions aux bord/coutures soient constantes dans ces identifications : $F_{\sigma} \simeq (F_{\sigma})_e \times [0,\pm \infty ) \times \lbrace k\rbrace$,  et $F_{k,b} \simeq (F_{k,b})_e \times [0,\pm \infty ) \times \lbrace k\rbrace$.

\end{enumerate}

\end{defi}

\paragraph{Invariant relatif associé à une fibration de Lefschetz matelassée}  

Soit $\underline{\pi}\colon (\underline{E},\underline{F}) \rightarrow \underline{S}$ une fibration de Lefschetz matelassée comme précédemment, et $\underline{J} = ( J_k)_k$ une famille de structures presque complexes sur $\underline{E}$, qui coïncide avec les structures de référence $\widetilde{\underline{J}}$ au voisinage des hypersurfaces $\underline{R}$, et qui rend les projections pseudo-holomorphes, et compatibles avec les formes $\omega_k$ le long des fibres,

Si $u\colon (\underline{S},\mathcal{S})\rightarrow (\underline{E},\underline{F})$ est une section pseudo-holomorphe, le linéarisé de l'opérateur de Cauchy-Riemann  est défini par :
\[
 D_u \colon \begin{cases}  \Omega^0(u^*T^{vert}\underline{E}, u^*T^{vert}\underline{F} ) \rightarrow &  \Omega^{0,1}(u^*T^{vert}\underline{E}) \\

  \xi \mapsto & \frac{d}{dt}_{t=0} \Pi_{t\xi}^{-1}\overline{\partial}_J \exp_u (t \xi) ,

\end{cases}
\]
où $\Omega^0(u^*T^{vert}\underline{E}, u^*T^{vert}\underline{F} )$,  désigne l'espace des sections matelassées de la fibration $u^*T^{vert}\underline{E}$ à valeurs dans $u^*T^{vert}\underline{F}$ au-dessus des coutures (pour des normes de Sobolev adéquates),  $\Omega^{0,1}(u^*T^{vert}\underline{E})$ désigne les $(0,1)$-formes à valeurs dans cette fibration, $\overline{\partial}_J u = \frac{1}{2} (du + J(u)\circ du \circ j)$ désigne l'opérateur de Cauchy-Riemann associé à $J$, et $\Pi_{t\xi} \colon T_{u(x)}{M}\rightarrow T_{\exp_{u(x)} (t\xi)}{M}$ désigne le transport parallèle.

Dès que les conditions aux bouts sont transverses, $D_u$ est un opérateur Fredholm, voir \cite[Lemma 3.5]{WWquilts}, et surjectif pour des structures presque complexes génériques, voir \cite[Theorem 4.11]{WWtriangle}.

Pour de telles structures presque complexes, l'espace des modules  des sections $  \underline{s}\colon  \underline{S} \rightarrow \underline{E}$ pseudo-holomorphes, de conditions aux bords et coutures $\underline{F}$, d'intersection nulle avec la famille d'hypersurfaces $\underline{R}$ , et ayant pour limites \[\underline{x} \in \prod_{e\in \mathcal{E}_-(\underline{S})} \I(L^{(k_{e,0},b_{e,0})}  , \cdots , L^{(k_{e,l(e)},b_{e,l(e)})}), \]
et  
\[\underline{y} \in \prod_{e\in \mathcal{E}_+(\underline{S})} \I(L^{(k_{e,0},b_{e,0})}  , \cdots , L^{(k_{e,l(e)},b_{e,l(e)})})\] aux bouts correspondants est une réunion de variétés lisses $ \mathcal{M}(\underline{E},\underline{F},\underline{J},\underline{x},\underline{y})_k$ de dimensions $k\geq 0$. Leur dimension est donnée par l'indice de l'opérateur $D_u$. Cet indice généralise l'indice de Maslov et peut-être calculé à partir de données topologiques, voir \cite{WWorient}.

Dans ces conditions, $ \mathcal{M}(\underline{E},\underline{F},\underline{J},\underline{x},\underline{y})_0$ est une variété compacte de dimension 0, ce qui permet de définir l'application :

\begin{align*}
C\Phi_{\underline{E},\underline{F}} \colon & \bigotimes_{e\in \mathcal{E}_-(\underline{S})} CF(L^{(k_{e,0},b_{e,0})}  , \cdots , L^{(k_{e,l(e)},b_{e,l(e)})})\\
&\rightarrow \bigotimes_{e\in \mathcal{E}_+(\underline{S})} CF(L^{(k_{e,0},b_{e,0})} , \cdots , L^{(k_{e,l(e)},b_{e,l(e)})} )
\end{align*}
par la formule suivante :

\[ C\Phi_{\underline{E},\underline{F}}( \otimes_{e\in \mathcal{E}_-(\underline{S})}(x_e^0, \cdots x_e^{l(e)}) ) = \sum_{\otimes y_e^i} \sum_{\underline{s}\in \mathcal{M}(\underline{E},\underline{F},\underline{x},\underline{y})}{o(\underline{s}) q^{A(\underline{s})} \otimes_e (y_e^i)_i}, \]

Pour des structures presque complexes génériques, cette application commute avec les différentielles. Pour montrer cela, on applique un raisonnement standard en théorie de Floer : on observe que les coefficients de  $\partial C\Phi_{\underline{E},\underline{F}} - C\Phi_{\underline{E},\underline{F}} \partial$ représentent le cardinal du bord d'une variété compacte de dimension 1.

\begin{lemma} Il existe un ensemble $G_\delta$-dense de structures presque complexes sur $\underline{E}$ pour lesquelles la compactification de Gromov de l'espace des modules $ \mathcal{M}(\underline{E},\underline{F},\underline{J},\underline{x},\underline{y})_1$ est une variété compacte à bord de dimension 1, et son bord s'identifie à :

\begin{align*}
 \partial \mathcal{M}(\underline{E},\underline{F},\underline{J},\underline{x},\underline{y})_1  = &\bigcup_{\underline{x}'}{  \widetilde{\mathcal{M}}(\underline{x},\underline{x}')_1 \times  \mathcal{M}(\underline{E},\underline{F},\underline{J},\underline{x}',\underline{y})_0  } \\
 \cup& \bigcup_{\underline{y}'}{     \mathcal{M}(\underline{E},\underline{F},\underline{J},\underline{x},\underline{y}')_0 \times \widetilde{\mathcal{M}}(\underline{y}',\underline{y})_1}  ,
\end{align*}
où $\underline{x}'$ (resp. $\underline{y}'$) parcourt l'ensemble des générateurs du complexe de départ (resp. d'arrivée),  $\widetilde{\mathcal{M}}(\underline{x},\underline{x}')_1 $ et $\widetilde{\mathcal{M}}(\underline{y}',\underline{y})_1$ sont les quotients par $\rr$ des espaces de trajectoires de Floer matelassées d'indice 1, c'est-à-dire les coefficients intervenant dans la différentielle du complexe.

\end{lemma}

\begin{proof} Le fait que le second membre est inclus dans le premier est un résultat classique de recollement, voir \cite[Theorem 3.9]{WWquilts}. Le fait qu'il ne se produit pas d'autre type de dégénérescences provient du théorème de compacité de Gromov et du lemme \ref{nobubbling} qui suit.
\end{proof}
\begin{lemma}\label{nobubbling}Il existe un ensemble $G_\delta$-dense de structures presque complexes sur $\underline{E}$ pour lesquelles il ne se produit pas de bubbling dans les espaces de modules de sections d'indice inférieur ou égal à 1.

\end{lemma}

L'argument est analogue à celui de la preuve de \cite[Prop. 2.10]{MW}. Nous le rappelons ici dans notre cadre. Il se base sur le lemme suivant, qui assure que toutes les sections pseudo-holomorphes des espaces de dimension inférieure ou égale à 1 intersectent les diviseurs de manière transverse.

\begin{lemma}(voir \cite[Lemma 2.3]{MW})\label{codimtangence}
Il existe un sous-ensemble $G_\delta$-dense de structures presque complexes régulières sur $\underline{E}$ pour lesquelles les espaces de modules de sections pseudo-holomorphes sont lisses, et les sous-espaces consistant en les sections rencontrant  $\underline{R}$ en des points de tangence d'ordre $k$ avec $\underline{R}$ sont contenus dans des réunions de sous-variétés de codimension $2k$.
\end{lemma}
\begin{proof} C'est un analogue de \cite[Proposition 6.9]{CieliebakMohnke}, appliquée à chaque morceau du quilt. La proposition est énoncée pour des surfaces sans bord, mais la preuve s'adapte à notre cadre : le sous-ensemble de l'espace des modules universel $\lbrace (\underline{u},\underline{J})  ~ | ~ \overline{\partial}_{\underline{J}} \underline{u} = 0   \rbrace $ des couples de courbes et structures presque complexes telles que $u_i$ admet une tangence d'ordre $k$ en $R_i$ est une sous-variété de Banach de codimension $2k$. Il s'agit alors d'appliquer le théorème de Sard-Smale à la projection $(u,J)\mapsto J$ définie sur cet espace.

\end{proof}

\begin{proof}[Preuve du lemme \ref{nobubbling}]

D'après le lemme \ref{codimtangence}, pour des structures presque-complexes génériques, toutes les courbes des espaces de modules de dimension 0 ou 1 intersectent les hypersurfaces transversalement.

Un théorème de compacité analogue à celui concernant les courbes pseudo-holomorphes non-matelassées reste valable, voir \cite[Theorem 3.9]{WWquilts}. Soit $\underline{s}_\infty$ une limite de sections matelassées : c'est une application matelassée nodale consistant en une composante principale $\underline{u}_\infty$ éventuellement brisée, à laquelle est attachée des bulles à l'intérieur des patch, ainsi que des bulles matelassées (sphères matelassées consistant en deux disques cousus par leur bord, voir figure \ref{bubbling}), attachées au niveau des coutures, et des disques attachés aux bords.

\begin{figure}[!h]
    \centering
    \def\svgwidth{\textwidth}
    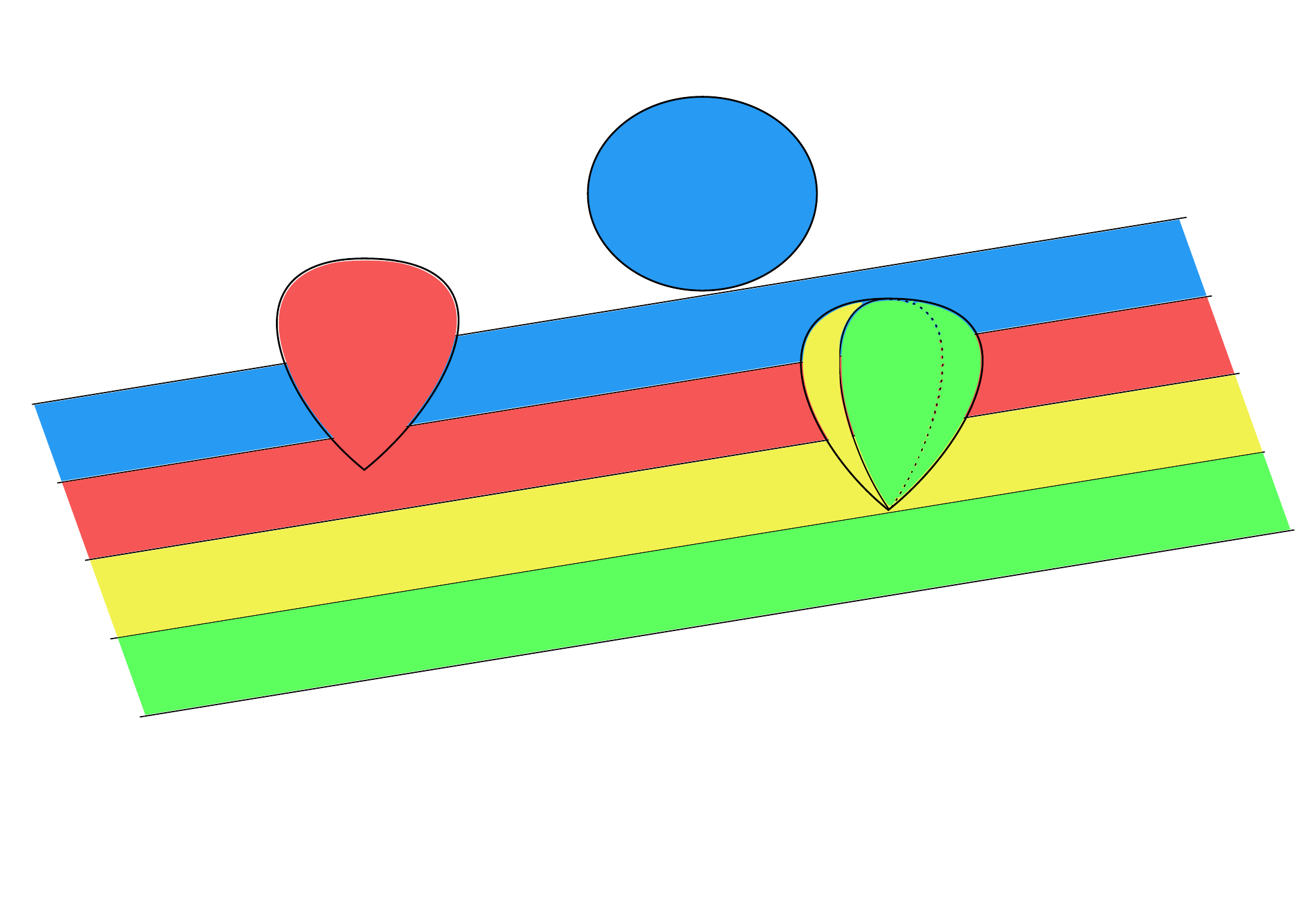
      \caption{Une section avec bubbling.}
      \label{bubbling}
\end{figure}

Chaque disque et chaque bulle, qui proviennent d'un zoom au voisinage d'un point de la base, est nécessairement contenu dans une fibre de $\underline{E}$. Ainsi, d'après le lemme 2.9 de \cite{MW}, tous les disques et toutes les bulles, ont un indice positif ou nul. Il est donc nul, sinon il serait plus grand que 4 (qui divise le nombre de maslov minimal), or la configuration initiale est d'indice plus petit que 2. Ainsi l'aire de ces disques ou de ces bulles pour les formes monotones $\tilde{\omega_i}$ est nulle : ils sont donc contenus dans les hypersurfaces $\underline{R}$. Les Lagrangiennes étant disjointes des hypersurfaces, il n'y a donc pas de disques.
Il n'y a donc à priori que des sphères,  matelassées ou non, d'aire nulle et donc contenues dans les hypersurfaces. Elles ont chacune un nombre d'intersection avec les hypersurfaces inférieur à -2, par définition de la catégorie $\Symp$. Or le nombre d'intersection total $\underline{s}_\infty . \underline{R}$ est nul, ainsi $\underline{u}_\infty$ intersecte $\underline{R}$ en des points auxquels aucune bulle n'est attachée, et de manière transverse, ce qui est impossible pour une limite de courbes qui n'intersectaient pas $\underline{R}$.

\end{proof}

Ainsi, l'application $C\Phi_{\underline{E},\underline{F}}$ commute avec les différentielles des complexes, et induit un morphisme $\Phi_{\underline{E},\underline{F}} $ au niveau des groupes d'homologie, indépendant des structures presque complexes régulières $\underline{J}$, et invariant par isotopies Hamiltoniennes. La preuve de ces deux faits est un argument habituel, similaire à celui donné dans la partie \ref{homotopie}, et consiste à relier deux structures presque complexes par un chemin et à considérer un espace de modules paramétré de dimension un, se compactifiant en une variété à bord et fournissant une homotopie entre les applications définies au niveau des complexes de chaînes.

\paragraph{Fibration de Lefschetz associée à un twist de Dehn généralisé}

Une fibration de Lefschetz est munie de sa connexion symplectique canonique \cite[formule (2.1.5.)]{Seidel}, 
\[T^hE = (\mathrm{Ker}~D_e\pi)^{\omega}.\]
On peut alors définir la monodromie le long d'un lacet de la base ne passant pas par des valeurs critiques.

Comme remarqué par Arnold dans \cite{Arnold}, la monodromie d'une fibration de Lefschetz autour d'une valeur critique est un twists de Dehn généralisé. Réciproquement, si $\tau _S$ est un twist de Dehn modèle autour d'une sphère Lagrangienne $S\subset M$ (disjointe de l'hypersurface $R$), il existe une fibration de Lefschetz $E_S$, dite fibration standard associée à $\tau _S$, au-dessus du disque, avec un seul point critique au-dessus de $0$, dont la fibre au-dessus de $1$ est $M$, et la monodromie autour de $0$ correspond à ce twist, voir par exemple  \cite[Lemma 1.10, Prop. 1.11]{Seidel}. Si $M$ est monotone, $E_S$ l'est aussi dès que $S$ est de dimension plus grande que 2, d'après \cite[Prop. 4.9]{WWtriangle}. Nous renvoyons à  \cite[Lemma 1.10]{Seidel} pour la construction de cette fibration. 


Rappelons les deux définitions suivantes, tirées de \cite{Seidel} :

\begin{defi}Une structure presque complexe $J$ sur $E$ est dite horizontale si elle préserve la décomposition $TE = T^vE \oplus T^hE$.
\end{defi}

\begin{defi}Une fibration de Lefschetz matelassée est dite à courbure positive si pour tout vecteur tangent horizontal $v$, $\omega(v, J v) \geq 0$. 
\end{defi}

Elles garantissent la proposition suivante :

\begin{prop}\label{aireposit}Soit $(\underline{E},\underline{F})$ une fibration de Lefschetz matelassée à courbure positive, $\underline{J}$ une famille de structures presque complexes horizontales, et $\underline{u}$ une section $\underline{J}$-holomorphe. Alors $\underline{u}$ est d'aire positive : $\sum_i\int  u_i^* \omega_i \geq 0$.
\end{prop}

\begin{proof}
Soit $v+h \in T_x E_i = T_x E_i^v \oplus T_x E_i^h$ un vecteur tangent à l'espace total. $\omega_i (v+h, J_i (v+h)) \geq 0$, en effet c'est la somme des 4 termes suivants :

$\omega_i (v, J_i v) \geq 0$, car $\omega_i$ est symplectique en restriction aux fibres, et $J_i$ est compatible à $\omega_i$.

$\omega_i (h, J_i h) \geq 0$, car la fibration est de courbure positive.

$\omega_i (v, J_i h)  = \omega_i (h, J_i v) = 0$, car $J_i$ est horizontale, et par définition   $T_xE_i^h$ est l'orthogonal de $T_xE_i^v$ pour $\omega_i$.

Il s'en suit que la forme bilinéaire $\omega_i(., J_i .)$ est positive, ce qui entraîne le résultat annoncé.
\end{proof}

\begin{remark}Les fibrations de Lefschetz standard $E_S$ associées à des twists de Dehn modèle sont à courbure positive, d'après \cite[Lemma 1.12, (iii)]{Seidel}.
\end{remark}

\paragraph{Composition d'invariants relatifs}

Soient $\underline{\pi}_1 \colon \underline{E}_1 \rightarrow \underline{S}_1$ et $\underline{\pi}_2 \colon \underline{E}_2 \rightarrow \underline{S}_2$ des fibrations de Lefschetz  matelassées comme dans la définition \ref{quiltfiblefschetz}, de conditions aux bords et coutures respectives $\underline{F}_1$ et $\underline{F}_2$. Supposons qu'il existe une bijection entre les entrées $\mathcal{E}_{2,-}$ de $\underline{S}_2$ et les sorties $\mathcal{E}_{1,+}$ de $\underline{S}_1$ telle que $\underline{\pi}_1$ et $\underline{\pi}_2$ coïncident au niveau de chaque bouts, c'est-à-dire que le nombre de morceaux, les variétés symplectiques et les correspondances associées aux coutures correspondent.

Soit $\rho >0$ , on note $\underline{S}_1 \cup_\rho \underline{S}_2$ la surface matelassée obtenue en recollant les morceaux $ [0, \rho ] \times [k-1,k] \subset [0, + \infty ) \times [k-1,k]$ et $ [- \rho,0]  \times [k-1,k] \subset  (- \infty,0] \times [k-1,k]$ , et $\underline{E}_1 \cup_\rho \underline{E}_2$ la fibration matelassée recollée. 

La proposition suivante est l'analogue de  \cite[Theorem 4.18]{WWtriangle}, sa preuve est identique. 
\begin{prop}\label{compoinvrelat} Pour $\rho$ assez grand, il existe un ensemble $G_\delta$-dense de structures presque complexes produits pour lesquelles les espaces de sections pseudo-holomorphes d'indice 0 et 1 sont lisses et s'identifient aux produits fibrés :
 \begin{align*} \mathcal{M}(\underline{E}_1 \cup_\rho \underline{E}_2, \underline{F}_1 \cup_\rho \underline{F}_2)_0  & \simeq \mathcal{M}(\underline{E}_1, \underline{F}_1)_0 \times_{ev_1, ev_2} \mathcal{M}(\underline{E}_2, \underline{F}_2)_0, \\
\mathcal{M}(\underline{E}_1 \cup_\rho \underline{E}_2, \underline{F}_1 \cup_\rho \underline{F}_2)_1 & \simeq \mathcal{M}(\underline{E}_1, \underline{F}_1)_0 \times_{ev_1, ev_2} \mathcal{M}(\underline{E}_2, \underline{F}_2)_1 \\ &\cup \mathcal{M}(\underline{E}_1, \underline{F}_1)_1 \times_{ev_1, ev_2} \mathcal{M}(\underline{E}_2, \underline{F}_2)_0 
, \end{align*}
où $ev_i\colon \mathcal{M}(\underline{E}_i, \underline{F}_i)\rightarrow \I (\mathcal{E}_{1,+})$ est l'application qui à une section associe la limite en les bouts entrants (resp. sortants) pour $\underline{E}_2$ (resp.  $\underline{E}_1$).  

Il s'en suit que $C\Phi_{\underline{E}_1 } \circ C\Phi_{\underline{E}_2 }= C\Phi_{\underline{E}_1 \cup_\rho \underline{E}_2}$

\end{prop}
\arnaque

\subsection{Construction des flèches}\label{flechestri}

Afin de construire les deux morphismes de complexes 
\begin{align*}
C\Phi_1 & \colon CF (\tau_S L_0,S^T, S, \underline{L};\Lambda) \rightarrow CF (\tau_S L_0, \underline{L};\Lambda) \\
C\Phi_2 & \colon CF (\tau_S L_0, \underline{L};\Lambda) \rightarrow CF (L_0, \underline{L};\Lambda) 
\end{align*}
qui  approcheront les flèches de la suite exacte courte \ref{suitecourte} induite par les inclusions des points d'intersection de la proposition \ref{inter}, nous appliquons la construction précédente aux deux fibrations de Lefschetz matelassées décrites plus bas.

\paragraph{Définition de $C\Phi_1$ :} 

L'application $C\Phi_1$ est définie comme étant l'invariant relatif associé à la fibration de Lefschetz matelassée $(\underline{E}_1, \underline{F}_1) \rightarrow \underline{S}_1$ décrite dans la figure \ref{quiltedpants} : la surface matelassée $\underline{S}_1$ est composée de $k$  bandes  $[0,1]\times \rr$ cousues parallèlement, et d'une bande en pantalon cousue aux autres bandes le long d'un bord, voir figure \ref{quiltedpants}. La fibration $\underline{E}_1$ est triviale sur chaque morceau, les différentes fibres $M_0$, ... ,  $M_k$ sont spécifiées comme dans la figure. Les conditions Lagrangiennes $\underline{F}_1$ sont prises constantes dans les trivialisations et valent $\underline{L}$ entre les bandes parallèles, $S$ au niveau de la composante de bord joignant les deux entrées, et $\tau_S L_0$ sur le dernier bord du pantalon.

\begin{figure}[!h]
    \centering
    \def\svgwidth{\textwidth}
    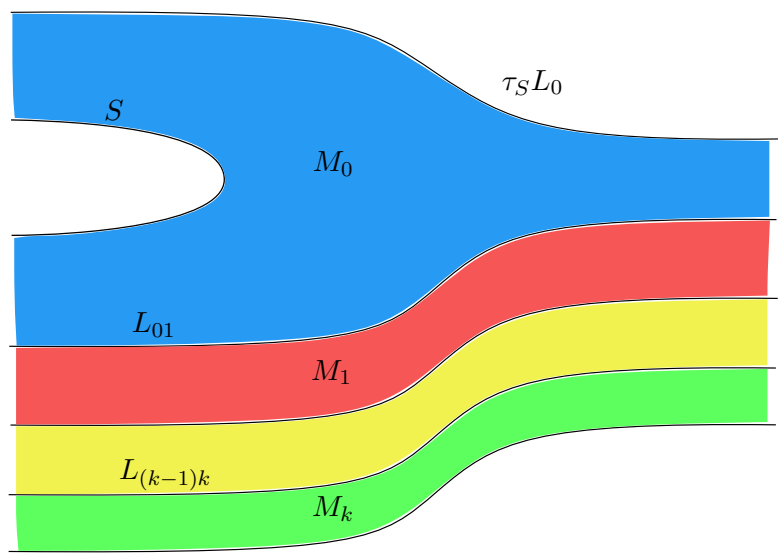
      \caption{Surface matelassée définissant $C\Phi_1$.}
      \label{quiltedpants}
\end{figure}

On note $\Phi_1\colon HF (\tau_S L_0,S^T, S, \underline{L};\Lambda) \rightarrow HF (\tau_S L_0, \underline{L};\Lambda) $ la flèche induite par  $C\Phi_1$ en homologie.

\paragraph{Définition de $C\Phi_2$ :}

L'application $C\Phi_2$ est définie comme étant l'invariant relatif associé à la fibration de Lefschetz matelassée $(\underline{E}_2, \underline{F}_2) \rightarrow \underline{S}_2$ décrite dans la figure \ref{quiltedlefschetz} : $\underline{S}_2$ est composée de $k + 1$ bandes cousues parallèlement. La restriction de $\underline{E}_2$ au dessus de la première bande est $E_S$, la fibration standard associée à $S$, et est triviale sur les autres bandes, les fibres sont résumées dans le dessin. Les conditions Lagrangiennes $\underline{F}_2$ sont prises constantes dans les trivialisations, valent $\underline{L}$ entre les bandes parallèles. En ce qui concerne le patch correspondant à $M_0$, conformément à Seidel, on a dessiné en pointillés un chemin reliant le point critique à un point du bord, et trivialisé la fibration sur le complémentaire de ce chemin. Ainsi, vues dans cette trivialisation, les conditions Lagrangiennes de part et d'autre du chemin diffèrent de la monodromie de cette fibration, c'est-à-dire du twist $\tau_S$.

\begin{figure}[!h]
    \centering
    \def\svgwidth{\textwidth}
    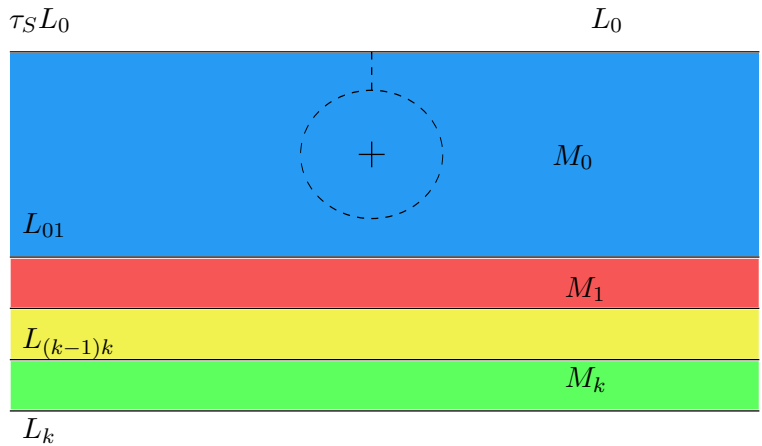
      \caption{Fibration de Lefschetz matelassée définissant $C\Phi_2$.}
      \label{quiltedlefschetz} 
\end{figure}

On note $\Phi_2\colon HF (\tau_S L_0, \underline{L};\Lambda) \rightarrow HF (L_0,\underline{L};\Lambda) $ la flèche induite en homologie.

\subsection{La composée est homotope à zéro}\label{homotopie}

D'après la proposition \ref{compoinvrelat}, la composée $C\Phi_{2} \circ C\Phi_{1}$ correspond à l'invariant relatif associé à la fibration recollée $\underline{S}_1 \cup_\rho \underline{S}_2$, pour un paramère de recollement $\rho$ suffisamment grand. En déformant la surface base, nous allons montrer que $C\Phi_{2} \circ C\Phi_{1}$ est homotope à une composée $C\Phi_{4} \circ C\Phi_{3}$ de deux invariants relatifs, puis nous verrons que le morphisme $C\Phi_{3}$ est homotope à 0.

\begin{figure}[!h]
    \centering
    \def\svgwidth{\textwidth}
    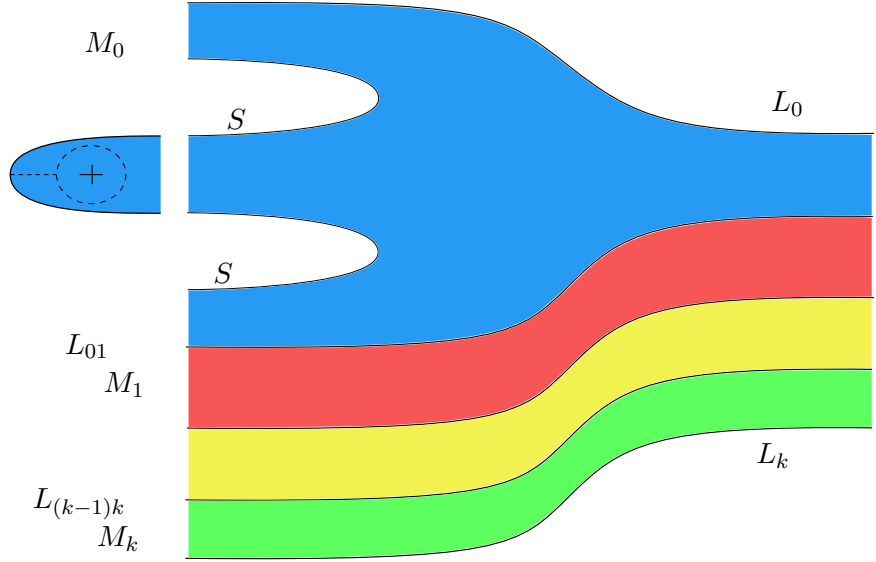
      \caption{Fibrations matelassées définissant $C\Phi_4$ et $C\Phi_3$.}
      \label{PHI_4_PHI_3}
\end{figure}

Soient $\underline{E}_3 \rightarrow \underline{S}_3$ et $\underline{E}_4 \rightarrow \underline{S}_4$ comme spécifiées dans la figure \ref{PHI_4_PHI_3}, et $\rho'$ un paramètre assez grand de telle sorte que, d'après la proposition \ref{compoinvrelat}, $C\Phi_{4} \circ C\Phi_{3} = C\Phi_{\underline{E}_3 \cup_{\rho'} \underline{E}_4 } $. 

Les fibrations $ \underline{E}_1 \cup_{\rho} \underline{E}_2$ et $ \underline{E}_3 \cup_{\rho'} \underline{E}_4$ sont difféomorphes. En tant que variétés lisses, notons $\underline{E}$ leur espace total commun, $\underline{S}$ leur base commune et $\underline{\pi}$ leur projection commune.

On va décrire une famille à un paramètre de structures presque complexes sur cette fibration, $(\underline{E}_t, \underline{S}_t)_{t\in [0,1]}$, qui interpolent entre $ \underline{E}_1 \cup_{\rho} \underline{E}_2$ et $ \underline{E}_3 \cup_{\rho'} \underline{E}_4$ :

Soit $(\underline{j}_t)_{t\in [0,1]}$ une famille à un paramètre de structures complexes sur $ \underline{S}_1 \cup_\rho \underline{S}_2 \simeq \underline{S}_3 \cup_{\rho'} \underline{S}_4 $ telles que $\underline{j}_0$ corresponde à la structure complexe de  $ \underline{S}_1 \cup_\rho \underline{S}_2$ et $\underline{j}_1$ corresponde à celle de $\underline{S}_3 \cup_{\rho'} \underline{S}_4$. 

Soit $(\underline{J}_t)_{0\leq t \leq 1}$ une famille de structures presque complexes sur l'espace total $\underline{E}_t$ telle que $\underline{J}_0$ corresponde à la structure presque complexe de $ \underline{E}_1 \cup_{\rho} \underline{E}_2$, $\underline{J}_1$ corresponde à la structure presque complexe de $ \underline{E}_3 \cup_{\rho'} \underline{E}_4$, et telle que pour tout $t$, la projection $\pi$ soit $(\underline{J}_t,\underline{j}_t)$-holomorphe.

Le raisonnement standard suivant, voir par exemple \cite[Theorem 3.1.6]{McDuSal}, permet de montrer qu'une telle déformation générique induit une homotopie entre les applications $C\Phi_{2} \circ C\Phi_{1}$ et $C\Phi_{4} \circ C\Phi_{3}$. Il s'agit de considérer les espaces des modules paramétrés suivants : pour $k = -1$ ou $0$, soit $\mathcal{M}^k_{param} = \bigcup_{t}{ \lbrace t\rbrace \times \mathcal{M}_t^k  }$, où $\mathcal{M}_t^k $ désigne la réunion sur tous les $\underline{x}\in \I(L_0,S,S^T, \underline{L})$, $ \underline{y}\in \I(L_0,\underline{L})$, des espaces de modules  $\mathcal{M}_t(\underline{x},\underline{y})_k $ de sections pseudo-holomorphes de $\underline{E}_t$ d'indice $k$,  ayant pour limites $\underline{x}$ et $\underline{y}$ aux bouts. Cet espace correspond au lieu d'annulation d'une section d'un fibré de Banach, dont le linéarisé au voisinage d'une solution est un opérateur Fredholm : l'opérateur de Cauchy-Riemann linéarisé paramétré, voir \cite[Def. 3.1.6]{McDuSal}). Pour un choix générique des familles $\underline{j}_t$ et $\underline{J}_t$, il est surjectif. Dans ces conditions,  $\mathcal{M}^k_{param}$ est une variété à bord de dimension $k+1$.

Ainsi $\mathcal{M}^{-1}_{param}$, de dimension zéro, permet de définir une application \[h\colon CF(L_0,S,S^T, \underline{L};\Lambda) \rightarrow CF(L_0,\underline{L};\Lambda)\] 
par :
\[ h ( \underline{x} ) = \sum_{\underline{y}}{\sum_{\underline{u}\in \mathcal{M}^{-1}_{param}(\underline{x},\underline{y})}{o(\underline{u}) q^{A(\underline{u})}} \underline{y}}, \]
et $\mathcal{M}^{0}_{param}$, de dimension 1, permet de montrer que $h$ est une homotopie. En effet il se compactifie en une variété à bord compacte, dont le bord s'identifie à la réunion :

\[\mathcal{M}_0 \sqcup \mathcal{M}_1 \bigsqcup_{\underline{x}'}{ \widetilde{\mathcal{M}}(\underline{x},\underline{x}') \times \mathcal{M}_{par}^0(\underline{x}',\underline{y}) } \bigsqcup_{\underline{y}'}{\mathcal{M}_{par}^0(\underline{x},\underline{y}') \times \widetilde{\mathcal{M}}(\underline{y}',\underline{y})}, \]
où $\widetilde{\mathcal{M}}$ désigne l'espace des modules des bandes pseudoholomorphes intervenant dans les différentielles,   et $\underline{x}'$, $\underline{y}'$ parcourent l'ensemble des générateurs des complexes de départ et d'arrivée.

\begin{proof}Il ne se produit pas de bubbling sur les Lagrangiennes, ni sur les hypersurfaces, pour les mêmes raisons que dans la preuve du lemme. \ref{nobubbling}.
\end{proof}

Il s'en suit que $ C\Phi_{\underline{E}_1 \cup_{\rho} \underline{E}_2}+ C\Phi_{\underline{E}_3 \cup_{\rho'} \underline{E}_4} +  \partial h + h \partial = 0$, ce qui prouve que $C\Phi_{2} \circ C\Phi_{1}$ et $C\Phi_{4} \circ C\Phi_{3}$ sont homotopes.

Il reste à voir que $C\Phi_{3}$ est homotope à $0$. Cela résulte de \cite[Cor. 4.23]{WWtriangle} : d'une part, pour $r>0$ assez petit, la fibration standard sur le disque de rayon $r$ n'admet pas de sections pseudo-holomorphes d'indice nul, car il existe une famille de sections d'indice $c-1$, avec $c$ la dimension de la sphère $S$, qui est strictement plus grande que 2, et dont l'aire tend vers $0$ lorsque $r\rightarrow 0$. Par monotonie, toute autre section d'indice plus petit est d'aire négative et ne peut être pseudo-holomorphe, la fibration étant de courbure positive. Ainsi les sections au-dessus du disque de rayon fixées sont cobordantes à l'ensemble vide, un cobordisme est donné par un espace des modules paramétré  $\bigcup_{r\in [r_0, 1]}{ \lbrace r\rbrace \times \mathcal{M}_r  }$ réunion des espaces des modules correspondants aux sections d'indice nul de la fibration standard au-dessus du disque de rayon $r$, et $r_0$ est assez petit de sorte à avoir $\mathcal{M}_{r_0} = \emptyset$. 

\subsection{Contributions de basse énergie}\label{secbasseénergie}
L'objectif de ce paragraphe, proposition \ref{basseénergie}, est de décrire les parties de bas degré (en $q$) des applications $C\Phi_1$ et $C\Phi_2$ lorsque le twist de Dehn est "suffisamment fin". Un énoncé analogue dans le cadre de Wehrheim et Woodward est \cite[Theorem 5.5]{WWtriangle}. Dans notre cadre, nous démontrons la proposition \ref{basseénergie} en suivant la preuve originale de Seidel (\cite[Parag. 3.2-3.3]{Seidel}).

\paragraph{Préliminaires}
Alors que la preuve de Wehrheim et Woodward est basée sur des arguments de nature analytique comme l'inégalité de la valeur moyenne, la preuve de Seidel dans le cas exact repose sur des calculs d'aire à priori, faisant intervenir des fonctionnelles d'action $a_{L_0,L_1}$ associées à des paires de Lagrangiennes $(L_0,L_1)$. En homologie de Floer matelassée, l'analogue de ce type  de fonctionnelles est l'action matelassée, \cite[Parag. 5.1]{WWqfc}. Rappelons sa définition :

\begin{defi}[Action matelassée]

Soit $\widetilde{\underline{L}}\colon pt\rightarrow M_0 \rightarrow \cdots \rightarrow pt$ une correspondance Lagrangienne généralisée  vérifiant les hypothèses de la définition \ref{defcat}.

\begin{itemize}
\item[$(i)$] On note $\mathcal{P}(\widetilde{\underline{L}}) = \lbrace \underline{\alpha} = (\alpha_i \colon [0,1] \rightarrow M_i\setminus R_i)_i\ |\ (\alpha_i(1),\alpha_{i+1}(0))\in L_{i,i+1}  \rbrace $. Les points d'intersection $\I(\widetilde{\underline{L}})$ s'identifient notamment aux lacets constants.
Notons que $\mathcal{P}(\widetilde{\underline{L}})$ est connexe par arcs car les correspondances Lagrangiennes le sont, et les variétés $M_i$ sont simplement connexes.

\item[$(ii)$] L'\emph{action symplectique} est la fonctionnelle $a_{\widetilde{\underline{L}}}\colon \mathcal{P}(\widetilde{\underline{L}})\rightarrow \rr/M\zz$, où $M = \kappa N $\footnote{N est le nombre de Maslov minimal, $\kappa = \frac{1}{4}$ est la constante de monotonie} est l'aire minimale d'une sphère pseudo-holomorphe, définie comme suit. 

Fixons un lacet base $\underline{\alpha}^{bas}$ dans $\mathcal{P}(\widetilde{\underline{L}})$. Si $\underline{\alpha}\in \mathcal{P}(\widetilde{\underline{L}})$, choisissons un chemin $\underline{\alpha}_t$ reliant $\underline{\alpha}^{bas}$ et $\underline{\alpha}$ dans $\mathcal{P}(\widetilde{\underline{L}})$, qui peut être vu comme une surface matelassée 

\[ \widetilde{\underline{\alpha}} = (\widetilde{\alpha_i}\colon [0,1]\times[0,1]\rightarrow M_i\setminus R_i). \]

Posons alors \[a_{\widetilde{\underline{L}}}(\underline{\alpha}) = \sum_{i}{\int_{[0,1]^2}{\widetilde{\alpha_i}^{*}\tilde{\omega}_i}}, \] où $\tilde{\omega}_i$ désigne la forme monotone de $M_i$.

\end{itemize}
\end{defi}

En vertu de la monotonie de $M_i\setminus R_i$ et de la simple connexité des $L_{i,i+1}$, cette quantité est bien définie modulo $M\zz$. La fonctionnelle d'action est donc bien définie, à une constante près, dépendant du choix du lacet base.

Ainsi, si $\underline{u}$ est une bande matelassée reliant $\underline{x},\underline{y}\in \I(\widetilde{\underline{L}})$, son aire symplectique modulo $M$ est donnée par la différence d'action :
\[ A(\underline{u}) =  a_{\widetilde{\underline{L}}}(\underline{y}) - a_{\widetilde{\underline{L}}}(\underline{x}).\]

Dans notre cadre, définissons $a_{L_0,\underline{L}}$, $a_{S,\underline{L}}$ et $a_{L_0,S}$ de manière à ce que, si $\tilde{x}_0\in L_0 \cap S$, $\underline{x}\in \I(S,\underline{L})$ et $\underline{y}\in \I(L_0, \underline{L})$, la quantité \[ \chi (\tilde{x}_0, \underline{x}, \underline{y}) = a_{L_0,\underline{L}}(\underline{y}) -  a_{L_0,S}(\tilde{x}_0) - a_{S,\underline{L}}(\underline{x}) \] coïncide modulo $M$ avec l'aire d'un triangle matelassé dont les conditions aux coutures sont spécifiées dans le dessin \ref{trianglechi}. Ceci est vérifié dans le cas suivant : en choisissant un lacet base pour $a_{L_0,S}$, puis un autre pour $a_{S,\underline{L}}$ dont le point de départ coïncide avec le point d'arrivée du précédent, puis en prenant le lacet concaténé comme lacet base pour $a_{L_0,\underline{L}}$.

\begin{figure}[!h]
    \centering
    \def\svgwidth{\textwidth}
    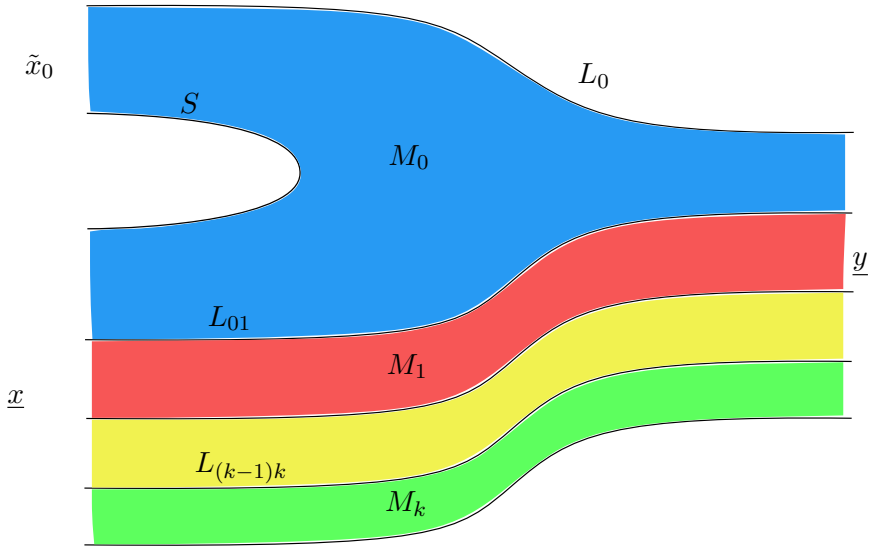
      \caption{triangle d'aire $ \chi (\tilde{x}_0, \underline{x}, \underline{y})$ modulo $M$.}
      \label{trianglechi}
\end{figure}

Définissons a présent $a_{\tau_S L_0,S}$ et $a_{\tau_S L_0,\underline{L}}$ de sorte qu'elles coïncident avec $a_{ L_0,S}$ et $a_{ L_0,\underline{L}}$ pour des lacets dont la composante de $M_0$ est en dehors de $\iota(T(\lambda))$. De cette manière, si $\tilde{x}_0\in \tau_S L_0 \cap S$, $\underline{x}\in \I(S,\underline{L})$ et $\underline{y}\in \I(\tau_S L_0,\underline{L})$, la quantité

\[ \chi_{\tau_S} (\tilde{x}_0, \underline{x}, \underline{y}) = a_{\tau_S L_0,\underline{L}}(\underline{y}) -  a_{\tau_S L_0,S}(\tilde{x}_0) - a_{S,\underline{L}}(\underline{x}) \] 
représente l'aire d'un triangle matelassé comme dans la figure \ref{quiltedpants} définissant l'application $C\Phi_1$. C'est cette quantité que l'on cherche à exprimer en fonction de la fonction $R$ et des données avant le twist.

\begin{prop}\label{airetriangles} On suppose :
\begin{itemize}
\item[$(i)$] Que les hypothèses de la proposition \ref{inter} sont vérifiées, de manière à avoir : 

\[ \I(\tau_S L_0,\underline{L}) = i_2(I( L_0,\underline{L})) \cup i_1( (L_0 \cap S)\times \I(S,\underline{L})).\]

\item[$(ii)$] Que $L_{01}$ est un produit au voisinage de chaque point d'intersection de $\I(S,\underline{L})$, soit : 

\[\forall \underline{x}\in \I(S,\underline{L}), \exists U_0, U_1 \colon U_0\times U_1 \cap L_{01} = T(\lambda)_{x_0} \times L_1(\underline{x}),\]
avec $\underline{x} = (x_0, x_1, \cdots )$, $U_0$ (resp. $U_1$)  un voisinage de $x_0$ dans $M_0$ (resp. de $x_1$ dans $M_1$), et $L_1(\underline{x})\subset U_1$ une Lagrangienne (dépendant de $\underline{x}$).
\end{itemize}
Alors, 
\begin{enumerate}

\item Si $\tilde{x}_0\in \tau_S L_0 \cap S$, $\underline{x}\in \I(S,\underline{L})$ et $y_0$ est la coordonnée en $M_0$ de $i_1(\tilde{x}_0, \underline{x})$, 

\[\chi_{\tau_S} (\tilde{x}_0, \underline{x}, i_2(\tilde{x}_0, \underline{x}) ) = K_{\tau_S}(y_0) - 2\pi R(0) \ \ (\text{mod }M),\]
avec $K_{\tau_S}(y_0) = 2\pi ( R'(\mu(y_0)) \mu(y_0) - R(\mu(y_0)))$ la fonction associée au twist comme dans \cite{Seidel}, et comme dans la sections \ref{rappeltwists}, $\mu$ et $R$ désignant respectivement la norme d'un covecteur et la fonction utilisée pour définir le twist (primitive de la fonction angle). 

De plus, $K_{\tau_S}(y_0) - 2\pi R(0)$ est exactement l'aire d'un triangle d'indice nul.

\item Si $\tilde{x}_0\in \tau_S L_0 \cap S$, $\underline{x}\in \I(S,\underline{L})$ et $\underline{y}\in \I(L_0,\underline{L}) = i_2(\I(L_0,\underline{L}))$,

\[ \chi_{\tau_S} (\tilde{x}_0, \underline{x}, \underline{y}) = \chi (\mathbb{A}(\tilde{x}_0), \underline{x}, \underline{y}) - 2\pi R(0) \ \ (\text{mod }M).\]

\item Si $\tilde{x}_0\in \tau_S L_0 \cap S$, $\underline{x}\in \I(S,\underline{L})$ et $\underline{y} = i_2(\tilde{z}_0, \underline{z})\in  i_2(\I(L_0,\underline{L}))$,

\begin{align*}
\chi_{\tau_S} (\tilde{x}_0, \underline{x}, \underline{y}) &= \chi_{\tau_S} (\tilde{z}_0, \underline{z}, i_2(\tilde{z}_0, \underline{z})) + a_{L_0,S}(\mathbb{A}(\tilde{z}_0)) + a_{S,\underline{L}}(\underline{z}) \\ 
 &- ( a_{L_0,S}(\mathbb{A}(\tilde{x}_0)) + a_{S,\underline{L}}(\underline{x}) ) \ \ (\text{mod }M).
\end{align*}

\end{enumerate}

\end{prop}

\begin{proof} Etant donné que $L_{01}$ est un produit au voisinage des points de $\I(S,\underline{L})$, la partie dans $M_1, M_2, \cdots, M_k$ des triangles matelassés intervenant dans le calcul de $\chi_{\tau_S}$ est la même que celle intervenant dans $\chi$ : seule la partie dans $M_0$ change d'aire, et le calcul se ramène à celui de Seidel, voir formule (3.7) dans la preuve de \cite[Lemma 3.2]{Seidel}.
\end{proof}

\begin{remark}Ces formules sont illustrées dans la figure \ref{triangle} :  la quantité $\chi_{\tau_S} (\tilde{x}_0, \underline{x}, \underline{y})$ représente l'aire d'un triangle matelassé comme dans la figure \ref{quiltedpants}. C'est la somme de l'aire d'un polygone  indépendant du twist (le polygone vide pour le triangle mauve, un triangle pour le triangle vert, et un rectangle pour le triangle jaune) et d'une petite quantité dépendante de la primitive $R$ de la fonction angle du twist.
\end{remark}

\begin{figure}[!h]
    \centering
    \def\svgwidth{\textwidth}
    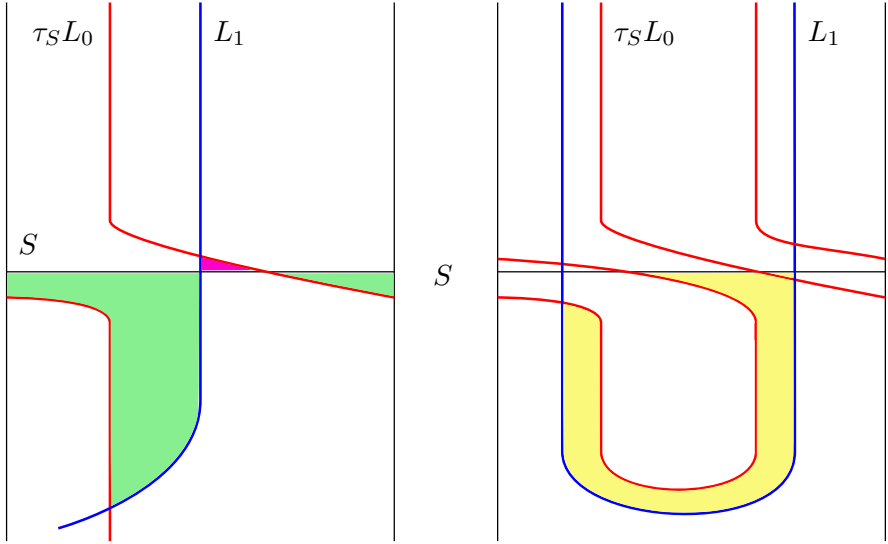
      \caption{Trois triangles dont l'aire est donnée par $\chi_{\tau_S}$.}
      \label{triangle}
\end{figure}

\begin{prop}[Contributions de basse énergie] \label{basseénergie}
Soit $\epsilon >0$  assez petit, on suppose :
\begin{itemize}

\item[$(i)$] que les hypothèses de la proposition \ref{airetriangles} sont vérifiées,

\item[$(ii)$] \begin{itemize} 

\item[$(a)$] $\forall \underline{x} \neq \underline{y} \in \I(L_0,\underline{L}),\ a_{L_0,\underline{L}}(\underline{x}) - a_{L_0,\underline{L}}(\underline{y})\notin (-3\epsilon,3\epsilon),$ 

\item[$(b)$]$\forall (\tilde{x}_0, \underline{x}) \neq (\tilde{z}_0, \underline{z}) \in (L_0 \cap S) \times \I(S,\underline{L}),$
\[a_{L_0,S}(\tilde{z}_0) + a_{S,\underline{L}}(\underline{z}) - ( a_{L_0,S}(\tilde{x}_0) + a_{S,\underline{L}}(\underline{x}) )\notin (-3\epsilon,3\epsilon),\]
 
\item[$(c)$]$\forall \tilde{x}_0 \in L_0 \cap S,  \underline{x} \in  \I(S,\underline{L}), \underline{y} \in  \I(L_0,\underline{L})$,  \[  a_{L_0,\underline{L}}(\underline{y}) - ( a_{L_0,S}(\tilde{x}_0) + a_{S,\underline{L}}(\underline{x}) )\notin (-5\epsilon,5\epsilon).\] 

\end{itemize}

\item[$(iii)$] $0 \geq 2\pi R(0) > -\epsilon,$ et $\tau_S$ est concave, au sens de la définition \ref{concave}.

\end{itemize}

Alors, sous ces hypothèses,
\begin{itemize}

\item[$(a)$] $C\Phi_1 = C\Phi_{1,\leq \epsilon} + C\Phi_{1,\geq 2\epsilon} $, avec :
\begin{itemize}

\item[$(i)$] $C\Phi_{1,\leq \epsilon}(\underline{x}) =\pm q^{A(\underline{x})}i_{1}(\underline{x})$, où $A(\underline{x})$ est un nombre vérifiant $ 0 \leq A(\underline{x}) \leq \epsilon ,$

\item[$(ii)$] $C\Phi_{1,\geq 2\epsilon}$ est d'ordre supérieur à $2\epsilon$.

\end{itemize}

\item[$(b)$] $C\Phi_2 = C\Phi_{2,\leq \epsilon} + C\Phi_{2,\geq 2\epsilon} $, avec :
\begin{itemize}
\item[$(i)$] $C\Phi_{2,\leq \epsilon}(i_1(\underline{x})) = 0 $ et $\tilde{C\Phi_2} (i_2(\underline{x})) =  \pm \underline{x} $, 

\item[$(ii)$] $C\Phi_{2,\geq 2\epsilon} $ est d'ordre supérieur à $2 \epsilon$.

\end{itemize}
\item[$(c)$] L'homotopie $h$, ainsi que les trois différentielles sont d'ordre supérieur à $2\epsilon$.
\end{itemize}

\end{prop}

\begin{remark}\label{justif}
Il est toujours possible de se ramener aux hypothèses de la proposition \ref{basseénergie}. En effet, il suffit dans un premier temps de perturber $L_0$, $L_{01}$ et $S$ par des isotopies Hamiltoniennes et prendre $\epsilon$ assez petit afin de garantir les inégalités, puis reperturber légèrement et éventuellement diminuer $\lambda$ pour garantir les hypothèses de la proposition \ref{inter}.

\end{remark}

\begin{proof}
(a) Notons $\mathcal{M}(\tilde{x}_0, \underline{x},\underline{y})_0$ l'espace des modules des triangles matelassés d'indice nul comme dans la figure \ref{quiltedpants}, ayant pour limites $\tilde{x}_0$, $\underline{x}$ et $\underline{y}$ au niveau des bouts. Supposons $\underline{y} = i_1(\tilde{x}_0, \underline{x})$, et  $ \underline{u} \in \mathcal{M}(\tilde{x}_0, \underline{x},\underline{y})_0$. On a 

\[ \langle C\Phi_1 (\tilde{x}_0, \underline{x}),\underline{y} \rangle = \# \mathcal{M}(\tilde{x}_0, \underline{x},\underline{y}) q^{A(\underline{u})}.\]

D'une part, $A(\underline{u}) = \chi_{\tau_S} (\tilde{x}_0, \underline{x}, i_2(\tilde{x}_0, \underline{x}) ) = K_{\tau_S}(y_0) - 2\pi R(0) \in [0, \epsilon)$, en vertu de la proposition \ref{airetriangles}.

D'autre part, $\# \mathcal{M}(\tilde{x}_0, \underline{x},\underline{y})_0 = \pm 1$. Cela peut être prouvé par un argument de cobordisme similaire à celui intervenant dans la preuve du fait que $C\Phi_{2} \circ C\Phi_{1}$ et $C\Phi_{4} \circ C\Phi_{3}$ sont homotopes. Nous rappelons brièvement cet argument, et renvoyons à \cite[Prop. 3.4]{Seidel} pour plus de détails. On considère une famille  $(f_t\colon S^n\rightarrow S)_{t\in [0,1]}$ de paramétrages de $S$ tels que $f_0$ coïncide avec le plongement $\iota\colon T(0)\rightarrow M$ et $f_1$ envoie l'antipode $\mathbb{A}(x_0)$ sur le point  $\tilde{x}_0$. Cette famille permet de définir un espace des modules paramétré $(\mathcal{M}_t)_{t\in [0,1]}$, où $\mathcal{M}_0 = \mathcal{M}(\tilde{x}_0, \underline{x},\underline{y})_0$ et $\mathcal{M}_t$ désigne l'espace des modules correspondant au twist modèle induit par $f_t$. Pour une famille $f_t$ générique, cet espace est un cobordisme de dimension 1 entre $\mathcal{M}_0$ et $\mathcal{M}_1$. De plus ce cobordisme est compact, d'après le lemme \ref{nobubbling} et parce qu'il n'y a pas assez d'énergie pour former des brisures d'aires strictement positive, d'après les hypothèses. Par ailleurs $\mathcal{M}_1$ consiste en un seul point, le triangle constant.

Supposons à présent $\underline{y} \neq i_1(\tilde{x}_0, \underline{x})$. D'après les hypothèses et la proposition \ref{airetriangles}, 

\[ \chi_{\tau_S} (\tilde{x}_0, \underline{x},  \underline{y} ) \notin [0, \epsilon) \subset \rr/M\zz . \]

Ainsi, l'aire de tout triangle pseudo-holomorphe, nécessairement positive, est plus grande que $\epsilon$.

(b) Pour $C\Phi_2$, le raisonnement de Seidel s'applique tel quel, voir \cite[Section 3.3]{Seidel}. En voici brièvement l'idée : si $\underline{x}\in \I(\tau_S L_0,\underline{L})$ et  $\underline{y}\in \I( L_0,\underline{L})$, quitte à prendre une structure presque complexe horizontale, ce qui est possible d'après \cite[Lemma 2.9]{Seidel}, la fibration étant de courbure strictement positive, les seules sections J-holomorphes d'aire nulle sont des sections constantes, c'est-à-dire des points d'intersection de $\I(L_0, S^T,S, \underline{L} )$. Il y en a un lorsque $\underline{y} = i_2(\underline{x})$, et zéro si $\underline{y}\neq i_2(\underline{x})$. Toutes les autres sections ont une aire strictement positive, d'après la proposition \ref{aireposit},  car la fibration est de courbure positive, et cette aire est donnée par $a_{L_0,\underline{L}}(\underline{y}) - a_{\tau_S L_0,\underline{L}}(\underline{x}) $. Si l'on note $\underline{\tilde{x}}$ le point dont la première coordonnée est l'image de celle de $\underline{x}$ par l'application antipodale, cette quantité vaut, d'après la formule (3.2) de \cite{Seidel} :
\[ a_{L_0,\underline{L}}(\underline{y}) - a_{ L_0,\underline{L}}(\underline{\tilde{x}}) - 2\pi R(0) + 2\pi \int_0^{||y||}{(R'(||y||) - R'(t)    )dt},\]
où $y = \iota^{-1}(\tilde{x}_0)\in T(\lambda)$ . D'après les hypothèses $(ii)$ $(a)$ et $(iii)$, cette quantité est plus grande que $2\epsilon$.

(c) De même, un calcul d'action permet de conclure pour l'ordre de l'homotopie : soient $(\tilde{x}_0 ,\underline{x})\in \I(\tau_S L_0,S,S,\underline{L})$ et $\underline{y}\in \I(L_0,\underline{L})$, l'aire d'une section contribuant au coefficient $(h(\tilde{x}_0 ,\underline{x}), \underline{y})$ vaut :
\begin{align*}
 & a_{S,\underline{L}}(\underline{y}) + a_{\tau_S L_0,S}(\tilde{y}_0) - a_{S,\underline{L}}(\underline{x}) - a_{\tau_S L_0,S}(\tilde{x}_0)  \\ 
 & = a_{ L_0,S}({y}_0)- a_{ L_0,S}(\hat{x}_0) + a_{S,\underline{L}}(\underline{y})- a_{S,\underline{L}}(\underline{x}),  \end{align*}
en notant $\tilde{y}_0$ l'antipode de la première coordonnée $y_0$ de $\underline{y}$, et $\hat{x}_0$ l'antipode de $\tilde{x}_0$. Cette quantité est supérieure à $2\epsilon$ d'après $(ii)(b)$.
Enfin, les trois différentielles sont d'ordre $\geq 2\epsilon$ d'après   $(ii)(b)$ pour celle de $CF(\tau_S L_0,S,S,\underline{L})$, et $(ii)(a)$ pour les celles de $CF(\tau_S L_0,\underline{L})$ et $CF( L_0,\underline{L})$.
\end{proof}

\subsection{Preuve du triangle}\label{preuvedutriangle}
On est à présent en mesure de démontrer le théorème \ref{quilttri} en suivant le même raisonnement que dans \cite[Parag. 5.2.3]{WWtriangle}. Supposons  à présent que les hypothèses de la proposition \ref{basseénergie} sont vérifiées. Introduisons les notations suivantes : désignons les trois complexes à coefficients dans $\zz$ par :
\begin{align*}
 A_0 = & CF (\tau_S L_0,S^T, S, \underline{L};\zz) \\ 
 A_1 = & CF (\tau_S L_0, \underline{L};\zz) \\
 A_2 = & CF (L_0, \underline{L};\zz ),
\end{align*}
et notons $C_i$ les complexes à coefficients dans $\Lambda$, $C_i = A_i \otimes_\zz \Lambda$ en tant que modules, et munis de leurs différentielles respectives $\partial_0$, $\partial_1$ et $\partial_2$. Les applications $C\Phi_1$, $C\Phi_2$ et l'homotopie $h$ entre $C\Phi_2 \circ C\Phi_1$ et l'application nulle construite dans le paragraphe \ref{homotopie} sont spécifiées dans le diagramme suivant :

\begin{equation*} 
\xymatrix{
 {C_0} \ar[r]^{C\Phi_1} \ar@/_1.5pc/[rr]^{ h }& {C_1} \ar[r]^{C\Phi_2} & {C_2.}
}
\end{equation*}

Le mapping cone $\mathrm{Cone}~C\Phi_1$ est le complexe $C_0\oplus C_1$ dont la différentielle est donnée matriciellement par :\[ \partial_{\mathrm{Cone}~C\Phi_1} =
\left[ \begin{matrix}
 - \partial_0 & 0  \\
 - C\Phi_1 & \partial_1  
 
\end{matrix}\right] .  \] 
D'après le lemme du serpent, la suite exacte courte au niveau des complexes induit une suite exacte longue :
\[ \cdots \rightarrow H_*(C_0)\rightarrow H_*(C_1)\rightarrow H_*(\mathrm{Cone}~C\Phi_1 ) \rightarrow \cdots ,\]
où la première flèche est $\Phi_1$. Il suffit donc de montrer que le morphisme de complexes \[(h, - C\Phi_2)\colon \mathrm{Cone}~C\Phi_1 \rightarrow C_2\]  induit un isomorphisme en homologie. Ceci est le cas si et seulement si son cone est acyclique. En tant que $\Lambda$-module, $\mathrm{Cone}~(h, - C\Phi_2) = C_0\oplus C_1\oplus C_2 $. Dans cette décomposition, sa différentielle a pour expression  : 
\[  \partial = 
\left[ \begin{matrix}
\partial_0 & 0 & 0 \\
 C\Phi_1 & - \partial_1 & 0 \\
 - h & C\Phi_2 & \partial_2 
 
\end{matrix}\right]. \]
Le lemme \ref{lemmeperutz} suivant  permet de démontrer l'acyclicité d'un complexe de chaînes sur $\Lambda$ en ne connaissant que le terme dominant de sa différentielle. La conclusion ne vaut que dans la complétion $q$-adique de $\Lambda $, i.e. l'anneau de Novikov universel :

\[\hat{\Lambda} = \left\lbrace \sum_{k = 0}^{\infty}{ a_k q^{\lambda_k}} : a_k\in \zz,\, \lambda_k \in \rr,\, \lim_{k\to +\infty} \lambda_k = +\infty \right\rbrace\] 

Commençons par rappeler le vocabulaire des modules $\rr$-gradués.
\begin{defi} Un \emph{module $\rr$-gradué} $A$ est un module muni d'une décomposition $A = \bigoplus_{r\in\rr }{A_r}$. Son \emph{support} est défini par $Supp A =  \lbrace r: A_r\neq 0 \rbrace$. Soit $I\subset \rr$, $A$ est dit \emph{$I$-lacunaire} si $\forall r,r' \in Supp A, r-r' \notin I$.

Si $r'\in \rr$, on note $A[r]$ le \emph{shift} défini par $A[r]_s = A_{r+s}$. On a $Supp A[r] = Supp A -r$.

Une application linéaire $f\colon A\to B$ entre deux modules gradués est dite 
\begin{itemize}
\item \emph{d'ordre $I$} si pour tout $r$, $f(A_r) \subset \bigoplus_{i\in I}{B_{r+i}}$.

\item \emph{$I$-lacunaire} si pour tout $r$, l'image $f(A_r)$ est $I$-lacunaire.
\end{itemize}
\end{defi}

\begin{lemma}( \cite[Lemma 5.3]{Perutzgysin})\label{lemmeperutz}
Soient $\epsilon >0$, $(A,d)$ un module $\rr$-gradué $[\epsilon, 2\epsilon)$-lacunaire, équipé d'une différentielle $d$ d'ordre $[0, \epsilon)$. Soit $D = A\otimes_\zz \Lambda$, $\hat{D}= A\otimes_\zz \hat{\Lambda}$ son complété, et $\partial$ une différentielle sur $\hat{D}$ telle que :
\begin{itemize}
\item[$(i)$] $\partial$ est $\hat{D}$-linéaire et continue.
\item[$(ii)$]  $\partial(A) \subset A\otimes_\zz \hat{\Lambda}_+ $, où $\hat{\Lambda}_+ = \left\lbrace \sum_{k = 0}^{\infty}{ a_k q^{\lambda_k}} \in \hat{\Lambda} :  \lambda_k \in \rr_+ \right\rbrace $.
\item[$(iii)$] $\partial = \partial_{\leq \epsilon} + \partial_{\geq 2 \epsilon}$, où $\partial_{\leq \epsilon} = d\otimes \hat{\Lambda} $ est la différentielle induite par $d$, et $\partial_{\geq 2 \epsilon}$  d'ordre $[2\epsilon, +\infty)$.
\item[$(iv)$] $(A,d)$ est acyclique.
\end{itemize}
Alors,  $(\hat{D},\partial)$ est acyclique.
\end{lemma}

Afin d'appliquer ce lemme au double mapping cone, équipons $A= A_0 \oplus A_1 \oplus A_2$ de la graduation suivante\footnote{Ici appara\^it la différence avec la preuve de Seidel, qui se place dans le cadre exact : il n'est pas possible de graduer le complexe par l'action, qui n'est définie que modulo $M$. La conclusion est donc à priori plus faible : l'acyclicité n'est donc que dans la complétion. Nous verrons néanmoins que les hypothèses de monotonie permettent  de récupérer l'acyclicité dans $\zz$.} :
\begin{itemize}
\item $A_0$ est concentré en degré 0.
\item $A_1$, qui est isomorphe à $A_0\oplus A_2$, est gradué comme suit : sa première composante est graduée de sorte que $C\Phi_{2,\leq \epsilon}$ soit de degré 0 : si $(\tilde{x}_0, \underline{x})\in \I(L_0, \underline{L})$, on pose $\deg i_2(\tilde{x}_0, \underline{x}) = \chi_{\tau_S} (\tilde{x}_0, \underline{x}, i_2(\tilde{x}_0, \underline{x}) )$. La seconde composante est concentrée en degré 0.
\item $A_2$ est concentré en degré 0.
\end{itemize}

D'après la propositon \ref{basseénergie}, $\supp A\subset [0, \epsilon)$. De plus, par construction, la différentielle
\[ d =\begin{pmatrix}
 0 & 0 & 0 \\
 C\Phi_{1,\leq \epsilon} & 0 & 0 \\
 0 & C\Phi_{2,\leq \epsilon} & 0 \end{pmatrix},\]
respecte cette graduation.

Le module $\hat{D}$ s'identifie alors à $C = \mathrm{Cone }(h,-C\Phi_2)$. Sa différentielle $\partial$ définie plus haut vérifie les hypothèses $(i)$, $(ii)$ et $(iii)$ du lemme \ref{lemmeperutz}, d'après la proposition \ref{basseénergie}. De plus, $d$ est acyclique. En effet, en décomposant $A_1 = A_0 \oplus A_2$, toujours d'après la proposition \ref{basseénergie},

\[ C\Phi_{1,\leq \epsilon} =  \begin{pmatrix}
  D_0 \\0
  \end{pmatrix} \text{ et } C\Phi_{2,\leq \epsilon} = \begin{pmatrix}
  0 & D_1  
  \end{pmatrix},
\]
avec $D_0$ et $D_1$ deux matrices diagonales pour lesquelles les coefficients diagonaux valent $\pm 1$. Il vient donc $d = \begin{pmatrix}
0 & 0 & 0 & 0 \\
D_0 & 0 & 0 & 0 \\
0 & 0 & 0 & 0 \\
0 & 0 & D_1& 0 \end{pmatrix}$, qui est clairement acyclique.

Expliquons enfin pourquoi la monotonie des correspondances Lagrangiennes permet d'obtenir l'acyclicité sur $\zz$, et donc la suite exacte du théorème \ref{quilttri}. Rappelons que si $x,y$ désignent des générateurs d'un même complexe $C_i$, l'aire symplectique d'une bande $u$ connectant $x$ et $y$ est donnée par $ A(u) = \frac{1}{8} I(u) + c(x,y)$\footnote{$\frac{1}{8} = \frac{1}{2} \kappa $, où $\kappa = \frac{1}{4}$ est la constante de monotonie des formes $\tilde{\omega}$.}, où $c(x,y)$  est une quantité indépendante de la bande $u$. Ces quantités vérifient $ c(x,z) = c(x,y)+ c(y,z)$. De même, si cette fois $x$ et $y$ désignent des générateurs de complexes différents $C_i$ et $C_j$, il existe des quantités $c(x,y)$ similaires donnant l'aire d'une section  des fibrations définissant $C\Phi_1$ et $C\Phi_2$, et ces quantités vérifient la même relation d'additivité, par additivité de l'aire et de l'indice.

Rappelons également que les différentielles $\partial_0$, $\partial_1$ et $\partial_2$, comptent des bandes d'indice 1, $C\Phi_1$ et $C\Phi_2$ des sections d'indice 0, et $h$ des sections d'indice $-1$. Ainsi, en posant $d(x,y) = c(x,y) +i(x) -i(y)$, où $i(x)  = 0,1,2$ désigne le sous-complexe auquel $x$ appartient, le coefficient $(\partial x, y)$ est de la forme $m(x,y) q^\frac{d(x,y)}{2}$. 

Fixons $x_0$ un générateur quelconque de $C$, et soit $f\colon C\to C$ définie sur les générateurs par : $f(x) = q^{d(x_0,x)} x$, où $i(x)  = 0,1,2$ désigne le sous-complexe auquel $x$ appartient.

Un calcul élémentaire montre alors que $\partial f = q^\frac{1}{8} f \partial_\zz$, avec  $(\partial_\zz x, y) = m(x,y)$. Il s'en suit que $H_*(\hat{D},\partial)$ et $H_*(\hat{D},\partial_\zz)$ sont isomorphes. Ainsi, d'après le théorème des coefficients universels, 
\[ H_*(\hat{D},\partial) \simeq H_*(A,\partial_\zz)\otimes_\zz \hat{\Lambda} \oplus \mathrm{Tor}_\zz ( H_*(A,\partial_\zz), \hat{\Lambda} ) [-1].\]
Or, $H_*(\hat{D},\partial) =0$ d'après le lemme \ref{lemmeperutz}, donc $H_*(A,\partial_\zz) =0$ d'après la classification des groupes abéliens de type finis, et du fait que $(\Z{n})\otimes_\zz \hat{\Lambda} \neq 0$.

\section{Action d'un twist de Dehn sur la surface}\label{sectiontwistespmod}


Le but de cette section et d'étudier la nature géométrique de la transformation induite par un twist de Dehn sur une surface $\Sigma$ le long d'une courbe non-séparante au niveau des espaces des modules $\N(\Sigma)$. On verra suite à la proposition \ref{expressionflot} que cette transformation peut s'exprimer comme un flot Hamiltonien sur le complémentaire d'une sous-variété coisotrope, puis on montrera dans le théorème \ref{twistinduit} que, lorsque $\Sigma$ est un tore privé d'un disque, cette transformation correspond presque à un twist de Dehn, à l'exception du fait que son support n'est pas compact dans $\N(\Sigma)$, mais on construira à partir de ce symplectomorphisme un twist de Dehn généralisé qui nous permettra de démontrer la suite exacte de chirurgie du théorème \ref{trianglechir}.



Le groupe des difféomorphismes de $\Sigma$ valant l'identité sur le bord agit de manière naturelle sur $\N (\Sigma )$ par tiré en arrière. Dans ce paragraphe nous montrons que l'action d'un twist de Dehn $\tau_K$ le long d'une courbe $K \subset  \Sigma$ s'exprime comme le temps 1 du flot d'un Hamiltonien lisse en dehors de la sous-variété coisotrope sphériquement fibrée $C_- = \{[A]|\mathrm{Hol}_K (A) = -I\}$.

La stratégie que nous adoptons, inspirée par \cite[Section 3]{WWtriangle}, est de découper la surface le long de $K$, puis d'introduire un espace des modules intermédiaire $\N (\Sigma_{cut} )$ associé à la surface découpée $\Sigma_{cut}$ (voir figure \ref{sigma_sigmacut}), dont la réduction symplectique pour une action $SU(2)$-Hamiltonienne naturelle s'avère être le complémentaire $\N(\Sigma) \setminus C_-$ (voir paragraphe \ref{liencut}).

Le fait que le twist de Dehn $\tau_K$ est isotope à l'identité dans $\Sigma_{cut}$ nous permettra d'exprimer son pull-back comme un flot Hamiltonien dans $\N (\Sigma_{cut} )$ invariant pour l'action de $SU(2)$ précédente, définissant ainsi un flot Hamiltonien dans le quotient symplectique $\N(\Sigma) \setminus C_-$.


\subsection{Groupes fondamentaux de $\Sigma$ et $\Sigma_{cut}$}

\begin{figure}[!h]
    \centering
    \def\svgwidth{\textwidth}
    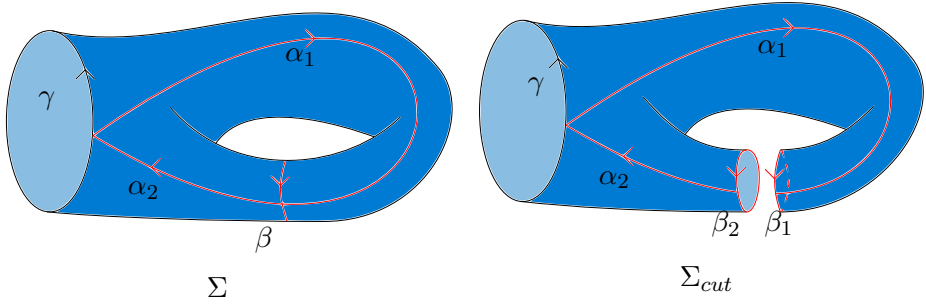
      \caption{Surfaces $\Sigma$ et $\Sigma_{cut}$.}
      \label{sigma_sigmacut}
\end{figure}

Soit $p\colon \rr/\zz \rightarrow \partial \Sigma$ le paramétrage du bord, et $* = p(0)$ le point base. La courbe $K$ étant non-séparante, il existe une courbe simple $\alpha\colon \rr/\zz \rightarrow  \Sigma$ basée en $*$ et intersectant transversalement $K$ en un seul point $\alpha (\frac{1}{2} )$. On note $\alpha_1 = \alpha_{[0,\frac{1}{2}]}$ et $\alpha_2 = \alpha_{[\frac{1}{2}, 1]}$. Notons $\beta\colon \rr/\zz \rightarrow  \Sigma$ un paramétrage de $K$ basé en $\alpha(\frac{1}{2} )$ et orienté de sorte que $\alpha . \beta = +1$ (et non comme des courbes d'une triade de chirurgie), et $\tilde{\beta} = \alpha_2^{-1} \beta \alpha_2 $. La surface $\Sigma \setminus (\alpha \cup \beta)$ est de genre $h -1$, soient  $u_2 , v_2 , \cdots , u_h , v_h$ des générateurs de son groupe fondamental, tels que dans $\Sigma$, en notant $\gamma = [p]$ la courbe du bord, $\gamma  = [\alpha , \tilde{\beta} ]\prod_{i=2}^{h}{[u_i , v_i ]}$. Les courbes  $\alpha , \tilde{\beta}, u_2 , v_2 , \cdots , u_h , v_h$ forment alors un système de générateurs de $\pi_1(\Sigma, *)$, et l'espace des modules étendu admet la description usuelle :

\[  \N(\Sigma,p) = \left\lbrace  (g, A, \tilde{B}, U_2 , V_2 , \cdots , U_h , V_h )|e^g = [A, \tilde{B}]\prod_{i=2}^{h}{[U_i , V_i ]} \right\rbrace
,\]
où $g\in \mathfrak{su(2)}$, de norme $< \pi \sqrt{2}$, est tel que la connexion vaut $g ds$ au voisinage du bord, $A, \tilde{B}, U_2 , V_2 , \cdots , U_h , V_h$ sont les holonomies le long des  courbes génératrices $\alpha , \tilde{\beta}, u_2 , v_2 , \cdots , u_h , v_h$.

Soit $\Sigma_{cut}$ la surface compacte obtenue en découpant $\Sigma$ le long de $K$. Notons $\beta_1$ et $\beta_2$ des paramétrages des nouvelles composantes de bord, coïncidant avec $\beta$ dans $\Sigma$, $\beta_1$ touchant $\alpha_1$ et $\beta_2$ touchant $\alpha_2$, voir la figure \ref{sigma_sigmacut}. On associe alors à $\Sigma_{cut}$ l'espace des modules suivant, défini dans \cite[Parag. 5.2]{jeffrey} par :
 \[\mathscr{M}^{\mathfrak{g},3}(\Sigma_{cut}) = \mathscr{A}_F^\mathfrak{g}(\Sigma_{cut}) /\Gc (\Sigma_{cut} ),\] 
où l'exposant $3$ fait référence aux nombre de composantes de bord de $\Sigma_{cut}$, $\mathscr{A}_F^\mathfrak{g}(\Sigma_{cut})$ est l'espace des connexions plates ($F$ pour "flat") sur $SU(2)\times\Sigma_{cut}$ de la forme $g ds$, $b_1 ds$ et $b_2 ds$ aux voisinages de $\gamma$,   $\beta_1$, $\beta_2$, et $s\in \rr/\zz$ représente le paramètre du bord. Le groupe  $\Gc (\Sigma_{cut} )$  des transformations de jauge triviales au voisinage du bord agit de manière naturelle sur $\mathscr{A}_F^\mathfrak{g}(\Sigma_{cut})$. 

On se restreindra à l'ouvert $\N(\Sigma_{cut}) \subset  \mathscr{M}^{\mathfrak{g},3}(\Sigma_{cut})$ des connexions pour lesquelles les vecteurs $g$, $b_1$ et $b_2$ sont dans la boule de rayon $\pi \sqrt{2}$. 

Cet espace admet la description suivante (voir \cite[Prop. 5.3]{jeffrey}) :

\begin{multline*}\N(\Sigma_{cut} ) \simeq \\ \left\lbrace (g, A_1 , A_2 , b_1 , b_2 , U_2 , V_2 , \cdots )~\left| ~e^g = A_1 e^{b_1} A_1^{-1} A_2^{-1} e^{b_2} A_2 \prod_{i=2}^{h}{[U_i , V_i ]}\right\rbrace \right. , \end{multline*}
où $g, b_1, b_2 \in \mathfrak{su(2)}$ sont les valeurs de la connexion le long des bords (éléments de la boule de rayon $\pi \sqrt{2}$), et $A_1 , A_2 ,  U_2 $ ,$ V_2 , \cdots , U_h , V_h \in SU(2)$ les holonomies le long des courbes $\alpha_1, \alpha_2, u_2 , v_2 , \cdots $, $u_h , v_h$.

Cet espace est muni d'une forme symplectique définie comme celle de $\N(\Sigma)$ :  si $[A]\in\N(\Sigma_{cut}) $ est l'orbite d'une connexion plate, et $\eta$, $\xi$ sont des 1-formes $\mathfrak{su(2)}$-valuées représentant des vecteurs tangents de $T_{[A]} \N(\Sigma_{cut})$, c'est-à-dire proportionelles à $ds$ sur chaque bord, et $d_A$-fermées, alors

\[ \omega_{[A]}([\eta],[\xi]) = \int_{\Sigma'} \langle \eta \wedge\xi \rangle.  \]

\subsection{Lien entre $\N(\Sigma )$ et $\N(\Sigma_{cut} )$ }\label{liencut}

Le groupe $SU(2)^3$ agit de manière Hamiltonienne sur $\N(\Sigma_{cut} )$, en effet $SU(2)^3$ s'identifie au quotient $\G ^{const} (\Sigma_{cut} ) / \Gc (\Sigma_{cut} )$, où $\G ^{const} (\Sigma_{cut} )$ est le groupe des transformations de jauge constantes au voisinage du bord. Le moment de cette action est donné par :

\[ \Psi = (\Phi_\gamma,\Phi_1,\Phi_2) \colon \N(\Sigma_{cut} ) \rightarrow \mathfrak{su(2)}^3,  \]
où $\Phi_\gamma ([A]) = g$, $\Phi_1([A]) = -b_1$ et $\Phi_2([A]) = b_2$, si $A$ est une connexion plate de la forme $g ds$, $b_1 ds$ et $b_2 ds$ aux voisinages de $\gamma$, $\beta_1$ et $\beta_2$ (le signe négatif dans $\Phi_1$ vient du fait que $\beta_1$ est orienté comme $\beta$, et non par normale sortante). Dans la description  holonomique de $\N(\Sigma_{cut} )$, cette action a pour expression :

\begin{multline*} (G,G_1 , G_2 ).(g, A_1 , A_2 , b_1 , b_2 , U_2 , V_2 , \cdots , U_h , V_h ) = \\ 
(ad_{G} g, G A_1 G_1^{-1} , G_2 A_2 G^{-1} , ad_{G_1} b_1 , ad_{G_2} b_2 ,G U_2 G^{-1},G V_2 G^{-1}, \cdots ).
\end{multline*}

En particulier, l'action de $SU(2)$ définie par $G.([A]) = (1,G,G).([A])$ est également Hamiltonienne, a pour moment $\Phi  = \Phi_1 + \Phi_2$, et pour expression :

\begin{multline*}G.(g, A_1 , A_2 , b_1 , b_2 , U_2 , V_2 , \cdots , U_h , V_h ) = \\
(g, A_1 G^{-1} , G A_2 , ad_G b_1 , ad_G b_2 , U_2 , V_2 , \cdots , U_h , V_h ).\end{multline*}

Notons $\N(\Sigma_{cut} )\red SU(2) $ le quotient symplectique pour cette action, et définissons une application $\N(\Sigma_{cut} )\red SU(2) \rightarrow \N(\Sigma)$ de la manière suivante : si $A$ est une connexion sur $\Sigma_{cut}$ telle que $\Phi([A]) = 0$, alors $b_1 = b_2$ et $A$ se recolle en une connexion sur $\Sigma$. Ceci définit une application $ \Phi^{-1} (0) \rightarrow \N(\Sigma)$. Si $G\in SU(2)$, et $\varphi \in \G ^{const}(\Sigma_{cut})$ est une transformation de jauge correspondant à $(1,G,G)$, $\varphi$ coïncide sur $\beta_1$ et $\beta_2$, et se recolle en une transformation de jauge de $\G^0(\Sigma)$, si bien que $A$ et $\varphi . A$ définissent le même élément de $\N(\Sigma)$, autrement dit l'application précédente passe au quotient en une application de $\N(\Sigma_{cut} )\red SU(2) $ vers $\N(\Sigma)$.

\begin{prop}\label{actioncut} Dans la description holonomique de $\N(\Sigma_{cut} )$ et $\N(\Sigma)$, cette application a pour expression :

\begin{multline*}[(g, A_1 , A_2 , b_1 , b_2 , U_2 , V_2 , \cdots , U_h , V_h )] \mapsto \\
(g, A = A_1 A_2 , \tilde{B} = A_2^{-1} e^{b_1} A_2 , U_2 , V_2 , \cdots , U_h , V_h ).\end{multline*}

Cette application réalise un symplectomorphisme sur son image $\N(\Sigma' ) \setminus C_-$ , où $C_- = \lbrace\tilde{B} = -I\rbrace$.
 
\end{prop}

\begin{proof} La description vient du fait que $\alpha  = \alpha_1 \alpha_2$ et $\tilde{\beta}  = \alpha_2^{-1} \beta_2 \alpha_2$.

L'exponentielle réalise un difféomorphisme entre la boule \[\{b_1 \in \mathfrak{su(2)} \ |\ |b_1 | < \pi \sqrt{2} \}\] et $G \setminus \{-I\}$, on note  $\log$ sa réciproque. On vérifie aisément que l'application réciproque est donnée par :
\begin{multline*}(g, A, \tilde{B}, U_2 , V_2 , \cdots , U_h , V_h ) \mapsto \\
 [(g, A_1 = A, A_2 = I, b_1 = \log(\tilde{B}), b_2 = b_1 , U_2 , V_2 , \cdots , U_h , V_h )]\end{multline*}
Ceci prouve le caractère bijectif. Enfin cette application préserve les formes symplectiques, car ces dernières sont définies de manière analogue, par intégration des formes sur $\Sigma$ et $\Sigma_{cut}$.
\end{proof}

\subsection{Description du twist de Dehn dans les espaces de modules}

Commençons par décrire l'action d'un twist de Dehn dans $\N(\Sigma_{cut} )$. Pour tout $t \in [0, 1]$, notons $\tau_t$ le difféomorphisme de $ \Sigma_{cut}$ valant l'identité en dehors d'un voisinage de la courbe $ \beta_1$ , et sur $\nu \beta_1\simeq \rr/\zz \times  [0, 1]$, $\tau_t (s, x) = (s + t\psi (x), x)$, où $\psi : [0, 1] \rightarrow [0, 1]$ est une fonction lisse valant 1 sur $[0, \frac{1}{3}]$ et 0 sur $[\frac{2}{3}, 1]$.

\begin{figure}[!h]
    \centering
    \def\svgwidth{0.65\textwidth}
    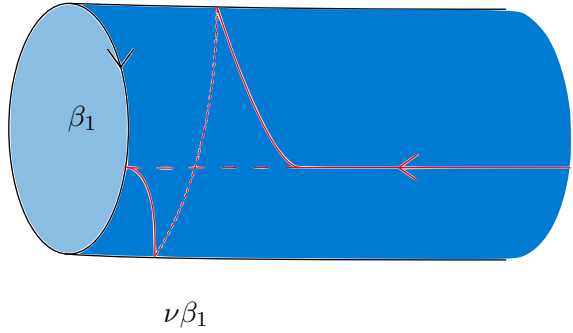
      \caption{Le twist $\tau_1$ au voisinage de $\beta_1$.}
      \label{twist}
\end{figure}

Seuls $\tau_0$ et $\tau_1$ se recollent en des difféomorphismes de $\Sigma$, respectivement en l'identité et en un twist de Dehn le long de $\beta$. Notons alors le tiré en arrière $\varphi_t = \tau_t^*: \N(\Sigma_{cut} ) \rightarrow \N(\Sigma_{cut} )$, défini par $\varphi_t ([A]) = [\tau_t^* A]$. 

\begin{prop}\label{expressionflot}
\begin{itemize}

\item[$(i)$] Dans la description holonomique de $\N(\Sigma_{cut} )$,  le tiré en arrière $\varphi_t$ a pour expression :

\begin{multline*}\varphi_t (g, A_1 , A_2 , b_1 , b_2 , U_2 , V_2 , \cdots , U_h , V_h ) = \\
(g, A_1 e^{tb_1} , A_2 , b_1 , b_2 , U_2 , V_2 , \cdots , U_h , V_h )\end{multline*}

\item[$(ii)$] Pour tout $t\in [0,1]$, $\varphi_t$ est le flot Hamiltonien au temps $t$ de la fonction $H : \N(\Sigma_{cut} ) \rightarrow \rr$ définie par $H([A]) =  \frac{1}{2}|\Phi_1([A]) |^2$.

\end{itemize}
\end{prop}

Afin de prouver la proposition, rappelons le fait suivant :

\begin{lemma}\label{gradcompo} Soit $G$ un groupe de Lie, $\mathfrak{g}$ son algèbre de Lie, $(M, \omega , \Phi )$ une variété $G$-Hamiltonienne ($\Phi  : M \rightarrow \mathfrak{g} \simeq \mathfrak{g}^*$) et $f : \mathfrak{g} \rightarrow \rr$ une fonction lisse, alors le gradient symplectique de $f \circ  \Phi  : M \rightarrow \rr$ est donné par :
\[\bigtriangledown^\omega (f \circ  \Phi )_m = X_{\bigtriangledown f (\Phi (m))} (m)\]
où $\bigtriangledown f$ est le gradient de $f$ pour un produit scalaire sur $\mathfrak{g}$ réalisant l'isomorphisme $\mathfrak{g} \simeq \mathfrak{g}^*$ et, pour $\eta \in \mathfrak{g}$, $X_\eta$ désigne le champs de vecteur sur $M$ correspondant à l'action infinitésimale de $G$.
\end{lemma}

\begin{proof}[Preuve du lemme]
Par définition, $\bigtriangledown^\omega (f \circ  \Phi )$ est tel que, pour tout $m \in M$ et $y \in T_m M$ ,
\[\omega_m (\bigtriangledown^\omega (f \circ  \Phi )_m , y) = \mathrm{D}_m (f \circ  \Phi ).y\]
or,
\begin{align*}
\mathrm{D}_m (f \circ  \Phi ).y &= \mathrm{D}_{\Phi  (m)}f \circ \mathrm{D}_m \Phi .y\\
 &= \left\langle \bigtriangledown f (\Phi (m)), \mathrm{D}_m \Phi .y \right\rangle \\
 &= \mathrm{D}_m (f_m).y \\
 &= \omega_m (X_{\bigtriangledown f (\Phi (m))}(m), y),
\end{align*}
où $f_m$ est la fonction sur $M$ qui à $m'$ associe $\langle \bigtriangledown f (\Phi (m)), \Phi (m') \rangle$, pour $m$ fixé.

\end{proof} 

\begin{proof}[Preuve de la proposition \ref{expressionflot}] $(i)$ Comme $\tau_t$ est l'identité aux voisinages de $\gamma$ et $\beta_2$, et est une rotation au voisinage de $\beta_1$, les valeurs de $g$, $b_1$ et $b_2$ ne sont pas modifiées par $\varphi_t$. Par ailleurs, $\tau_t$ laisse inchangées les courbes $\alpha_2 , u_2 , \cdots , v_h$ : les holonomies ne sont donc pas modifiées. De plus, il envoie $\alpha_1$ sur une courbe homotope à $\alpha_1 \cup \beta_1([0, t])$ , d'où \[\mathrm{Hol}_{\alpha_1} (\tau_t^* A) = \mathrm{Hol}_{\alpha_1 \cup \beta_1([0,t])} (A) = \mathrm{Hol}_{\alpha_1} (A) e^{tb_1}.\] 

$(ii)$ Tout d'abord on remarque d'après le point précédent que $\varphi_t ([A]) = (1,e^{-tb_1} , 1)[A] = (1,e^{t \Phi_1 ([A])} , 1)[A]$ pour l'action de $SU(2)^3$ définie précédemment.

Appliquons le lemme \ref{gradcompo} à $M=\N(\Sigma_{cut})$, muni de l'action de $SU(2)$ de moment $\Phi_1$, avec $f (\xi ) = \frac{1}{2} |\xi |^2$. $\bigtriangledown f (\Phi_1 ([A])) = \Phi_1 ([A]) = - b_1 ([A])$. De la première observation il vient $\frac{\partial \varphi}{\partial t} |_{t=0} = X_{\Phi_1 ([A])} ([A])$, et d'après le lemme, $X_{\Phi_1 ([A])} ([A]) = X_{\bigtriangledown f (\Phi_1 ([A]))} ([A]) = \bigtriangledown^\omega H([A])$. Ceci, et le fait que $\varphi_t$ vérifie la propriété de flot $\varphi_{t+h} = \varphi_t \circ \varphi_h$ achève la démonstration du point (ii).

\end{proof}

Rappelons la proposition suivante, dont la preuve repose essentiellement sur une version équivariante du théorème de plongement coisotrope (\cite[Theorem 39.2]{guilleminsternberg}) :


\begin{prop}(\cite[Prop. 2.15]{WWtriangle})
On suppose que $(M,\omega,\Phi)$ est une variété $SU(2)$-Hamil\-tonienne telle que l'application moment $\Phi \colon M \rightarrow \mathfrak{su(2)}^*$ prend ses valeurs dans la boule $\lbrace \xi \in \mathfrak{su(2)}~|~ |\xi |<\pi \sqrt{2}\rbrace$, et telle que le stabilisateur de l'action en chaque point de $\Phi^{-1}(0)$ est trivial (resp. $U(1)$). Soit $\psi \in C^\infty ([0,+\infty))$ telle que $\psi ' (0) = \pi \sqrt{2} $, à support compact, et dont le temps $1$ du flot Hamiltonien de $\psi \circ |\Phi|$ s'étend de manière lisse à  $\Phi^{-1}(0)$.

Alors  $\Phi^{-1}(0)$  est une sous-variété coisotrope sphériquement fibrée, de codimension 3 (resp. 2), et le temps $1$ du flot de $\psi \circ |\Phi|$ est un twist de Dehn fibré autour de  $\Phi^{-1}(0)$.
\end{prop}

\begin{remark}
Dans \cite[Prop. 2.15]{WWtriangle}, le résultat est énoncé pour $\psi ' (0) = \frac{1}{2} $ et le temps $2\pi$ du flot, mais le produit scalaire sur $\mathfrak{su(2)}$  utilisé est différent.
\end{remark}

Cette proposition s'applique pour l'action de $SU(2)$ sur $\N(\Sigma_{cut})$ de moment $\Phi_1$. En effet, d'une part $Im \Phi_1 \subset \lbrace |\xi |<\pi \sqrt{2}\rbrace$ par définition de $\N(\Sigma_{cut})$. D'autre part, d'après sa description holonomique, cette action est libre, et le flot se prolonge. Ainsi :

\begin{cor} \label{cortwist} Soit $R\colon \rr\rightarrow \rr$ une fonction nulle pour $t>\frac{\pi \sqrt{2}}{2}$ et telle que $R(-t) = R(t) -2 \pi \sqrt{2} t$. Alors le flot de $H = R\circ |\Phi_1|$ au temps $1$ se prolonge en un twist de Dehn fibré de $\N(\Sigma_{cut})$ autour de $C_+ = \Phi_1^{-1}(0)$, qui est une sous-variété coisotrope sphériquement fibrée.

\end{cor}

Rappelons maintenant le résultat suivant  afin d'établir le résultat pour $\N(\Sigma)$.
\begin{prop}(\cite[Theorem 2.10]{WWtriangle})\label{twistred}
Soient $G$ un groupe de Lie, $(M,\omega, \Phi)$ une variété $G$-Hamiltonienne telle que $0$ est une valeur régulière du moment $\Phi$. Soit $C\subset M$ une sous-variété coisotrope sphériquement fibrée au-dessus d'une base $B$ et stable sous l'action de $G$. On suppose que $C$ intersecte $\Phi^{-1}(0)$ transversalement, et que, en notant $\Phi_B \colon B\rightarrow \mathfrak{g}$ le moment induit sur $B$, l'action induite sur la base  $\Phi_B^{-1}(0)\subset B$ est libre. Soit $\tau_C \in Diff(M,\omega)$ un twist de Dehn fibré autour de $C$ qui est $G$-équivariant. 

Alors, le symplectomorphisme induit $[\tau_C]\colon M /\!\!/ G \rightarrow M /\!\!/ G $ est un twist de Dehn  fibré le long de $C /\!\!/ G$

\end{prop}
On considère $M = \N(\Sigma_{cut})$, munie de l'action de moment $\Phi =  \Phi_1 + \Phi_2$. La sous-variété $C = \Phi_1^{-1}(0)$ est une sous-variété coisotrope sphériquement fibrée au-dessus de \[B = \lbrace (g, A_2, b_2, \cdots)\rbrace \simeq \N(\Sigma_{cut, cap1}),\] où la surface $\Sigma_{cut, cap1}$ est obtenue à partir de $\Sigma_{cut}$ en collant un disque sur la composante de bord $\beta_1$. Le temps 1 $\tau_C$ du flot  de $R\circ |\Phi_1|$, où $R$ est une fonction comme dans le corollaire précédent, est un twist de Dehn fibré. 
On peut alors appliquer la proposition \ref{twistred} à cette situation. En effet, $\N(\Sigma_{cut})$ s'identifie à l'ouvert suivant de $\mathfrak{su(2)}^2\times SU(2)^{2h}$ des éléments $ (b_1,b_2, A_2, A_2, U_2, V_2, \cdots U_h, V_h)$ vérifiant 
\begin{align*}
 &\abs{b_1}<\pi \sqrt{2}, \ \abs{b_2}<\pi \sqrt{2},  \\ & A_1 e^{b_1} A_1^{-1} A_2^{-1} e^{b_2} A_2 \prod_{i=2}^{h}{[U_i , V_i ]} \neq -I .
\end{align*}
Sous cette identification, $\Phi$ et $\Phi_1$ correspondent respectivement à la différence des deux premières coordonnées et à l'opposé de la projection sur la première coordonnée. le vecteur nul $0\in \mathfrak{su(2)}$ est ainsi une valeur régulière de $\Phi$, et $C$ intersecte $\Phi^{-1}(0)$ transversalement le long de $\lbrace \Phi_1 = \Phi_2 = 0\rbrace$. De plus, l'action induite sur $\Phi_B^{-1}(0)$ est libre (elle affecte l'holonomie $A_2$ par multiplication à gauche), et le twist  $\tau_C$ est $SU(2)$-équivariant, en effet $\tau_C$ a pour expression :
\[ \tau_C (g, A_1, A_2, b_1, b_2, \cdots) = (\theta, A_1 e^{-t b_1} ,  A_2,  b_1,  b_2, \cdots),\]
où $t =  R'(|b_1|)$. Si $H\in G$ et $H.$ désigne l'action de moment $\Phi$, on a :
\begin{align*}
\tau_C \left( H.(g, A_1, A_2, b_1, b_2, \cdots)\right)  &=  H.\tau_C(g, A_1, A_2, b_1, b_2, \cdots) \\
 &= (g, A_1 e^{-t b_1} H^{-1}, H A_2, ad_H b_1, ad_H b_2, \cdots).
\end{align*}

Ainsi, il vient d'après le corollaire \ref{cortwist} et la proposition \ref{twistred} :
\begin{prop}\label{twistsigma}
Soit $R$ comme dans le corollaire \ref{cortwist}. Le temps 1 du flot Hamiltonien de la fonction $R(|\log ( \tilde{B})|  )$ est un twist de Dehn fibré de $\N(\Sigma)$ autour de $\lbrace \tilde{B} = I\rbrace$.
\end{prop}
Notons toutefois que lorsque $\Sigma$ est de genre supérieur ou égal à 2, la sous-variété $\lbrace \tilde{B} = I\rbrace$ n'est pas compacte dans $\N(\Sigma)$ : son adhérence dans $\Nc(\Sigma)$ intersecte l'hypersurface $R$. En revanche si $\Sigma$ est de genre 1, elle est contenue dans le  niveau $\theta = 0$. Ainsi :
\begin{theo}
\label{twistinduit}
 Soit $H$ un tore solide bordé par un tore $T$, et $T'$ la surface obtenue en retirant un petit disque. Notons $i\colon T \rightarrow \partial H$ l'inclusion, et $L(H)\subset \N(T')$ la Lagrangienne associée. Soit $\tau_K$ un twist de Dehn le long d'une courbe $K \subset T'$ non-séparante, $i' = i\circ \tau_K$ et $H' = (H, i')$ le cobordisme entre $\emptyset$ et  $T'$, et $L(H')\subset \N(\Sigma)$. Alors il existe un twist de Dehn le long de $S = \lbrace \mathrm{Hol}_K = -I\rbrace$ envoyant $L(H')$ sur $L(H)$. 
\end{theo}



\begin{remark} Le symplectomorphisme induit par le twist sur la surface n'est pas à priori un twist de Dehn de $\N(T')$ car l'Hamiltonien qui l'engendre n'est pas à support compact, mais on va construire un twist de Dehn (que l'on appellera $tw$) en tronquant l'Hamiltonien.
\end{remark}

\begin{proof} Rappelons que l'on a identifié $\N(T')$ à la partie 
\[\left\lbrace (g, A, \widetilde{B}) \in \mathfrak{su(2)}\times SU(2)^{2}  : e^g = [A,\widetilde{B}] \right\rbrace ,\] 
où $A$ et $\widetilde{B}$ désignent les holonomies le long des lacets $\alpha$ et $\widetilde{\beta}$.  Définissons trois fonctions \[H^f,H^{tw},H^\tau \colon\N(T') \rightarrow \rr\] par : 
\begin{align*}
H^f  (A, \widetilde{B}) & = \frac{1}{2} |\log(\tilde{B})|^2,\text{ en posant }|\log(-I)| = \pi \sqrt{2} \\
H^{tw}  (A, \tilde{B}) & = \phi(A, \tilde{B}) H^f  (A, \tilde{B}), \\
H^\tau  (A, \tilde{B}) & = R(|\log (- \tilde{B})|  ),
\end{align*}
où $\phi$ est une fonction à support compact et valant 1 sur un voisinage de $\lbrace g = 0\rbrace$, par exemple $\lbrace |g| < \epsilon\rbrace$ pour $\epsilon$ assez petit, $R\colon \rr_+\rightarrow \rr$ est nulle pour $t> \frac{\pi \sqrt{2}}{2}$, et telle que $R(t) = \pi^2 - \pi \sqrt{2} t + \frac{1}{2} t^2$ pour $t< \frac{\pi \sqrt{2}}{4}$. 

Ces trois fonctions coïncident au voisinage de $\lbrace \tilde{B} = -I\rbrace$ : ceci est clair pour $H^f$ et $H^{tw}$ car $\lbrace \tilde{B} = -I\rbrace \subset \lbrace g = 0 \rbrace$, et si $-\tilde{B}$ est conjuguée à $\begin{pmatrix}
e^{i\alpha}& 0 \\
0 & e^{-i \alpha}
\end{pmatrix}$,  
avec $\alpha\in [0,\pi]$, alors $\tilde{B}$ est conjuguée à $\begin{pmatrix}
e^{i(\pi -\alpha)}& 0 \\
0 & e^{-i (\pi -\alpha)}
\end{pmatrix}$, 
et $\frac{1}{2} |\log ( \tilde{B}) |^2 = (\pi -\alpha)^2 = R(|\log (- \tilde{B})| )$, car $|\log (- \tilde{B})| = \alpha \sqrt{2}$. Ainsi, $H^{tw} = H^\tau$ au voisinage de $\lbrace \tilde{B} = -I\rbrace$.

D'après la proposition  \ref{expressionflot}, le temps 1 du flot de $H^f$ est induit par le twist géométrique et se prolonge de manière lisse à $\N(\Sigma)$, il en va donc de même pour les flots de $H^{tw}$ et $H^\tau$ : notons alors $f$, $tw$ et $\tau$ ces prolongements respectifs.

D'une part, l'ensemble $ \lbrace g = 0 \rbrace$ est invariant par le flot  de $H^f$ en tout temps, il s'en suit que $f$ et $tw$ y coïncident, et donc $L(H') = tw(L(H))$, car  $L(H') = f(L(H))$ et $L(H)$ est contenue dans $ \lbrace g = 0 \rbrace$.

Par ailleurs, d'après la proposition \ref{twistsigma}, $\tau$ est l'inverse d'un twist de Dehn  le long de $\lbrace \tilde{B} = -I\rbrace$. En effet, en notant $\varphi$ l'involution \[(A, \tilde{B}) \mapsto (A, -\tilde{B})\] de $\N(T')$, l'application $\varphi \tau \varphi^{-1}$ est un twist de Dehn fibré le long de $\lbrace \tilde{B} = I\rbrace$.

Ainsi, $tw$ s'écrit comme la composée $(tw\circ \tau^{-1})\circ\tau$, avec $\tau$ un twist de Dehn le long de  $\lbrace \tilde{B} = -I\rbrace$. Remarquons enfin que  $tw\circ \tau^{-1}$ est une isotopie Hamiltonienne à support compact : en effet, en dehors de $\lbrace \tilde{B} = -I\rbrace$ c'est le temps 1 du flot de l'Hamiltonien 

\[H^{comp}(t,x) = H^{tw}(x) - H^\tau( \phi_{tw}^{t} (x)  ),\]
où $\phi_{tw}^{t}$ est le flot au temps $t$ de $H^{tw}$. Or, au voisinage de $\lbrace \tilde{B} = -I\rbrace$, $\phi_{tw}^{t}$ coïncide avec le flot de $H^\tau$, donc $H^{comp}(t,x) = H^{tw}(x) - H^\tau (x)$ au voisinage de $\lbrace \tilde{B} = -I\rbrace$, et $H^{comp}$ se prolonge de manière lisse à $\lbrace \tilde{B} = -I\rbrace$.
\end{proof}

\subsection{Preuve du triangle de chirurgie}

Dans ce paragraphe, nous prouvons le théorème \ref{trianglechir}. 

\begin{proof}

Soient $\alpha$, $\beta$ et $\gamma$ des courbes sur le tore épointé $T' = \partial Y \setminus \lbrace\mathrm{petit~disque}\rbrace$ formant une triade, on a  $\beta^{-1} = \tau_\alpha \gamma$, où $\tau_\alpha$ est un twist de Dehn autour de $\alpha$. Ainsi, en notant 
\begin{align*}  L_{\alpha} ^{-} &= \lbrace \mathrm{Hol}_\alpha = -I \rbrace, \\
 L_\beta &= \lbrace \mathrm{Hol}_\beta = I \rbrace,  \\
L_\gamma &= \lbrace \mathrm{Hol}_\gamma = I \rbrace
\end{align*} les trois sphères Lagrangiennes de $\Nc(T')$, d'après le théorème \ref{twistinduit}, il existe un twist de Dehn généralisé $\tau_S$ de $\Nc(T')$ autour de $S = L_{\alpha} ^{-}$ qui envoie $L_\gamma$ sur $L_\beta$. En effet, soit $H$  le corps à anses de genre 1 dans lequel $\beta^{-1}$ borde un disque, et $i \colon T'\rightarrow \partial H$ l'inclusion, on a $i(\beta^{-1}) = \partial D^2$. Si $i' = i\circ \tau_\alpha$, on a  $ i'(\tau_\alpha ^{-1} \beta ^{-1}) = i'(\gamma)$.

En notant $\underline{L} = L(Y,c)$, $S = L_{\alpha} ^{-}$ et $L_0 = L_\beta$, Le théorème \ref{quilttri} fournit alors la suite exacte :
\[ \ldots \rightarrow HF(\tau_S L_0, \underline{L}) \rightarrow HF(L_0, \underline{L}) \rightarrow HF(L_0,S^T, S, \underline{L})\rightarrow \cdots . \]

Il reste à identifier les termes : les Lagrangiennes $L_\beta$ et $L_\gamma$ étant associées au cobordisme consistant à l'ajout d'une 2-anse le long de $\beta$ (resp. $\gamma$) et sans classe d'homologie, il vient pour les deux premiers termes : $HF(\tau_S L_0, \underline{L}) = HF(L_\beta, \underline{L}) = HSI(Y_\beta,  c_\beta)$, et $ HF(L_0, \underline{L})= HF( L_\gamma, \underline{L}) = HSI(Y_\gamma,  c_\gamma)$. Enfin, $S = L_{\alpha} ^{-}$ correspond au cobordisme consistant à l'ajout d'une 2-anse le long de $\alpha$, avec classe d'homologie $k_\alpha$, il vient d'après la formule de Künneth (proposition \ref{sommecnx}) et en utilisant que $HF(L_0,S) = HSI(S^3) = \zz$ : $ HF(L_0,S^T, S, \underline{L}) = HF(L_0,S) \otimes_\zz HF( S, \underline{L}) = HSI(Y_\alpha, k_\alpha + c_\alpha)$, ce qui achève la preuve.

\end{proof}

\section{Applications de la suite exacte}\label{sectionapplic}

Dans cette section nous donnons quelques applications immédiates de la suite exacte du théorème \ref{trianglechir}. Elles ne supposent aucune propriété des morphismes qui y interviennent, et se basent sur une observation due à Ozsv{\'a}th et Szab{\'o}. Nous commençons par la rappeler, puis nous donnons des classes de variétés pour lesquelles l'homologie HSI est minimale, toutes ces variétés sont des L-espaces en théorie d'Heegaard-Floer.

\subsection{L'observation d'\OSz }Le fait suivant a été remarqué par Ozsv{\'a}th et Szab{\'o}, voir par exemple \cite[Exercice 1.13 ]{OSzlectures}. Il peut être démontré directement, ou se déduire de la suite exacte de chirurgie (pour $\widehat{HF}$ ou $HSI$) en prenant la caractéristique d'Euler.

\begin{lemma} Soient $Y_\alpha$, $Y_\beta$ et $Y_\gamma$ une triade de chirurgie. En notant, pour un ensemble $H$, la quantité :
\[ \abs{H} = \begin{cases} \mathrm{Card} H \text{ si $H$ est fini} \\ 0 \text{ sinon} \end{cases},\]
on a, quitte à réordonner les variétés, $ \abs{H_1(Y_\alpha;\zz)} = \abs{H_1(Y_\beta;\zz)} + \abs{H_1(Y_\gamma;\zz)} $.
\end{lemma}
\arnaque

Définissons les variétés HSI-minimales, analogues des L-espaces en théorie d'Heegaard-Floer :
\begin{defi}Une 3-variété $Y$ est dite \emph{HSI-minimale} si pour toute classe $c\in H_1(Y;\Z{2})$,  $HSI(Y,c)$ est un groupe abélien libre de rang $\abs{H_1(Y;\zz)}$.
\end{defi}

\begin{remark}D'après la proposition \ref{genreun},  $S^2\times S^1$ n'est pas HSI-minimale, les lenticulaires le sont.
\end{remark}

Il découle alors de la suite exacte de chirurgie (théorème \ref{trianglechir})  et de la formule donnant la caractéristique d'Euler de l'homologie HSI (proposition \ref{eulercara}) :

\begin{prop}\label{triadehsimin}Soit $(Y_\alpha, Y_\beta, Y_\gamma)$ une triade de chirurgie, avec $Y_\beta$ et $Y_\gamma$ des variétés HSI-minimales, et $\abs{H_1(Y_\alpha;\zz)} = \abs{H_1(Y_\beta;\zz)} + \abs{H_1(Y_\gamma;\zz)}$. Alors $Y_\alpha$ est également HSI-minimale.
\end{prop} 
\begin{proof}Soit $c_\alpha \in H_1(Y_\alpha;\Z{2})$, et $c_\beta$, $c_\gamma$ deux autres classes sur $Y_\beta$ et $Y_\gamma$ pour lesquelles le théorème \ref{trianglechir} fournit une suite exacte entre les trois groupes d'homologie HSI. Supposons par l'absurde que la flèche entre $HSI(Y_\beta,c_\beta)$ et $HSI(Y_\gamma,c_\gamma)$ est non nulle, alors on aurait \[\mathrm{rk}HSI(Y_\alpha,c_\alpha) < \mathrm{rk}HSI(Y_\beta,c_\beta) + \mathrm{rk}HSI(Y_\gamma,c_\gamma) = \chi ( HSI(Y_\alpha,c_\alpha) ), \]
ce qui est impossible. Ainsi la suite exacte est une suite exacte courte, et  $HSI(Y_\alpha,c_\alpha)$ un groupe abélien libre de rang  $\abs{H_1(Y_\alpha;\zz)}$.

\end{proof}

\subsection{Quelques familles de variétés HSI-minimales}
Mentionnons à présent quelques applications classiques de l'observation précédente : 
\paragraph{Plombages}
Soit $(G,m)$ un graphe pondéré : $m$ est une fonction définie sur l'ensemble des sommets du graphe $G$, à valeurs dans $\zz$. Rappelons que l'on peut associer à $(G,m)$ une 4-variété à bord obtenue en plombant des fibrés en disques au-dessus de sphères associés aux sommets, dont le nombre d'Euler vaut $m(v)$. Son bord est une 3-variété fermée orientée $Y(G,m)$.

Dans \cite{OSplumbing}, \OSz ont calculé l'Homologie de Heegaard-Floer d'une grande famille de telles 3-variétés, en termes de vecteurs caractéristiques de la forme d'intersection associée. Leur calcul  a été ensuite étendu par Némethi dans \cite{Nemethi}. S'il nous manque un ingrédient principal de leur calcul, la formule d'adjonction, la proposition \ref{triadehsimin} permet néanmoins de montrer que les variétés suivantes sont HSI-minimales.  

\begin{prop}\label{bonplombage}Supposons que $G$ soit une réunion disjointe d'arbres, et que, en notant $d(v)$ le nombre d'arrêtes incidentes à un sommet $v$, la fonction $m$ vérifie, pour tout sommet $v$ de $G$, $m(v)\geq d(v)$, avec au moins un sommet pour lequel l'inégalité est stricte. Alors $Y(G,m)$ est HSI-minimale.
\end{prop}

\begin{remark} Si $m(v)= d(v)$ pour tous les sommets de $G$, alors une succession de contractions permet de montrer que $Y(G,m) \simeq S^2\times S^1$.
\end{remark}
\begin{proof} La preuve est analogue à celle de \cite[Theorem 7.1]{OSsymplectic4} : on procède par induction sur le nombre de sommets et les poids. Tout d'abord, si le graphe $G$ consiste en un seul sommet, alors $Y(G,m)$ est un lenticulaire, et le résultat découle de la proposition \ref{genreun}.
Montrons l'induction sur le nombre de sommets. L'ajout d'une feuille $v$ avec $m(v) = 1$ correspond à un éclatement, et ne change pas le type topologique de $Y(G,m)$.

Montrons finalement l'induction sur le poids d'une feuille. Soient $(G,m)$ un graphe satisfaisant les hypothèses du théorème, $v$ une feuille de $G$, $G'$ le graphe  obtenu en retirant $v$, $m'$ la restriction de  $m$ à $G'$, et $\tilde{m}$ coïncidant avec $m$ en dehors de $v$, et telle que $\tilde{m}(v) = m(v)+1$. Supposons que  $(G,m)$ et $(G',m')$ satisfont l'hypothèse d'induction.

Les variétés $Y(G,\tilde{m})$, $Y(G,m)$ et $Y(G',m')$ forment une triade de chirurgie, et $| H_1(Y(G,\tilde{m});\zz)| = | H_1(Y(G,m);\zz)| + | H_1(Y(G',m');\zz)|$, voir \cite[proof of Th. 7.1]{OSsymplectic4}. Ainsi l'induction découle de la proposition \ref{triadehsimin}.
\end{proof} 

\paragraph{Les revêtements doubles ramifiés de $S^3$ au-dessus d'entrelacs quasi-alternés}
Dans \cite[Def. 3.1]{OSzdoublecover}, \OSz  ont défini la classe suivante d'entrelacs, appelés "quasi-alternés" : il s'agit de la plus petite classe d'entrelacs vérifiant :

\begin{enumerate}
\item le noeud trivial est quasi-alterné,

\item   Soit $L$ un entrelacs. S'il existe une projection et un croisement de $L$ tels que les deux résolutions soient quasi-alternés,  $ \mathrm{det} L_0, \mathrm{det} L_1 \neq 0$ , et $  \mathrm{det} L =\mathrm{det} L_0 + \mathrm{det} L_1 $, alors $L$ est également quasi-alterné.
\end{enumerate}
D'après  \cite[Lemma 3.2]{OSzdoublecover}, cette classe contient les entrelacs admettant une projection connexe alternée. Il vient immédiatement de la proposition \ref{triadehsimin} :

\begin{prop}Les revêtements doubles d'entrelacs quasi-alternés sont des variétés HSI-minimales.
\end{prop}
\arnaque

\paragraph{Chirurgies entières sur certains noeuds} Enfin, soit $K\subset S^3$ un noeud tel que, pour un certain entier $n_0>0$, la chirurgie $S^3_{n_0}(K)$ est HSI-minimale. Étant donné que pour un entier $n>0$,  $\abs{ H_1(S^3_{n}(K) ,\Z{2}) } = n$, il vient que $S^3_{n}(K)$ est HSI-minimale pour tout $n>n_0$. 

\chapter{Applications associées à un cobordisme}\label{chapcob}

Dans tout ce chapitre, afin d'éviter des complications liées aux signes, tous les groupes d'homologie de Floer considérés seront à coefficients dans $\Z{2}$. Une construction similaire à coefficients dans $\zz$ est probablement possible également.

L'objectif de ce chapitre est d'associer à un 4-cobordisme entre deux 3-variétés un morphisme entre les groupes d'homologie instanton-symplectique des variétés. Soient $(Y,c)$ et $(Y',c')$ deux 3-variétés fermées orientées munies de classes  dans $H_1(Y;\Z{2})$ et $H_1(Y';\Z{2})$ respectivement. Soit $W$ un 4-cobordisme compact orienté de $Y$ vers $Y'$, muni d'une classe $c_W \in H^2(W; \Z{2})$ dont les restrictions à $Y$ et $Y'$ sont Poincaré-duales à $c$ et $c'$ respectivement. On va construire un  morphisme
\[ F_{W,c_W} \colon HSI(Y,c)\to HSI(Y',c').\]

Ce type de construction présente plusieurs intérets : 
\begin{itemize}
\item il permet d'obtenir des invariants pour des variétés de dimensions 4,

\item il donne une interprétation topologique des morphismes intervenant dans la suite exacte de chirurgie : voir le théorème \ref{thinterpfleches} et le corollaire \ref{corinterpfleches}. Cette dernière interprétation permet d'obtenir des critères d'annulation pour les morphismes,  en particulier une formule d'éclatement (corollaire \ref{éclatement}).
\end{itemize} 

\begin{remark}Par dualité de Poincaré-Lefschetz, la classe $c_W$ peut également être vue comme une classe d'homologie relative dans  $H_2(W,\partial W; \Z{2})$ dont l'image dans  $H_1(\partial W; \Z{2})$ par l'homomorphisme de connexion vaut $c+c'$. Il était plus commode de manipuler des classes d'homologie dans le chapitre \ref{chapfft}, mais il sera plus judicieux ici de manipuler des classes de cohomologie. 
\end{remark}

Après avoir construit ces applications et vérifié que ce sont des invariants topologiques, nous montrons que deux des morphismes intervenant dans la suite exacte de chirurgie correspondent à de telles applications, théorème \ref{thinterpfleches}. Enfin nous énonçons quelques propriétés satisfaites par ces applications.

%

\begin{remark}De même qu'en théorie d'Heegaard-Floer, le caractère bien défini de ces applications   repose en toute rigueur sur la naturalité des invariants. Dans la section \ref{sectionnaturalite}, nous conjecturons que ces groupes sont naturels, et indiquons une stratégie possible afin de démontrer cette naturalité. De même que les groupes devraient être associés à des variétés pointées, les applications, si l'on ne les regarde pas seulement à isomorphisme près, devraient être associées à des cobordismes munis de classes d'homotopie de chemins reliant les points bases.
\end{remark}

\section{Construction}

On procède de manière analogue à \cite{OSholotri} : on découpe $W$ en cobordismes élémentaires, correspondants à l'attachement d'une seule anse, puis on définit les morphismes associés à de tels cobordismes, enfin on vérifie que le morphisme obtenu en composant ne dépend pas de la décomposition.

\begin{remark} A la lumière de l'interprétation géométrique des classes $c$ et $c'$, le couple $(W,c_W)$ devrait s'interpréter comme une classe d'isomorphismes de $SO(3)$-fibrés principaux au-dessus de $W$. En effet, ces derniers sont classifiés par le leur seconde classe de Stiefel-Whitney car $W$ a le type d'homotopie d'un CW-complexe de dimension 3.
\end{remark}

\subsection{Ajout d'une 1-anse ou d'une 3-anse}

Soit $W$ un 4-cobordisme entre $Y$ et $Y'$ correspondant à l'attachement d'une 1-anse à $Y$. La variété  $Y'$ est homéomorphe à la somme connexe $(S^2\times S^1)\# Y$, et $W$ est homéomorphe à la somme connexe à bord \[ W \simeq (D^3 \times S^1)\#_\partial Y \times [0,1].\] 
 Il s'en suit que $H^2(W; \Z{2}) \simeq H_1(Y,\Z{2})$. Soit $c_W \in H^2(W; \Z{2})$, et  $c\in H_1(Y; \Z{2})$ (resp. $c'\in H_1(Y'; \Z{2})$) le dual de la restriction de $c_W$ à $Y$ (resp. $Y'$). Notons que $H_1(Y'; \Z{2})$ s'identifie à $H_1(Y; \Z{2})\oplus H_1(S^2\times S^1; \Z{2})$. Sous cette identification, $c' = c+0$.

Afin de pouvoir désigner des classes, on fixera la graduation absolue suivante  : $HSI(S^2\times S^1,0) = \Z{2}^{(3)} \oplus \Z{2}^{(0)}$, où l'exposant désigne le degré modulo 8. Sous cette identification, notons $\Theta\in \Z{2}^{(3)}$ l'élément non-trivial de $\Z{2}^{(3)}$.

D'après la formule de Künneth (proposition \ref{sommecnx}), \[HSI(Y',c') \simeq HSI(S^2\times S^1,0) \otimes HSI(Y,c)  .\]

\begin{defi} On pose $F_{W, c_W} (x) = \Theta \otimes x$.

Le cobordisme $\overline{W}$ muni de l'orientation opposée, vu comme un cobordisme de $Y'$ vers $Y$, correspond à l'ajout d'une 3-anse. On pose de même $F_{\overline{W}, c_W} (\Theta \otimes x) =  x$, et $F_{\overline{W}, c_W} (\Theta' \otimes x) = 0$, si $\Theta '$ désigne le générateur de degré 0 de $HSI(S^2\times S^1,0)$.
\end{defi}

\begin{remark} Nous verrons que le choix de $\Theta$ dans la définition précédente est imposé si l'on veut avoir invariance par naîssance/mort d'une 1-anse et d'une 2-anse.
\end{remark}

\subsection{Ajout d'une 2-anse}

Soit $Y$ une 3-variété, $K\subset Y$ un n{\oe}ud muni d'une trivialisation de son fibré normal, et $W$ le cobordisme correspondant à l'ajout d'une 2-anse le long de $K$. On note $Y'$ le deuxième bord de $W$, qui correspond donc à la chirurgie zéro le long de $K$. Notons $T = \partial (Y\setminus \nu K)$ le tore bordant le voisinage tubulaire de $K$.

D'après la suite exacte de Mayer-Vietoris appliquée à la réunion 
\[ W = \left( ( Y\setminus \nu K ) \times[0,1]\right) \cup_{\partial \nu K \times [0,1]} D^2\times D^2\] relativement à $Y \cup Y'$, il existe une classe $c\in H_1(Y\setminus \nu K; \Z{2})$ et des entiers $i, j \in \lbrace 0,1\rbrace$ tels que l'image $\partial_* PD(c_W) \in H_1(\partial W; \Z{2})$ s'écrit comme la somme de $c+ i[K] \in H_1(Y; \Z{2})$  et $ c + j[K']\in H_1(Y'; \Z{2})$, où $K'\subset Y'$ désigne l'âme de la chirurgie.

Notons $\lambda, \mu \subset T$ une longitude et un méridien de $K$ ne passant pas par un point base $z$ de $T$, $T'$ le tore $T$ éclaté en $z$, $\underline{L} = \underline{L}(Y\setminus K , c )$ la correspondance Lagrangienne généralisée de  $\Nc(T')$ vers $pt$. On considère les deux Lagrangiennes de $\Nc(T')$ suivantes :
\begin{align*}
L_0 =&  \left\lbrace  [A]\in\Nc(T') :  \mathrm{Hol}_\mu A = (-I)^i \right\rbrace \\
L_1 =&  \left\lbrace  [A]\in\Nc(T') :  \mathrm{Hol}_\lambda A = (-I)^j \right\rbrace.
\end{align*}
Des décompositions de $Y$ et $Y'$ correspondantes, il vient alors :  
\begin{align*}
HSI(Y, c+ i [K])  =&  HF(L_0, \underline{L}) \\
HSI(Y', c+ j [K'])  =&  HF(L_1, \underline{L}).
\end{align*}
Quelles que soient les valeurs de $i$ et $j$, les Lagrangiennes $L_0$ et $L_1$ s'intersectent transversalement en un point, en effet si l'on identifie $\N(T')$ à  un ouvert de $SU(2)^2$ via les holonomies $\Lambda$ et $M$ le long de $\lambda$ et $\mu$, elles ont pour équations respectives $\lbrace M = (-1)^i I \rbrace$ et $\lbrace \Lambda = (-1)^j I\rbrace$. Ainsi $HF( L_1,L_0) = \Z{2}$, on note $C$ le générateur correspondant au point d'intersection, $HF(L_1, L_0) \otimes HF(L_0, \underline{L}(Y\setminus K , c ))$ et $HF(L_0, \underline{L}(Y\setminus K , c )) $ sont alors canoniquement identifiés.

Soit  $\Phi\colon HF(L_1, L_0) \otimes HF(L_0, \underline{L}(Y\setminus K , c )) \to HF(L_1, \underline{L}(Y\setminus K , c ))$ le produit en pantalon, défini en comptant des pantalons matelassés comme pour la construction de l'application $C\Phi_1$  dans la preuve de la suite exacte du théorème \ref{quilttri} (voir figure \ref{quiltedpants}).

\begin{defi} On définit $F_{W,c_W}$ par $F_{W,c_W} (x) = \Phi(C \otimes x)$.
\end{defi}

\section{Indépendance de la décomposition}

Supposons que l'on a décomposé le cobordisme $W$ en $k$ cobordismes élémentaires :
\[ W = W_1 \cup_{Y_1} W_2 \cup_{Y_2} \cdots \cup_{Y_{k-1}} W_k,\]
avec $Y_0 = Y$, $Y_k = Y'$, et $W_i$ un cobordisme de $Y_{i-1}$ vers $Y_i$ correspondant à l'attachement d'une 1,2 ou 3-anse. Une classe $c_W \in  H^2( W; \Z{2}) $ induit des classes $ c_{W_i} \in  H^2( W_i; \Z{2})$  et des classes $ c_{Y_i} \in  H^2( Y_i; \Z{2})$, par tirés en arrière. Pour chaque bloc élémentaire $(W_i,c_{W_i})$, on vient de définir un morphisme $F_{W_i,c_{W_i}} \colon HSI(Y_{i-1},c_{Y_{i-1}}) \to HSI(Y_i,c_{Y_i})$. On définit alors $F_{W,c_W}\colon HSI(Y,c) \to HSI(Y',c')$ comme la composée :
\[F_{W,c_W} =F_{W_k,c_{W_k}}\circ \cdots \circ   F_{W_2,c_{W_2}}   \circ F_{W_1,c_{W_1}}.\]

Pour vérifier que cette application ne dépend pas de la décomposition choisie, il suffit, d'après la théorie de Cerf (Théorème \ref{GWWtheo}), de vérifier que le morphisme construit ne change pas pour une naîssance/mort, et pour une interversion de points critiques, les mouvements de difféo-équivalence et d'ajout de cobordisme trivial étant clairement vérifiés.

\subsection{Naîssance/mort} Nous allons montrer que pour l'ajout successif d'une 1-anse puis d'une 2-anse qui s'annulent, la composition des applications induites par cobordisme vaut l'identité. Le cas d'une 2-anse puis d'une 3-anse s'obtient en renversant le cobordisme. 

La situation que l'on va décrire est résumée dans la figure \ref{fig:naissancemort} :

\begin{figure}[!h]
    \centering
    \def\svgwidth{.65\textwidth}
    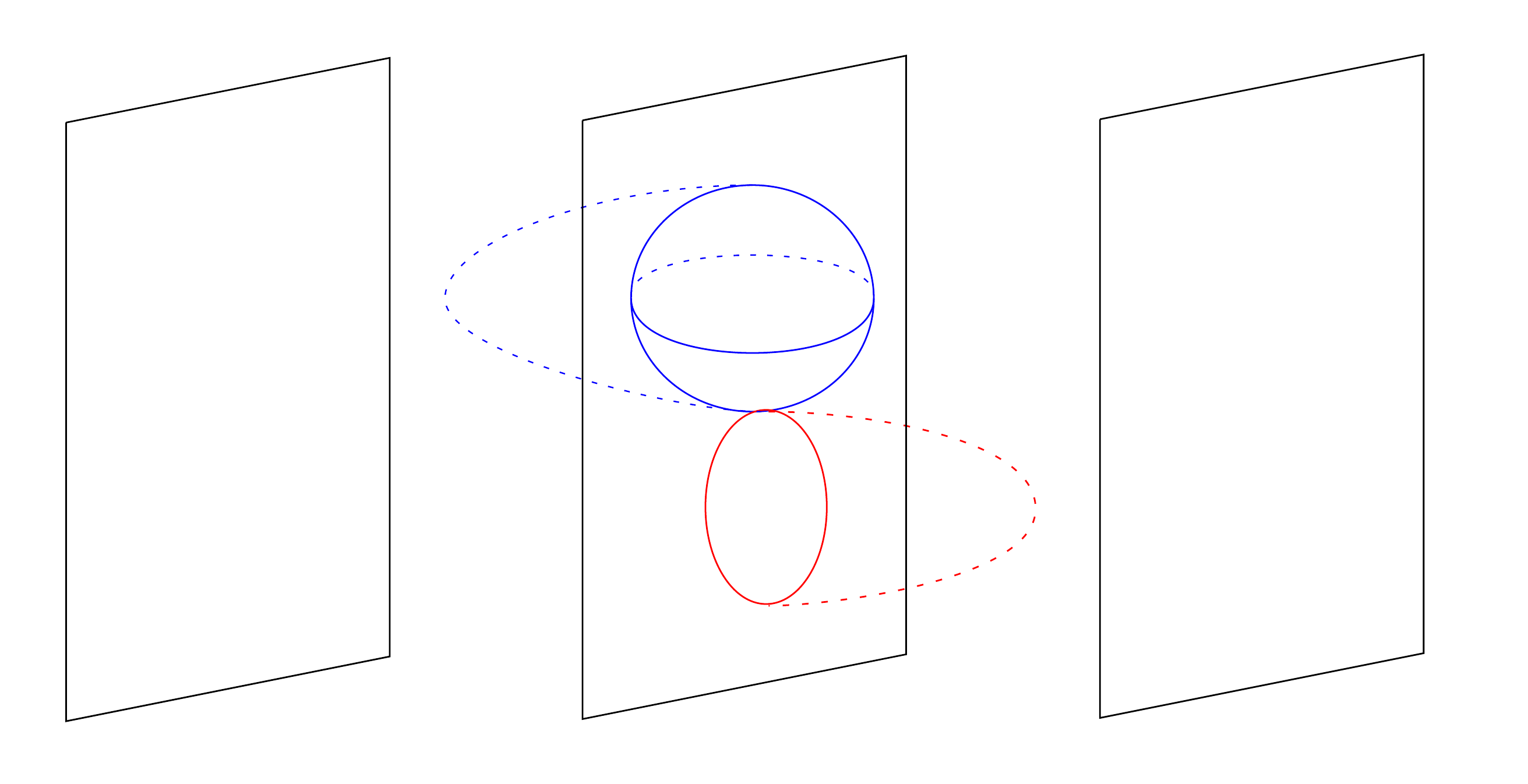
      \caption{Naîssance/Mort.}
      \label{fig:naissancemort}
\end{figure}


Soit $Y'$ une 3-variété, $S\subset Y'$ une 2-sphère et $K\subset Y'$ un noeud muni d'une trivialisation de son fibré normal, tels que $K$ et $S$ s'intersectent transversalement en un point. Soit $W_1$ l'opposé du cobordisme correspondant à l'attachement d'une 3-anse à $S$. Notons $Y$ l'autre bord de $W_1$. Soit $W_2$ le cobordisme correspondant à l'attachement d'une 2-anse le long de $K$, recollée à l'aide du framing spécifié. Notons $Y''$ l'autre bord de $W_2$. Soit $N$ un voisinage réguliers de $S\cup K$ dans $Y'$, $N\simeq (S^2\times S^1)\setminus B^3$, avec $B^3$ une 3-boule, $\partial N$ est une 2-sphère, $Y\simeq Y''\simeq (Y'\setminus N)\cup B^3$,  et $Y'\simeq Y\# (S^2\times S^1)$.

Soit $C$ un cercle de $S$ disjoint de  $p = K\cap S$ et $T\subset N$ le tore correspondant à $C\times K$ dans l'identification $N\simeq (S\times K)\setminus B^3$. Le tore $T$ sépare $N$ en respectivement un voisinage $\nu K$ de $K$, et $N\setminus \nu K$. On notera encore $T$ le tores correspondant dans $Y''$.

Soit $c$ une classe dans $H^2(W_1\cup_{Y'}W_2, \Z{2}) \simeq H^2(Y, \Z{2})$, notons $c_{W_1}$, $c_{W_2}$, $c_{Y}$, $c_{Y''}$ et $c_{Y'}$ les classes induites sur ${W_1}$, ${W_2}$, ${Y}$, $Y''$ et ${Y'}$ respectivement ($c_{Y} = c_{Y''}$ sous l'identification canonique de $H^2( Y, \Z{2})$ et $H^2(Y'' , \Z{2})$  avec  $H^2(Y'\setminus N , \Z{2})$).

On a la décomposition suivante :
\[ H^2(Y', \Z{2}) \simeq H^2(Y, \Z{2}) \oplus H^2(S^2\times S^1, \Z{2}) ,\]
dans laquelle $c_{Y'}$ correspond à $c_Y + 0$. 

Soient $\lambda,\mu\subset T$ des longitues et méridiens de $K$, et $T'$ le tore $T$ éclaté en un point ne rencontrant pas $\lambda$ et $\mu$. Notons, conformément au paragraphe précédent, les deux Lagrangiennes de $\Nc(T')$ suivantes, associées respectivement à $\nu K$ (ou $N\setminus \nu K$),  et au remplissage de Dehn de $Y''$ (ici, $i = j  =0$ car $c_{Y'}$ correspond à $c_Y + 0$) :
\begin{align*}
L_0 =&  \left\lbrace  [A]\in\Nc(T') :  \mathrm{Hol}_\mu A = I \right\rbrace \\
L_1 =&  \left\lbrace  [A]\in\Nc(T') :  \mathrm{Hol}_\lambda A = I \right\rbrace .
\end{align*}
Soit enfin $\underline{L}  = \underline{L}(Y'\setminus N , c )$ la correspondance Lagrangienne généralisée, allant de $pt$ vers $pt$. Les groupes d'homologie HSI des trois variétés sont alors donnés par :

\begin{align*}
HSI(Y, c_{Y}) =& HF(  \underline{L})  \\
HSI(Y', c_{Y'}) =& HF(L_0, L_0,  \underline{L}) \simeq HSI(S^2\times S^1,0) \otimes HSI(Y, c_{Y}) \\
HSI(Y'', c_{Y''}) =&  HF(L_1,L_0, \underline{L}) \simeq HSI(S^3,0) \otimes HSI(Y, c_{Y})\\
 \simeq & HSI(Y, c_{Y}).
\end{align*}

Par construction, si $x\in HF(  \underline{L})$,  $F_{W_1,c_{W_1}}(x) =  \Theta\otimes x$, où $\Theta$ est le générateur de degré 3. Afin de montrer que $F_{W_2,c_{W_2} } \circ F_{W_1,c_{W_1}}= Id_{HSI(Y, c_{Y})}$,  montrons que $F_{W_2,c_{W_2} }(\Theta\otimes x) = x$. 

Rappelons que $F_{W_2,c_{W_2} }$ est définie en comptant des triangles matelassés comme dans la figure \ref{quiltedpants}. Dans le contexte présent, ces triangles sont équivalents à ceux de la figure \ref{trianglebirthdeath} (on peut "découdre" le triangle supérieur): 

\begin{figure}[!h]
    \centering
    \def\svgwidth{\textwidth}
    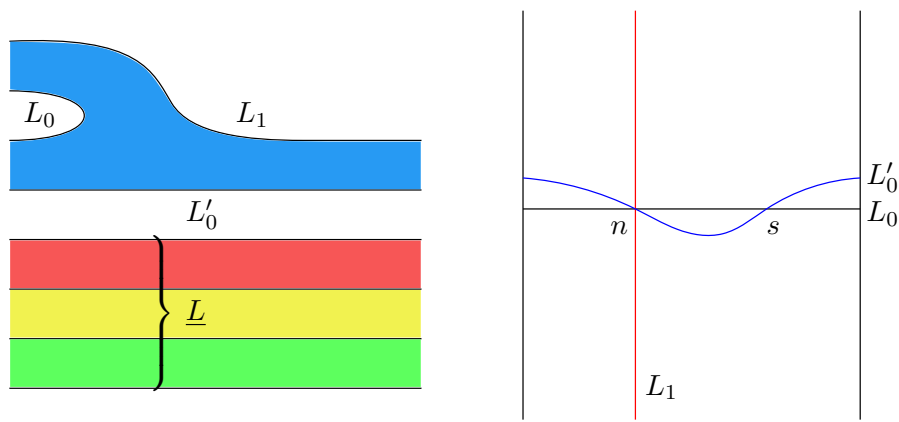
      \caption{Triangle apparaissant dans les coefficients de $CF_{W_2,c_{W_2} }$.}
      \label{trianglebirthdeath}
\end{figure}

Identifions un voisinage tubulaire de $L_0$ dans $\Nc(T')$ à un voisinage $T(\lambda)$ de la section nulle dans $T^* L_0$, tel que $L_1$ corresponde à la fibre au-dessus d'un point $n\in L_0$.  Soit $f\colon L_0\to \rr$ une fonction de Morse avec deux points critiques : un minimum en $s$  et un maximum en $n$. Notons $L_0'$ le graphe de $df$.

L'homologie $HF(L_0, L_0')$ est isomorphe à l'homologie de Morse de $f$, en effet il n'existe pas de bandes de Floer d'indice 1. Le générateur $\Theta$ en degré 3  correspond donc au maximum $n$ de $f$ sous cette identification.

Soient $\underline{x}, \underline{y}\in \I(\underline{L}) $, $x_0 \in L_0\cap L_0' $, $\tilde{x}_0 \in L_0\cap L_1 $, et $y_0 \in L_1\cap L_0'$. Supposons qu'il existe un triangle pseudo-holomorphe contribuant au coefficient \[\left( CF_{W_2,c_{W_2}}( \tilde{x}_0, x_0, \underline{x} ), (y_0, \underline{y} )  \right)\] de l'application au niveau des complexes définissant $F_{W_2,c_{W_2}}$.  D'une part, $\tilde{x}_0 =y_0 = n$,  car les Lagrangiennes s'intersectent en un seul point. 

L'indice total du triangle  et de la bande matelassée est nul : il s'en suit que l'indice du triangle  et de la bande matelassée sont nuls, par généricité. Supposons par l'absurde que $x_0 = s$, on sait qu'il existe un triangle d'indice 1 et d'aire négative valant $ \left( f(s) - f(n)\right)$. Par monotonie, un triangle d'indice nul serait d'aire strictement plus négative, ainsi $x_0 = n$.

Enfin, $ \left( CF_{W_2,c_{W_2}}( n, n, \underline{x} ), (n, \underline{x} )  \right) = 1$ : en effet les trois Lagrangiennes sont concourrantes, ainsi l'espace des modules des triangles matelassés d'indice nul est réduit à un point, correspondant à l'application constante. Par ailleurs, pour des perturbations génériques,  L'opérateur de Cauchy-Riemann linéarisé associé au triangle est injectif d'après \cite[Lemma 2.27]{Seidel}, et donc surjectif car le triangle est d'indice nul. De même, L'opérateur de Cauchy-Riemann linéarisé associé à la bande matelassée est injectif, d'après \cite[Theorem 3.2]{WWerrata}. Ainsi, le triangle matelassé constant intervenant ici est régulier.

\subsection{Interversion de points critiques}

Soit $Y$ une 3-variété. Supposons que $W_1$, un cobordisme de $Y$ vers $Y_1$,  corresponde à l'attachement d'une anse $h_1$ à $Y$, et que $W_2$, un cobordisme de $Y_1$ vers $Y_{12}$, corresponde à l'attachement d'une autre anse $h_2$ à $Y_1$ disjointe de $h_1$. Supposons que $W_2'$ et $W_1'$ correspondent à l'attachement dans l'ordre inverse, de sorte à ce que $W_1\cup_{Y_1} W_2 = W_2' \cup_{Y_2} W_1'$. On note $Y_2$ la 3-variété entre  $W_2'$ et $ W_1'$, et $Y_{21}$ l'autre bord de $W_1'$, comme résumé dans la figure \ref{interversionptcrit}. Nous allons montrer que $F_{W_2,c_{W_2}} \circ F_{W_1,c_{W_1}} = F_{W_1',c_{W_1'}} \circ F_{W_2',c_{W_2'}}$. Quatre cas sont à distinguer selon la dimension des anses : 
\begin{enumerate}
\item $h_1$ est une 1-anse et $h_2$ est une 1-anse.
\item $h_1$ est une 1-anse et $h_2$ est une 2-anse.
\item $h_1$ est une 2-anse et $h_2$ est une 2-anse.
\item $h_1$ est une 1-anse et $h_2$ est une 3-anse.
\end{enumerate}
Le cas restant (2-anse/3-anse) se déduit du second cas.

\begin{figure}[!h]
    \centering
    \def\svgwidth{.65\textwidth}
   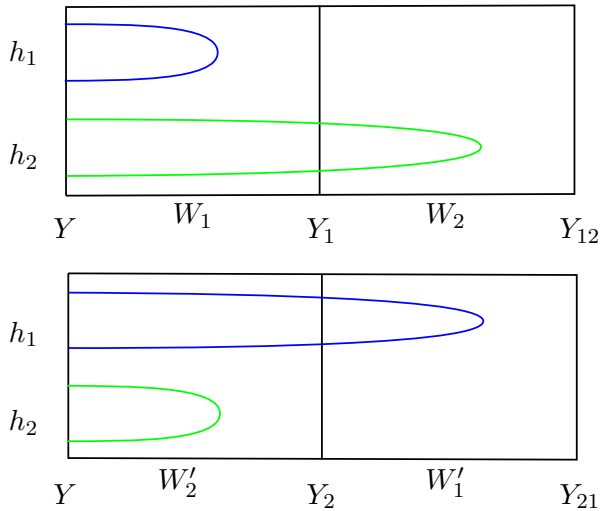
      \caption{Interversion de points critiques.}
      \label{interversionptcrit}
\end{figure}

\paragraph{Premier cas ($h_1$ est une 1-anse et $h_2$ est une 1-anse.) :}   Notons $A_1$ et $A_2$ des variétés homéomorphes à $S^2\times S^1$, de sorte que, avec $i=1$ ou 2 et $j= 3-i$, $Y_i =  A_i\#Y$ et  $Y_{ij} = A_j\#(A_i\#Y  ) $. Notons $H_Y = HSI(Y,c)$ et $H_i = HSI(A_i,0)$ , de sorte que, d'après la formule de Künneth, $HSI(Y_i,c_i) =H_i \otimes H_Y $ et $HSI(Y_{ij}) = H_j\otimes H_i \otimes H_Y$. Notons enfin $\Theta_i \in H_i$ les générateurs de degrés 3. Alors $ F_{W_2}\circ F_{W_1} \colon H_Y \to H_Y\otimes H_1 \otimes H_2$ et $ F_{W_1}'\circ F_{W_2}' \colon H_Y \to H_Y\otimes H_2 \otimes H_1$ sont données, pour $x\in H_Y$,  par : 
\begin{align*}
 F_{W_2}\circ F_{W_1}(x) =& \Theta_2\otimes \Theta_1 \otimes x \\
 F_{W_1}'\circ F_{W_2}'(x) =& \Theta_1\otimes \Theta_2 \otimes x .
\end{align*}
Elles sont donc identifiées via l'isomorphisme $HSI(Y_{12}) \simeq HSI(Y_{21}) $ induit par l'identité.

\paragraph{Second cas ($h_1$ est une 1-anse et $h_2$ est une 2-anse.) :} En reprenant les notations précédentes, $HSI(Y_1) = H_1\otimes HSI(Y) $ et $HSI(Y_{21}) =  H_1 \otimes HSI(Y_2)$.

L'assertion vient alors du fait que $F_{W_2} = Id_{H_1}\otimes F_{W_2'}  $. En effet, la dernière bande de la surface matelassée définissant $F_{W_2}$ peut être "découpée", voir la figure \ref{switch1anse2anse}.

\begin{figure}[!h]
    \centering
    \def\svgwidth{.65\textwidth}
    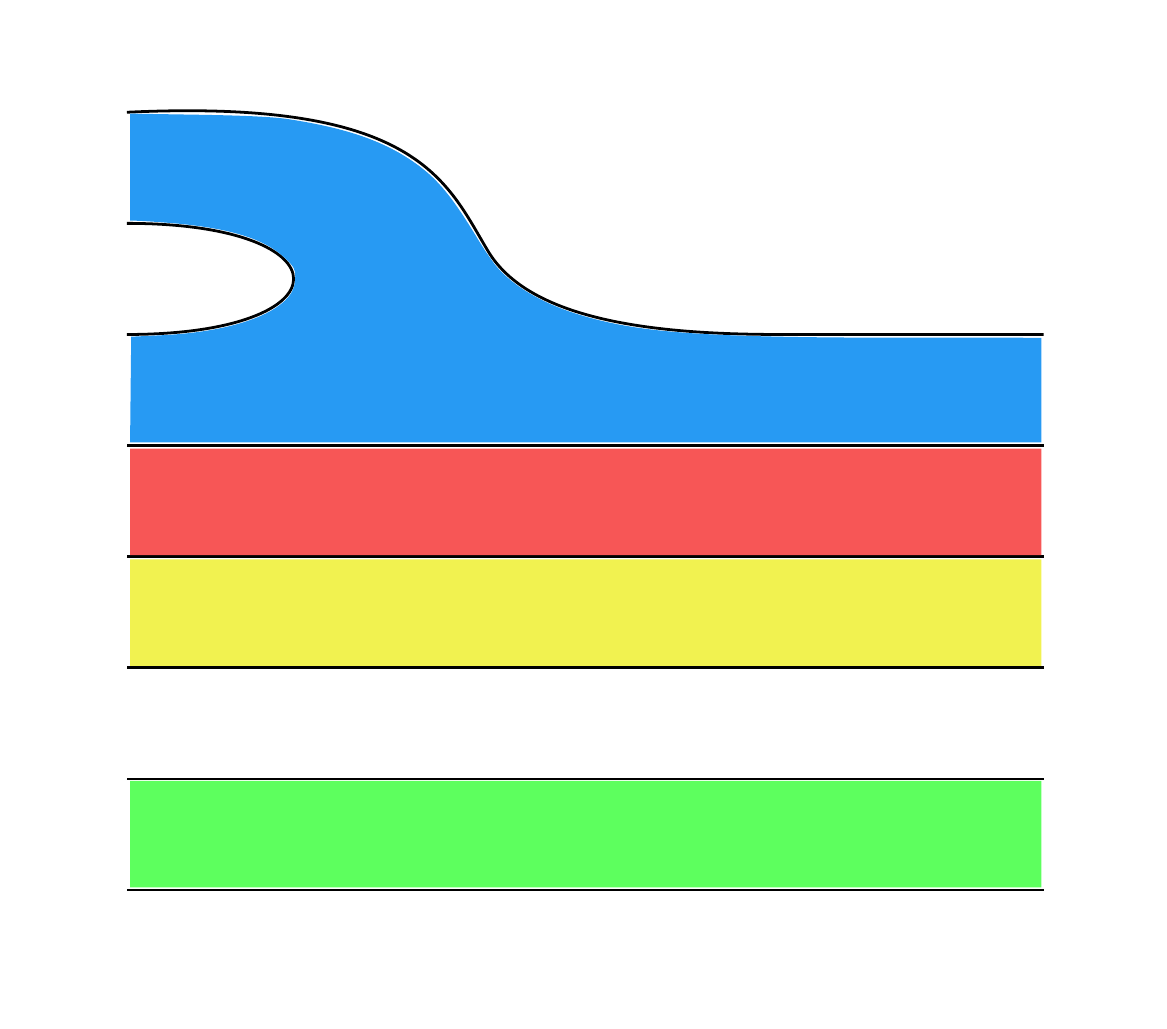
      \caption{Interversion d'une 1-anse et d'une 2-anse.}
      \label{switch1anse2anse}
\end{figure}

\paragraph{Troisième cas ($h_1$ est une 2-anse et $h_2$ est une 2-anse.) :}

Soient $K_1, K_2 \subset Y$ deux noeuds munis de framings, auxquels on attache les anses $h_1$ et $h_2$ respectivement. Soient $N_1$ et $N_2$ des voisinages tubulaires, $T_1$ et $T_2$ leurs bords, $M = Y\setminus \left( K_1 \cup K_2 \right)$ leur complémentaire, vu comme un cobordisme de $T_1$ vers $T_2$. 

On note $\underline{L}$ la correspondance Lagrangienne généralisée de $\Nc(T_1)$ vers $\Nc(T_2)$ associée à $M$, $L_1\subset\Nc(T_1)$ et $L_2 \subset\Nc(T_2)$ les Lagrangiennes correspondant à $N_1$ et $N_2$, et $L_1'\subset\Nc(T_1)$ et $L_2' \subset\Nc(T_2)$ les Lagrangiennes correspondant aux deux remplissages, de sorte que :
\begin{align*}
HSI(Y) &= HF(L_1, \underline{L}, L_2), \\
HSI(Y_1) &= HF(L_1', \underline{L}, L_2), \\
HSI(Y_2) &= HF(L_1, \underline{L}, L_2'), \\
HSI(Y_{12}) &= HF(L_1', \underline{L}, L_2').
\end{align*}
Les applications associées à $W_1$, $W_2$, $W_1'$ et $W_2'$ sont alors définies à l'aide de triangles matelassés comme dans la figure \ref{switch2anses}. Cela par définition pour $F_{W_1, c_1}$ et $F_{W_1',c_1'}$, et concernant $F_{W_2, c_2}$ et $F_{W_2',c_2'}$ par les deux observations suivantes : d'une part une variété orientée $Y$ décrite par une décomposition en anses 
\[ Y = Y_1\cup_{\Sigma_1} Y_2 \cup_{\Sigma_2} \cdots \cup_{\Sigma_{k-1}} Y_k \]
peut également être décrite par la décomposition renversée
\[ Y = Y_k\cup_{\overline{\Sigma}_{k-1}}\cdots \cup_{\overline{\Sigma}_2} Y_2 \cup_{\overline{\Sigma}_1} Y_1  ,\]
où les surfaces sont munies de leur orientations opposées. Les espaces des modules qui leur correspondent sont alors munis de la forme symplectique opposée. Par ailleurs un quilt pseudo-holommorphe dans une famille de variétés symplectiques correspond à son image mirroir dans la famille des variétés symplectiques opposées. 

Les  composées $F_{W_2, c_2}\circ F_{W_1, c_1}$ et $F_{W_1',c_1'}\circ F_{W_2',c_2'}$ sont ainsi associées aux surfaces recollées correspondantes : elles coïncident donc en vertu de l'homotopie suggérée dans la figure \ref{switch2anses}.

%

\begin{figure}[!h]
    \centering
    \def\svgwidth{\textwidth}
    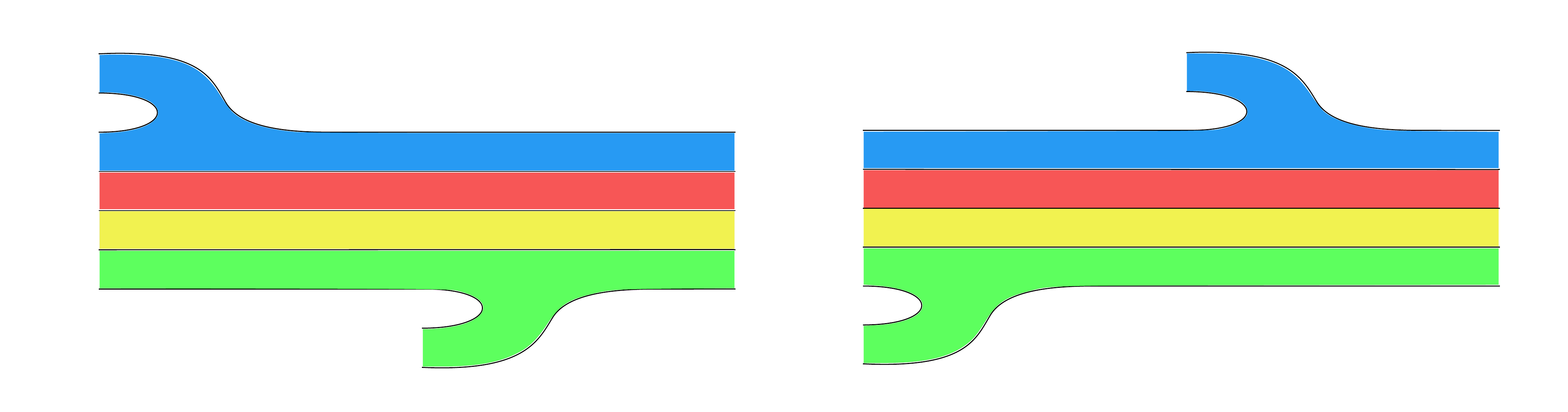
      \caption{Interversion de 2-anses.}
      \label{switch2anses}
\end{figure}

\paragraph{Quatrième cas ($h_1$ est une 1-anse et $h_2$ est une 3-anse.) :}

De même que pour le premier cas, notons $A_1$ et $A_2$ des variétés homéomorphes à $S^2\times S^1$, de sorte que 
\begin{align*}
 Y &\simeq A_2\#Y_2, \\
 Y_1 &\simeq A_1 \# Y\simeq A_1\# A_2\# Y_2, \\
 Y_{12} &\simeq Y_{21}\simeq A_1 \# Y_2 .
\end{align*}
Notons $H_{Y_2} = HSI(Y_2)$, $H_i = HSI(A_i)$, et $\Theta_i, \Theta_i' \in H_i$ les générateurs  de degrés respectifs 3 et 0.

Vérifions que $F_{W_2}\circ F_{W_1} = F_{W'_1}\circ F_{W'_2}$ : Soit $x\in H_{Y_2}$, on a d'une part : 
\begin{align*} F_{W_2}\circ F_{W_1} (\Theta_2\otimes x) &= F_{W_2} (\Theta_1 \otimes \Theta_2\otimes x) \\
&= \Theta_1 \otimes x \\
&=F_{W'_1}( x ) \\
&=F_{W'_1}\circ F_{W'_2}(\Theta_2 \otimes x ),
\end{align*}
et d'autre part :
\[F_{W_2}\circ F_{W_1} (\Theta_2'\otimes x) = F_{W'_1}\circ F_{W'_2}(\Theta_2'\otimes x) = 0 ,\]
ce qui achève la preuve de l'indépendance de la décomposition.
\arnaque

Ainsi les applications par cobordismes $F_{W,c}$ sont bien définies. La formule de composition suivante découle immédiatement de leur construction  :

\begin{prop}\label{compohoriz}(Composition horizontale) Soit $W = W_1 \cup W_2$ la composition de deux cobordismes $W_1$ et $W_2$, $c\in H^2(W,\Z{2})$ une classe, et  $c_1$, $c_2$ les classes induites sur $W_1$ et $W_2$. Alors,
\[ F_{W,c} = F_{W_2,c_2}\circ F_{W_1,c_1}. \]
\end{prop}
\arnaque

\section{Flèches dans le triangle de chirurgie}

Dans cette section nous prouvons que les analogues modulo 2 de deux des flèches intervenant dans la suite exacte du théorème \ref{trianglechir} peuvent s'interpréter comme des applications induites par cobordismes.

Soient, comme dans le théorème \ref{trianglechir}, $Y$ une 3-variété à bord torique, $\alpha, \beta,\gamma$ trois courbes dans $\partial Y$ vérifiant $ \alpha. \beta = \beta.\gamma = \gamma.\alpha = -1$,  et $Y_\alpha$,  $Y_\beta$, $Y_\gamma$ les variétés obtenues par remplissage de Dehn. Si $\delta \neq \mu \in \lbrace \alpha, \beta, \gamma\rbrace$, soit $W_{\delta, \mu }$ le cobordisme de $Y_\delta$ vers $Y_\mu$ correspondant à l'attachement d'une 2-anse le long de $\delta$ avec framing $\mu$.

Soit $c\in H_1(Y;\Z{2})$, on note pour $\delta \neq \mu \in \lbrace \alpha, \beta, \gamma\rbrace$, les trois classes $c_\delta \in H_1(Y_\delta ;\Z{2})$ induites par l'inclusion, et $c_{\delta \mu} \in H_2(W_{\delta\mu}, \partial W_{\delta\mu} ;\Z{2})$ les classes induites par $c\times [0,1] \in H_2(Y\times [0,1], Y\times \lbrace 0,1\rbrace ;\Z{2} )$ via l'inclusion de $Y\times [0,1]$ dans $W_{\delta\mu}$. Notons également $d_\alpha \in H_2(W_{\alpha\beta}, \partial W_{\alpha\beta} ;\Z{2} )$ la classe fondamentale de l'âme de l'anse attachée le long de $\alpha$, de sorte que $\partial_* d_\alpha = k_\alpha \in H_1(Y_\alpha;\Z{2})$.


\begin{theo}\label{thinterpfleches} Soient $C\Phi_1$ et $C\Phi_2$ les deux morphismes de la suite exacte du théorème \ref{trianglechir} que l'on a construit dans la partie \ref{flechestri}. Les morphismes induits en homologie à coefficients dans $\Z{2}$ coïncident avec les application $F_{W_{\alpha\beta}, c_{\alpha\beta} + d_\alpha}$ et $F_{W_{\beta\gamma}, c_{\beta\gamma} }$.
\end{theo}

Le troisième morphisme de la suite exacte, construit comme un homomorphisme de connexion, n'est pas à priori induit par $W_{\gamma\alpha}$. En revanche, la même démonstration que celle utilisée par Lisca et Stipsicz pour l'homologie de Heegaard-Floer \cite[Section 2, fin]{LiscaStipsicz} permet d'obtenir : 
\begin{cor}\label{corinterpfleches} En notant $d'_\alpha \in H_2(W_{\gamma\alpha}, \partial W_{\gamma\alpha} ;\Z{2} )$ la classe fondamentale de l'anse attachée le long de $\alpha$, les applications $F_{W_{\alpha\beta}, c_{\alpha\beta} + d_\alpha}$, $F_{W_{\beta\gamma}, c_{\beta\gamma} }$, et $F_{W_{\gamma\alpha}, c_{\gamma\alpha} + d'_\alpha}$ forment une suite exacte :
\[ \xymatrix{  HSI(Y_\beta,c_\beta)\ar[rr]^{F_{W_{\beta\gamma}, c_{\beta\gamma} }} & & HSI(Y_\gamma,c_\gamma)\ar[ld]^{\ \ \ \ \ F_{W_{\gamma\alpha}, c_{\gamma\alpha} + d'_\alpha}}  \\ & HSI(Y_\alpha,c_\alpha +k_\alpha)\ar[lu]^{F_{W_{\alpha\beta}, c_{\alpha\beta} + d_\alpha}}  & ,}\]
où les groupes sont à coefficients dans $\Z{2}$.
\end{cor}
\begin{proof}[Preuve du corollaire \ref{corinterpfleches}] On utilise la symétrie cyclique de la triade de chirurgie. On notera $L_\delta^\epsilon = \lbrace\mathrm{Hol}_\delta = \epsilon I \rbrace$. Rappelons au préalable que pour démontrer le théorème \ref{trianglechir}, on a appliqué la suite exacte \ref{quilttri} à $L_0 = L_\gamma$ et $S=L_\alpha^-$. Il s'en suit que $\mathrm{Ker}(F_{W_{\beta\gamma}, c_{\beta\gamma} }) = \mathrm{Im}(F_{W_{\alpha\beta}, c_{\alpha\beta} + d_\alpha})$.

Si l'on procède de même avec $L_0 = L_\alpha^-$ et $S= L_\beta$, on obtient une suite exacte similaire faisant intervenir $F_{W_{\beta\gamma}, c_{\beta\gamma} }$ et $F_{W_{\gamma\alpha}, c_{\gamma\alpha} + d'_\alpha}$.  Il s'en suit que $\mathrm{Ker}(F_{W_{\gamma\alpha}, c_{\gamma\alpha} + d'_\alpha}) = \mathrm{Im}(F_{W_{\beta\gamma}, c_{\beta\gamma} })$.

Enfin, le choix $L_0 = L_\beta$ et $S=L_\gamma$ permet de prouver  $\mathrm{Ker}(F_{W_{\alpha\beta}, c_{\alpha\beta} + d_\alpha}) = \mathrm{Im}(F_{W_{\gamma\alpha}, c_{\gamma\alpha} + d'_\alpha})$, d'où le résultat annoncé.
\end{proof}

\begin{proof}[Preuve du théorème]
D'une part, $\Phi_1 = F_{W_{\alpha\beta}, c_{\alpha\beta} + d_\alpha}$  par définition de l'application $F_{W_{\alpha\beta}, c_{\alpha\beta} + d_\alpha}$.

La seconde égalité s'obtient par un argument similaire à celui de la section \ref{homotopie}, en "poussant" la valeur critique sur le bord supérieur de la surface matelassée (voir figure \ref{contrac}). Il s'en suit que $C\Phi_2$ est homotope à la contraction du produit en pantalon $CF(L_0, \tau_S L_0 )\otimes CF(\tau_S L_0 , \underline{L}) \to CF( L_0 , \underline{L})$  par le cocycle $ c_{S,L_0} \in CF(L_0, \tau_S L_0 )$ défini par la fibration de Lefschetz spécifiée dans la figure.

Il reste à remarquer que $ c_{S,L_0}$ coïncide avec le générateur $C$ utilisé pour définir $F_{W_{\beta\gamma}, c_{\beta\gamma}}$. Ceci vient du fait que $L_0$ et  $\tau_S L_0$ s'intersectent en un seul point $x$, pour lequel il existe une unique section d'indice 0, la section "constante", qui est régulière d'après \cite[Lemma 2.27]{Seidel}. En effet, par monotonie, toute autre section d'indice nul aurait une aire nulle.
\end{proof}

\begin{figure}[!h]
    \centering
    \def\svgwidth{\textwidth}
    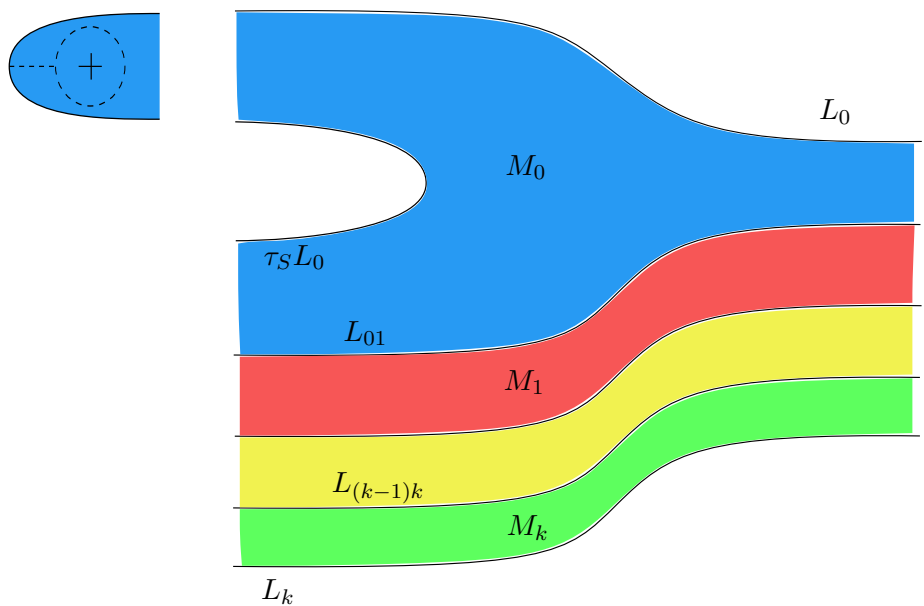
      \caption{Contraction d'un pantalon matelassé avec un cocycle.}
      \label{contrac}
\end{figure}

\section{Exemples, propriétés élémentaires}

\begin{prop}\label{cépédeux} Soit $W = \cc P^2 \setminus \lbrace \text{deux boules} \rbrace$, on note 0 et 1 les deux classes de $H^2(W;\Z{2}) \simeq \Z{2}$.
\begin{enumerate}
\item $F_{W,0} = F_{W,1} = 0 $.
\item En notant $\overline{W}$ le cobordisme muni de l'orientation opposée, $F_{\overline{W},0}$ est un isomorphisme, et $F_{\overline{W},1} = 0 $.
\end{enumerate}
\end{prop}

\begin{proof}
Rappelons que $W$ correspond à l'attachement d'une 2-anse à $S^3$ le long du noeud trivial, avec framing 1.
\begin{enumerate}
\item  La suite exacte de chirurgie appliquée à la triade correspondant à la chirurgie $\infty$, 1,2  sur le n{\oe}ud trivial est de la forme, quelles que soient les classes $i = 0,1$ et en notant $S^3_\alpha$ la chirurgie $\alpha$ sur le noeud trivial :
\[ \xymatrix{  HSI(S^3_{1})  = \Z{2} \ar[rr] & & HSI(S^3_{2}) = \Z{2}^2 \ar[ld] \\ & {HSI(S^3 _\infty ) = \Z{2}} \ar[lu]^{F_{W,i} }  & ,  } \]
ce qui entraîne le résultat.

\item On considère cette fois la triade correspondant à la chirurgie $\infty$, -1,0  sur le n{\oe}ud trivial. Si la classe sur $S^3_0$ est non-nulle, il vient : 
\[ \xymatrix{  HSI(S^3_{-1})  = \Z{2} \ar[rr] & & HSI(S^3_{0},1) = 0 \ar[ld] \\ & {HSI(S^3 _\infty ) = \Z{2}} \ar[lu]^{F_{\overline{W},0} }  & ,  } \]
et dans les deux autres cas, on obtient :
\[ \xymatrix{  HSI(S^3_{-1})  = \Z{2} \ar[rr] & & HSI(S^3_{0},0) = \Z{2}^2 \ar[ld] \\ & {HSI(S^3 _\infty ) = \Z{2}} \ar[lu]^{F_{\overline{W},1} }  & .  } \]
\end{enumerate}

\end{proof}

Introduisons l'invariant suivant pour des 4-variétés closes :
\begin{defi}Soit $X$ une 4-variété close, et $c\in H^2(X;\Z{2})$. Notons $W$ la variété $X$ privée de deux boules ouvertes, c'est un cobordisme de $S^3$ vers $S^3$. Son application associée $F_{W,c}\colon HSI(S^3) \to HSI(S^3)$ est une multiplication par un nombre que l'on note $\Psi_{X,c}$. 
\end{defi}

\begin{remark} Cette construction n'est pas tout à fait analogue à celle de l'invariant $\Phi_{X,\mathfrak{s}}$ en théorie d'Heegaard-Floer, voir \cite{OSholotri}.
\end{remark}
\begin{remark} Contrairement au cas d'un cobordisme $W$, un $SO(3)$-fibré principal au-dessus de $X$ peut éventuellement avoir une première classe de Pontryagin non-triviale. Ces fibrés ne seront pas pris en compte par cet invariant.
\end{remark}

La proposition \ref{compovertic} suivante décrit l'effet d'une composition "verticale" de cobordismes (composition correspondant au second type de compositions dans la 2-catégorie $\textbf{Cob}_{2+1+1}$).
\begin{defi}Soient $W$ et $W'$ deux cobordismes, de $Y_1$ vers $Y_2$, et de $Y_1'$ vers $Y_2'$ respectivement, et $l$, $l'$ deux chemins dans $W$ et $W'$ reliant les deux bords. On peut former leur \emph{composition verticale} en retirant des voisinages tubulaires de $l$ et $l'$, et en recollant les deux morceaux restants le long des bords des parties retirées. On obtient un nouveau cobordisme $W\#_{vert}W'$ de $Y_1\# Y_1'$ vers $Y_2\# Y_2' $.
\end{defi}

\begin{prop}\label{compovertic} (composition verticale)
Soit $W\#_{vert}W'$ une composition verticale comme précédemment. Alors, sous les identifications suivantes (rappelons que les groupes sont à coefficients dans $\Z{2}$, d'où l'absence de termes de torsion) :
\begin{align*}
 HSI (Y_1\# Y_1', c_1 +c_1') &= HSI (Y_1, c_1 ) \otimes HSI (Y_1',c_1') \\ HSI (Y_2\# Y_2', c_2 +c_2') &= HSI (Y_2, c_2) \otimes HSI (Y_2', c_2'),\\
\end{align*}
L'application associée à $W\#_{vert}W'$ est donnée par :
\[ F_{W\#_{vert}W',c+c'} = F_{W,c}\otimes F_{W',c'} .\]
En particulier, si $X$ est une variété fermée et $W\#X$ désigne la somme connexe en un point intérieur, 
\[ F_{W\#X,c_W + c_X} = \Psi_{X,c_X} \cdot F_{W,c_W}. \]
\end{prop}

\begin{proof}Soient deux décompositions en cobordismes élémentaires de $W$ et $W'$ de même longueur, elles induisent une décomposition en cobordismes élémentaires de $W\#_{vert}W'$. Le fait qu'il y ait un point au milieu des correspondances Lagrangiennes généralisées garantit que les applications induites par les anses de $W$ n'intéragissent pas avec celles de $W'$, ainsi les applications induites au niveau des complexes de chaînes sont de la forme $CF_{W}\otimes CF_{W'}$, et donc de la forme indiquée au niveau des groupes d'homologie d'après la naturalité de la formule de Künneth (voir \cite[Theorem 11.10.2]{tomDieck}).
\end{proof}

Combinée à la propositon \ref{cépédeux}, il vient alors :

\begin{cor}\label{éclatement} 
Soit $W$ un 4-cobordisme, 
\begin{enumerate}
\item $F_{W\# \cc P^2,c} =  0 $, pour toute classe $c\in H^2(W\# \cc P^2;\Z{2})$.
\item Si $c\in H^2(W\# \overline{\cc P}^2;\Z{2})$ est non-nulle en restriction à $\overline{\cc P}^2$, alors $F_{W\# \overline{\cc P}^2,c} =  0 $, sinon $F_{W\# \overline{\cc P}^2,c} =  F_{W,c_{|W}} $.
\end{enumerate}
\end{cor}

\begin{remark}On observe ici une  différence avec l'homologie d'Heegaard-Floer : si $(Y_\alpha, Y_\beta, Y_\gamma)$ est une triade de chirurgie et $W_{\delta \gamma}$ désignent les 4-cobordismes d'attachement d'anses, alors les trois morphismes associés à $W_{\delta \gamma}$ munis des classes nulles ne forment pas une suite exacte. En effet, pour la triade $(S^3_0, S^3_1, S^3_\infty)$ associée au noeud trivial dans la sphère, le cobordisme allant de $S^3_1$ vers $S^3_\infty$ induit un isomorphisme, mais $HSI(S^3_0) = HSI(S^2\times S^1)$ est de rang 2.
\end{remark}

\chapter{Perspectives}\label{chappersp}

\section{Naturalité}\label{sectionnaturalite}

Dans cette section, on s'intéresse au problème de la naturalité de l'homologie Instanton-Symplectique, à savoir, si le groupe abélien $\Z{8}$-relativement gradué $HSI(Y,c,z)$ est bien défini en tant que tel, ou seulement à isomorphisme près.

La preuve de la naturalité des groupes d'homologie de Heegaard-Floer est délicate : c'est un travail récent de Juhasz et Thurston dans \cite{JuhaszThurston} qui consiste à construire un système transitif de groupes sur l'ensemble des diagrammes  de Heegaard. La naturalité de l'homologie Instanton-Symplectique semble néanmoins plus simple à vérifier, car sa construction utilise la notion de scindement de Heegaard (ou de fonction de Morse), et non celle de  diagramme de Heegaard.

Dès lors que l'on a construit un foncteur $\Cob \to \Symp$, si l'on note $Gr$ la classe des groupes abéliens $\Z{8}$-relativement gradués, il suffit d'avoir une application convenable $HF\colon Hom_{\Symp}(pt,pt) \to Gr$ bien définie pour garantir la naturalité des invariants pour les 3-variétés.

L'ensemble $Hom_{\Symp}(pt,pt)$ peut être vu comme l'ensemble des classes d'isomorphismes du groupoïde suivant : soit $\mathbf{G}$ le groupoïde dont les objets sont les correspondances Lagrangiennes généralisées allant de $pt$ vers $pt$ (sujettes aux mêmes conditions que les objets de $Hom_{\Symp}(pt,pt)$) et les morphismes sont engendrés par la famille suivante : si $\underline{L} = \left(  L_{01}, \cdots\right)$, avec un indice $i$ pour lequel la composition  $L_{(i-1)i} \circ L_{i(i+1)}$ est plongée (et satisfait aux mêmes hypothèses que pour la relation d'équivalence de $\Symp$), on définit un morphisme $f_{\underline{L}, i}\colon \underline{L} \to \circ_i (\underline{L})$, où $\circ_i (\underline{L})$ désigne la correspondance $\left( \cdots , L_{(i-1)i} \circ L_{i(i+1)} , \cdots \right)  $.

Il s'agit de définir un foncteur $\mathbf{HF}\colon \mathbf{G} \to \mathbf{Gr}$ vers la catégorie des groupes abéliens relativement $\Z{8}$-gradués tel que  $\mathbf{HF}$ coïncide avec l'homologie de Floer matelassée au niveau des objets, et tel que tout endomorphisme de $\underline{L}$ soit envoyé sur l'identité de $HF(\underline{L})$.

Une possibilité pour construire un tel foncteur est de concaténer les "applications $Y$" considérées par Lekili et Lipyanskiy. Cela correspond à l'invariant relatif associé à une surface matelassée comme décrite dans la figure \ref{feuilledoignon}. Le problème est alors de vérifier qu'une succession de telles applications  allant de  $\underline{L}$ vers  $\underline{L}$ est l'identité.

Notons que le même raisonnement utilisé par Lekili et Lipyanskiy dans la preuve du fait  que les composées  $\Phi \circ \Psi$ et $\Psi \circ \Phi$ diffèrent de $Id_{HF(\underline{L})}$ et $Id_{HF(\underline{L'})}$ par des applications nilpotentes, voir \cite[Section 3.1]{LekiliLipyanskiy}, permet de montrer que la succession de telles "applications $Y$" est unipotente. 

On pourrait essayer d'utiliser un argument similaire au "strip-shrinking" de \cite{WWcompo} pour  démontrer un tel énoncé. Cela consisterait à "écraser" horizontalement la partie centrale de la surface matelassée de la figure \ref{feuilledoignon}, c'est-à-dire à faire tendre vers 0 la longueur $L$ de la partie centrale de la figure \ref{feuilledoignon}. A priori, ceci pourrait créer des bulles matelassées semblables aux "witch balls" considérées par Bottman et Wehrheim dans \cite{Bottman} et \cite{BottmanWehrheim}, comme dessinées dans la figure \ref{witchbubble},  qui sont des généralisations des bulles en figure 8. Néanmoins il semble que dans notre cadre, de telles bulles n'existent pas génériquement, en effet les hypothèses de monotonie et de grand indice de Maslov minimal, combinées au raisonnement  intervenant dans la preuve du lemme \ref{nobubbling}  sur la transversalité des disques pseudo-holomorphes avec les hypersurfaces, semblent permettre d'exclure tout bubbling de ce type. 

\begin{figure}[!h]
    \centering
    \def\svgwidth{.80\textwidth}
    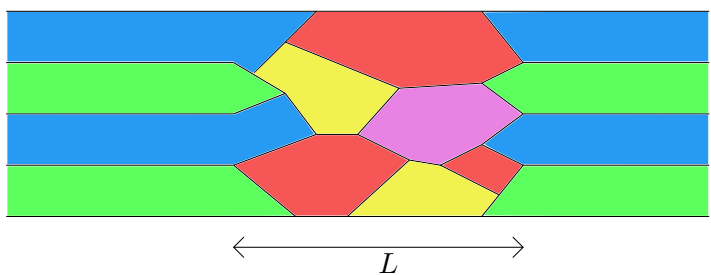
      \caption{Une surface matelassée associée à un endomorphisme de $\underline{L}$.}
      \label{feuilledoignon}
\end{figure}

\begin{figure}[!h]
    \centering
    \def\svgwidth{.80\textwidth}
    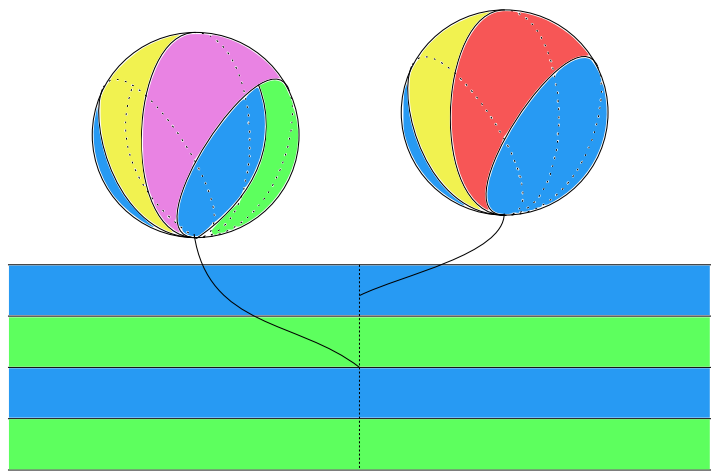
      \caption{Des "witch bubble" apparaissant lorsque $L\to 0$.}
      \label{witchbubble}
\end{figure}

\section{Un analogue des versions +,- et $\infty$ de l'homologie d'Heegaard-Floer}\label{sec:autresversions}

La construction de l'homologie $HSI$ par Manolescu et Woodward est formellement similaire à celle de la version $\widehat{HF}$ de l'homologie de Heegaard-Floer : partant d'un scindement de Heegaard $Y = H_0 \cup_\Sigma H_1$  d'une 3-variété $Y$, on associe à $\Sigma$ munie d'un point base $z$ une variété symplectique munie d'une hypersurface, et à chaque corps à anse une Lagrangienne. L'homologie est alors construite en comptant des disques de Whitney pseudo-holomorphes n'intersectant pas l'hypersurface. 

En revanche, les complexes de chaînes définissant les trois autres versions $HF^+$, $HF^-$ et $HF^\infty$ prennent en compte tous les disques de Whitney et capturent leur nombre d'intersection avec l'hypersurface. Il est alors naturel de vouloir faire de même pour l'homologie HSI, cela était d'ailleurs suggéré dans \cite[Section 2.2]{MW}. Ces raffinements présentent plusieurs intérets : d'une part les groupes $HF^+$ et  $HF^-$ contiennent plus d'informations sur les 3-variétés que $\widehat{HF}$, ce sont des $\zz[U]$-modules tels que $\widehat{HF}$ correspond au cône de l'application $U$ de $HF^+$. A l'opposé, le groupe $HF^\infty$ est plus simple : il est calculé dans \cite[Theorem 10.1]{OSzholodisk2} en fonction de données purement homologiques de $Y$, mais permet cependant de définir une graduation absolue sur les autres groupes. 
Une question naturelle est de savoir si des groupes analogues existent pour l'homologie HSI. Nous montrons que c'est le cas ci-dessous.

On se place dans le cadre de la section \ref{defhomolgiematelassée} : Soient $\underline{L} = (L_{01}, L_{12}, \cdots ) $ un morphisme de $\Symp$ de $pt$ vers $pt$, $\underline{J}\in \mathcal{J}_t(M_i, \mathrm{int}\left\lbrace \omega_i = \tilde{\omega}_i \right\rbrace ,  \tilde{J}_i)  $ et $\underline{H}$ des Hamiltoniens tels que l'intersection $\tilde{L}_{(0)} \cap \tilde{L}_{(1)}$ des produits alternés des correspondances est transverse. Notons $\widetilde{\mathcal{M}}(\underline{x},\underline{y})^k_j$ l'ensemble des applications $u_i\colon \rr \times [0,1]\to M_i$ telles que, avec $s\in\rr$ et $t\in [0,1]$ :

\[ \begin{cases} 0 = \partial_s u_i + J_t ( \partial_t u_i - X_{H_i}) \\
\mathrm{lim}_{s\rightarrow - \infty}{u_i(s,t)} = \varphi_i^t (y_i)\\
\mathrm{lim}_{s\rightarrow + \infty}{u_i(s,t)} = \varphi_i^t (x_i)\\
(u_i(s,1), u_{i+1}(s,0) )\in L_{i(i+1)} \\
\underline{u} \cdot \underline{R} = k \\
I(\underline{u}) = j.
\end{cases} \]


Posons  $\mathcal{M}(\underline{x},\underline{y})^k_j = \widetilde{\mathcal{M}}(\underline{x},\underline{y})^k_1 /\rr$. Pour des choix génériques de $\underline{J}$ et $\underline{H}$, $\mathcal{M}(\underline{x},\underline{y})^k_1$   est une variété compacte de dimension 0, on définit alors  $\partial_k \colon CF(\underline{L};\Z{2}) \to CF(\underline{L};\Z{2}) $ par $\partial_k ( \underline{x} )= \sum_{\underline{y}} {\# \mathcal{M}(\underline{x},\underline{y})^k_1 \cdot \underline{y}}$, de sorte que $\partial_0$ correspond avec la différentielle définissant l'homologie matelassée à coefficients dans $\Z{2}$ $HF(\underline{L};\Z{2})$.

\begin{remark}On se limite ici à l'homologie à coefficients dans $\Z{2}$ car les correspondances Lagrangiennes compactifiées ne sont pas spin à priori.
\end{remark}

\begin{lemma} Les applications $\partial_k$  vérifient $\sum_{i+j = k}{\partial_i\circ \partial_j} = 0$ pour tout $k\geq 0$. \end{lemma}

\begin{proof}On peut ici reproduire le raisonnement de la preuve du lemme \ref{nobubbling} : 

On cherche à comprendre la compactification de l'espace des modules $\widetilde{\mathcal{M}}_2^k$ des trajectoires de Floer d'indice 2 dont l'intersection avec $R$ vaut $k$. Son bord contient le produit $\bigcup_{i+j = k}{  \widetilde{\mathcal{M}}_1^i\times \widetilde{\mathcal{M}}_1^j}$,  montrons qu'il ne se produit pas d'autre type de dégénérescences : soit $u_\infty$ une trajectoire stable dans la compactification de $\widetilde{\mathcal{M}}_2^k$ : elle consiste à priori en une composante principale brisée $v_\infty$, à laquelle sont attachés des disques et des bulles. On peut de plus supposer que $v_\infty$ intersecte $R$ transversalement. Pour les mêmes raisons que dans \ref{nobubbling} (monotonie et grand indice de Maslov minimal), les disques et bulles sont d'aire nulle, donc il n'y a pas de disques, et les bulles sont contenues dans $R$ : elles ont chacune un nombre d'intersection avec $R$ inférieur ou égal à $-2$. Il s'en suit que s'il y a au moins une bulle, $v_\infty$ intersecte $R$ en un nombre de points  strictement supérieur à $k$, ce qui est impossible pour une limite de courbes de $\widetilde{\mathcal{M}}_2^k$.

\end{proof}

Ainsi, les $\Z{2} [U]$-modules suivants
\begin{align*} CF^\infty(\underline{L};\Z{2}) &= CF(\underline{L};\Z{2})\otimes_{\Z{2}} \Z{2} [U,U^-1] \\CF^+(\underline{L};\Z{2}) &= CF(\underline{L};\Z{2})\otimes_{\Z{2}} \Z{2} [U,U^-1]/U\Z{2}[U] \\CF^-(\underline{L};\Z{2}) &= CF(\underline{L};\Z{2})\otimes_{\Z{2}} \Z{2} [U] ,
\end{align*} 
sont équippés respectivement des différentielles $\partial^\infty$, $\partial^+$ et $\partial^-$, ayant pour expression commune $ \sum_{i\geq 0}{U^k \partial_k}$. On note leurs groupes d'homologie respectifs $HF^\infty(\underline{L};\Z{2})$, $HF^+(\underline{L};\Z{2})$ et $HF^-(\underline{L};\Z{2})$, et si $Y$ est une 3-variétée fermée orientée, $c\in H_1(Y,\Z{2})$ et $z$ un point base, on note $HSI^\infty(Y,c,z)$, $HSI^+(Y,c,z)$ et $HSI^-(Y,c,z)$ les groupes correspondants pour $\underline{L}(W,p,c)$, où $W = Y\setminus z \cup S^2$ et $p\colon \rr/\zz\times [0,1]\to S^2$ sont comme dans la partie \ref{defHSI}.

Dans l'espoir d'obtenir une graduation absolue sur HSI, on peut alors poser la question suivante :
\begin{question}Le groupe $HSI^\infty(Y,0)$ peut-il être calculé à partir de l'homologie singulière de $Y$, comme c'est le cas de $HF^\infty (Y)$?
\end{question}
 
\section{Relations avec la variété des représentations}

Comme nous l'avons vu, les espaces des modules intervenant dans la construction admettent des descriptions venant de la théorie des représentations. De plus, dans la construction de $HSI(Y,0)$ par Manolescu et Woodward, les Lagrangiennes s'intersectent le long de la variété des représentations de $Y$ dans $SU(2)$. Cette observation permet de relier l'homologie instanton-symplectique à l'homologie de la variété des représentations.

Dans la section \ref{vartordue}, nous définissons une version $R(Y,c)$ de la variété des représentations dans $SU(2)$ "tordue" par la classe $c$, correspondant à l'intersection des Lagrangiennes définissant $HSI(Y,c)$. Ensuite nous expliquerons les relations entre $HSI(Y,c)$ et $H_*(R(Y,c))$ dans la section \ref{sectionsspec}, puis nous proposerons d'éventuelles relations avec l'invariant de Casson dans la section \ref{relcasson}.

\subsection{Une variété des représentations tordue dans $SU(2)$}\label{vartordue}


Soient $\pi$ et $\tilde{\pi}$ deux groupes,  $\phi \colon \tilde{\pi} \to \pi$  un morphisme surjectif, et $\mu \in \tilde{\pi}$ un élément d'ordre au plus 2. 0n note 
\[\tilde{R}( \tilde{\pi} , \pi,\mu) = \lbrace \rho \colon \tilde{\pi} \to SU(2) : \rho(\mu) = -I \rbrace.\]

Notons que cet espace peut être vide si $\mu$ est l'élément neutre.

Soit maintenant $Y$  une 3-variété, $K\subset Y$ un noeud, et $*\in Y\setminus K$ un point base. Soit $\mu \in \pi_1(Y\setminus K,*) $ la classe d'un méridien de $K$. D'après le théorème de Seifert-Van Kampen, $\pi_1 (Y,*)  = \pi_1(Y\setminus K,*) / \mu $. Posons $\pi  =\pi_1 (Y,*) $ et  $\tilde{\pi} = \pi_1(Y\setminus K,*) / \mu^2 $. On peut vérifier que l'espace  $\tilde{R}( \tilde{\pi} , \pi,\mu)$ ne dépend que de $Y$ et de la classe $c$ de $K$ dans $H_1(Y,\Z{2})$, on notera $R(Y,c)$ cet espace. C'est un revêtement d'une composante de la variété des représentations de $\pi$ dans $SO(3)$.

\begin{prop}\label{interlag} Soit $\Sigma\subset Y$ un scindement de Heegaard séparant $Y$ en deux corps à anses $H_0$ et $H_1$, et $c_0 \in H_1(H_0;\Z{2})$, $c_1 \in H_1(H_1;\Z{2})$ deux classes d'homologie telles que $c = c_0 + c_1$. Alors $R(Y,c)$ s'identifie à l'intersection des deux Lagrangiennes $L(H_0,c_0)$ et $L(H_1,c_1)$ définissant $HSI(Y,c)$.
\end{prop}
\begin{proof}
C'est une conséquence du théorème de Seifert-Van Kampen.
\end{proof}

\subsection{Une suite spectrale convergeant vers $HSI(Y,c)$}\label{sectionsspec}

Soient $L_0,L_1\subset M$ une paire de Lagrangiennes dans une variété symplectique monotone, il existe un lien entre l'homologie de Floer Lagrangienne et l'homologie de l'intersection des Lagrangiennes. Lorsque $L_0 = L_1$, Oh établit dans \cite{Ohspectralseq} une suite spectrale de seconde page $H^*(L_0;\Z{2})$ convergeant vers $HF(L_0, L_0;\Z{2})$. Pour cela, il construit une homologie de Floer "locale", prouve qu'elle est isomorphe à l'homologie de Morse, puis conclut par un argument général utilisant la filtration de $H^*(L_0;\Z{2})$.

Lorsque les Lagrangiennes ne sont plus égales mais s'intersectent de manière "clean", Po\'zniak définira une homologie locale, et prouvera qu'elle est  isomorphe à l'homologie de l'intersection.

Enfin, Fukaya, Oh, Ohta et Ono établissent une telle suite spectrale dans un cadre très général, voir \cite[Theorem 24.5BM]{FOOObook1}. Mentionnons que leur suite spectrale est établie pour une version de l'homologie à coefficients dans un anneau de Novikov. Toutefois, un argument similaire à celui que nous avons utilsé dans la section \ref{preuvedutriangle} devrait permettre d'obtenir une suite spectrale à coefficients dans $\zz$. On devrait donc avoir :

\begin{conj}\label{conj:suitespec} Soit $(Y,c)$ une 3-variété munie d'une classe $c\in H_1(Y,\Z{2})$. Supposons qu'il existe un scindement de Heegaard pour lequel des deux Lagrangiennes  $L(H_0,c_0)$ et $L(H_1,c_1)$ définissant $HSI(Y,c)$ s'intersectent de manière clean. Alors il existe une suite spectrale de première page $H_*(R(Y,c))$ qui converge vers $HSI(Y,c)$.
\end{conj}


\subsection{Une inégalité entre le rang de HSI et l'invariant de Casson pour les sphères de Brieskorn}\label{relcasson}

La construction de l'homologie des instantons de Floer est intimement liée à l'invariant de Casson : en effet, la description de Taubes de cet invariant en théorie de jauge a été un des points de départ de sa catégorification par Floer : $\chi(I_*(Y)) = \lambda (Y)$. Cependant, comme nous l'avons vu, la caractéristique d'Euler de l'homologie Instanton-Symplectique n'est pas $\lambda (Y)$, mais $\abs{H_1(Y;\zz)}$. L'exemple des sphères de Brieskorn illustre ce phénomène :

Soient $p,q,r$ trois entiers premiers deux à deux, et \[Y = \Sigma(p,q,r) = \left\lbrace (x,y,z)\in \cc^3 : \abs{x}^2 + \abs{y}^2 + \abs{z}^2 =1,~ x^p + y^q + z^r = 0  \right\rbrace\] la sphère de Brieskorn correspondante. Fintushel et Stern ont démontré dans \cite{FSinstseifert} que les Lagrangiennes dans la variété des caractères d'un scindement de Heegaard s'intersectaient transversalement. Il s'en suit que la variété des caractères de $Y$ est un ensemble fini, et que l'intersection des Lagrangiennes dans la variété des représentation est clean ($PU(2)$ agit librement au voisinage d'une orbite de représentations irréductibles). 

Rappelons que l'on distingue généralement trois types de caractères dans $SU(2)$ : le caractère trivial, les caractères centraux, et les caractères irréductibles. A chacuns de ces types de caractères correspondent des orbites dans la variété des représentations, respectivement difféomorphes au point, à la sphère $S^2$, et à  $SO(3)$. Elles ont pour caractéristique d'Euler respective 1,2 et 0 : c'est pour celà que les caractères irréductibles n'apparaissent pas dans $\chi(HSI(Y))$. En revanche, bien que $\chi(SO(3)) = 0$, $H_*(SO(3)) \neq 0$, ce qui permet d'espérer que les caractères irréductibles apparaissent dans $HSI(Y)$ (ils apparaissent tout du moins dans la première page de la suite spectrale de la conjecture \ref{conj:suitespec}).

Concernant la sphère de Brieskorn $Y$, d'après ce qui précède, l'homologie de sa variété des représentations à coefficients dans $R = \qq$, $\zz$ ou $\Z{2}$ est isomorphe à \[H_*(pt;R) \oplus H_*(SO(3);R) \oplus \cdots \oplus H_*(S0(3);R),\] où il y a $2 \lambda(Y)$ copies de $H_*(SO(3);R)$, une par caractère irréductible (si l'on adopte la convention $ \lambda(\Sigma(2,3,5) = 1$). Par ailleurs, le rang de $H_*(SO(3);R)$ est égal, pour $R = \qq$, $\zz$ et $\Z{2}$ respectivement, à 2, 2 et 4. Par conséquent, sous réserve que la suite spectale soit vraie, on aurait les inégalités suivantes :

\begin{align*}
\mathrm{rk} HSI(Y, 0;\qq) &\leq 4 \lambda(Y) +1   \\
\mathrm{rk} HSI(Y, 0;\zz) &\leq 4 \lambda(Y) +1   \\
\mathrm{rk} HSI(Y, 0;\Z{2}) &\leq 8 \lambda(Y) +1   .
\end{align*} 

\section{Un invariant pour des entrelacs et des variétés suturées}

Le foncteur que nous avons construit dans le chapitre \ref{chapHSI} avait pour catégorie de départ la catégorie $\Cob$, et non  la catégorie classique des 3-cobordismes entre surfaces : en effet il est nécessaire de retirer des disques aux surfaces closes pour pouvoir considérer les espaces de modules étendus de Jeffrey, et donc considérer des cobordismes avec un tube "vertical" reliant les bords. Il est alors naturel de vouloir étendre cette construction à des surfaces à plusieurs composantes de bord, et à des cobordismes à plusieurs tels tubes verticaux, qui sont des cas particuliers de variétés suturées. 

Nous présentons dans la section \ref{sectionespmodplusieursbords}  la variante de l'espace des modules étendu de Jeffrey associée à une surface à plusieurs composantes de bord, que nous avons déjà rencontré dans la section \ref{sectiontwistespmod} ($\N (\Sigma_{cut})$). Puis, après avoir défini une catégorie $\Coincob$ élargissant la catégorie $\Cob$, nous proposons une construction possible d'un foncteur   de $\Coincob$ vers une version modifiée de $\Symp$, puis nous expliquons comment obtenir des invariants pour des entrelacs. Cette construction est proche de la construction de Juhasz \cite{Juhaszsut} pour l'homologie de Floer suturée, il n'est cependant pas clair que l'on puisse définir des groupes d'homologie associés à des variétés suturées à partir d'un tel foncteur, ce point sera discuté dans la remarque \ref{rem:groupesut}.

\subsection{Espace des modules d'une surface à plusieurs bords}\label{sectionespmodplusieursbords}
  
Soit $\Sigma$ une surface compacte orientée de genre $g$, à $N$ composantes de bord, et $p = p_1\cup \cdots \cup p_N \colon \rr/\zz \times \lbrace 1,\cdots , N \rbrace \to \partial \Sigma$ un paramétrage du bord (ne respectant pas nécessairement l'orientation). On désigne par $\partial_i \Sigma$ la i-ème composante du bord.

\begin{defi}(\cite[Section 5.2]{jeffrey}) 
On définit  $\Mg (\Sigma, p)$ comme le quotient $\mathscr{A}_F^\mathfrak{g}(\Sigma) /\Gc (\Sigma )$, où

\[ \mathscr{A}_F^\mathfrak{g}(\Sigma)  = \left\lbrace A \in \Omega^{1} (\Sigma )\otimes \mathfrak{su(2)} \ |\ F_A = 0,\  A_{|\nu \partial_i \Sigma} = \theta_i ds \right\rbrace, \]
avec $\nu \partial_i \Sigma$  un voisinage tubulaire du i-ème bord (non fixé),  $s$ désigne le paramètre de $\rr/\zz$, et le groupe
 
 \[ \Gc (\Sigma ) = \left\lbrace u \colon \Sigma \rightarrow G\ |\ u_{|\nu \partial \Sigma} = I \right\rbrace \] agit par transformations de jauge.
 
 On note $\N (\Sigma, p)$ l'ouvert défini par la condition $\abs{\theta_i} < \pi \sqrt{2}$, pour tout $i$. Notons que $\N (\Sigma, p)$  n'est pas l'espace $\N^{g,N}$ introduit par Jeffrey.
\end{defi}
 
Ces espaces admettent une description holonomique, voir \cite[Proposition 5.3]{jeffrey}. Ce sont des variétés algébriques réelles de dimension $6g + 6N -6$, et $\N (\Sigma, p)$ s'identifie à un ouvert dense de $SU(2)^{6g + 6N -6}$, et est donc lisse. Ils admettent une 2-forme fermée, qui est symplectique et monotone sur $\N (\Sigma, p)$ pour les mêmes raisons que pour une surface à une composante de bord.

Afin de pouvoir définir l'homologie de Floer dans  $\N (\Sigma, p)$, décrivons une compactification possible par découpage symplectique.  L'espace $\Mg (\Sigma, p)$ hérite d'une action Hamiltonienne de $SU(2)^{N}$, qui agit par transformations de jauge constantes sur les composantes du bord. Le moment $\Phi = (\Phi_1, \cdots , \Phi_N)$ a pour coordonnées $\Phi_i( [A] ) = \theta_i$, et les fonctions $\abs{\Phi_i}$ engendrent des actions du cercle sur le complémentaire de $\Phi_i^{-1}( 0 )$. On peut alors poser :  
\begin{align*}
\Nc (\Sigma, p) &= \Mg (\Sigma, p)_{\abs{\Phi_1} \leq \pi \sqrt{2}, \cdots ,\abs{\Phi_N} \leq \pi \sqrt{2} } \\
&=  \Nc (\Sigma, p) \bigcup R_1 \cup \cdots \cup R_N ,
\end{align*} 
où $R_i = \left\lbrace \abs{\Phi_1} = \pi \sqrt{2} \right\rbrace  \red U(1)$ est une hypersurface symplectique.

En vertu de \cite[Proposition 4.6]{MW}, la 2-forme $\tilde{\omega}$ induite sur $\Nc (\Sigma, p)$ est encore monotone, mais dégénère sur les $R_i$. Néanmoins il est possible d'obtenir une forme symplectique non-monotone $\omega$ en découpant à $\pi \sqrt{2} - \epsilon$. Il devrait alors être possible de reproduire le raisonnement de \MW et de définir l'homologie matelassée pour des correspondances compatibles avec chaque hypersurface.

Ainsi, la catégorie symplectique  d'arrivée devrait être modifiée en une catégorie $\mathbf{\widetilde{Symp}}$, dont les objets consisteraient en des variétés munies de deux 2-formes et de plusieurs hypersurfaces, et les morphismes des correspondances Lagrangiennes vérifiant une condition adéquate de compatibilité avec les hypersurfaces, ainsi que des hypothèses similaires à celles de $\Symp$.

  \subsection{Une catégorie de cobordismes à coins}
  
La catégorie de cobordisme qui se substituera à $\Cob$ sera la suivante.

\begin{defi}[Catégorie des cobordismes à coins]\label{sutcob}

On appelle \emph{catégorie des cobordismes à coins, avec classe d'homologie de degré 1 à coefficients dans $\Z{2}$}, que l'on notera $\Coincob$, la catégorie dont :
\begin{itemize}

\item les objets sont les couples $(\Sigma,p)$, où $\Sigma$ est une surface compacte, orientée, sans composante fermée, et $p\colon \rr/\zz \times \lbrace 1, N \rbrace \rightarrow \partial\Sigma$ est un difféomorphisme (paramétrage).

\item les morphismes de $(\Sigma_0,p_0)$ vers $(\Sigma_1,p_1)$ sont des classes de difféomorphismes  de 5-uplets $(W, \pi_{\Sigma_0},  \pi_{\Sigma_1}, p,  c)$, où $W$ est une 3-variété compacte orientée à bord, $\pi_{\Sigma_0}$,  $\pi_{\Sigma_1}$ et $p$ sont des plongements de $\Sigma_0$, $\Sigma_1$ et $\rr/\zz \times V $, où $V$ est une variété à bord compacte de dimension 1, dans $\partial W$, $\pi_{\Sigma_0}$ renverse l'orientation, $\pi_{\Sigma_1}$ préserve l'orientation, et tels que 
\[ \partial W = \pi_{\Sigma_0}(\Sigma_0)\cup  \pi_{\Sigma_1}(\Sigma_1)\cup  p(\rr/\zz \times V), \]
avec $\pi_{\Sigma_0}(\Sigma_0)$  et $\pi_{\Sigma_1}(\Sigma_1)$  disjoints, pour $i=0,1$, et
\begin{align*}
\left(  \pi_{\Sigma_0}(\Sigma_0) \cup \pi_{\Sigma_1}(\Sigma_1)  \right)   \cap p(\rr/\zz \times V)  &= \pi_{\Sigma_0}(p_0(\rr/\zz)) \cup \pi_{\Sigma_1}(p_1(\rr/\zz))  \\ &= p(\rr/\zz \times \partial V),
\end{align*}
pour $v\in \partial V$ et $i=0$ ou $1$ est tel que $p(s,v) \in \pi_{\Sigma_i}(p_i(\rr/\zz))$,  alors $p(s,v) = \pi_{\Sigma_i}(p_i(s))$. Enfin, $c \in H_1 (W, \Z{2}))$.
  
On appellera  $p(\rr/\zz \times V) $ partie suturée de $\partial W$, et on la notera $\partial^{sut} W$.

Deux tels 5-uplets $(W, \pi_{\Sigma_0},  \pi_{\Sigma_1}, p,  c)$ et $(W', \pi'_{\Sigma_0},  \pi'_{\Sigma_1}, p',  c')$ sont dits équivalents s'il existe un difféomorphisme $\varphi \colon W \rightarrow W'$ compatible avec les plongements et préservant la classe. 

\item la composition des morphismes consiste à recoller à l'aide des plongements, et à ajouter les classes d'homologie.

\end{itemize}

\end{defi}

\begin{remark}
Une variété suturée $(M, \gamma)$ telle que $R_+(\gamma)$ et $R_-(\gamma)$ n'admettent pas de composantes fermées peut être vue comme un morphisme de cette catégorie (pourvu que l'on se fixe un paramétrage de la suture). Notons que tout morphisme de cette catégorie n'est pas nécessairement une variété, car on autorise un cylindre à revenir sur la même surface, comme dans la figure \ref{tubeU}.
\end{remark}

\begin{remark}Cette catégorie est différente de la catégorie $\mathbf{Sut}$ apparaissant dans \cite{Juhaszcobsut}, dans laquelle les variétés suturées sont des objets et non des morphismes. Néanmoins ces deux catégories peuvent être réunies dans une 2-catégorie, et le foncteur que l'on a construit devrait pouvoir être promu en un 2-foncteur sur cette 2-catégorie, à valeur dans une version appropriée de la 2-catégorie de Wehrheim et Woodward.
\end{remark}

De même que dans le chapitre \ref{chapHSI}, on ne devrait pouvoir associer des correspondances Lagrangiennes qu'à certains cobordismes "élémentaires" au sens de la théorie de Cerf : il faudra introduire une catégorie $\Coincobelem$  analogue à la catégorie $\Cobelem$, définir un foncteur de $\Coincobelem$ vers $\mathbf{\widetilde{Symp}}$, puis vérifier qu'il se factorise par $\Coincob$. Une principale nouveauté dans la catégorie $\Coincobelem$ sera des cobordismes "en U" entre surfaces possédant un nombre différent de composantes de bord (qui ne sont pas à proprement parler des cobordismes suturés), et il faudra vérifier un  mouvement de Cerf supplémentaire, correspondant à l'annulation de deux tels tubes, voir figure \ref{tubeU}. 

\begin{figure}[!h]
    \centering
    \def\svgwidth{\textwidth}
    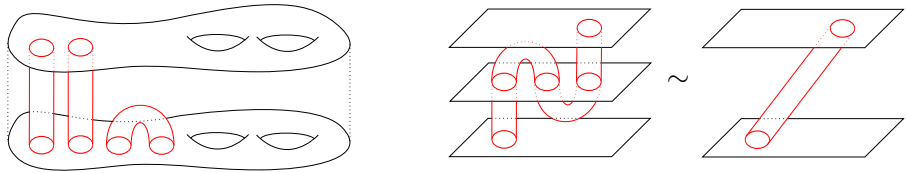
      \caption{A gauche, un cobordisme à coins élémentaire de type U, à droite le mouvement de Cerf de naîssance-mort pour de tels cobordismes.}
      \label{tubeU}
\end{figure}

  \subsection{Correspondances associées à un cobordisme élémentaire}
  
Soit $(W,\pi_{\Sigma_0},\pi_{\Sigma_1},p,c)$ un cobordisme à coins. Par souci de lisibilité, on omettra de préciser les plongements par la suite. On peut définir
$\Mg(W,c)$ et $L(W,c) \subset  \Mg(\Sigma_0)\times \Mg(\Sigma_1)$, puis $L^c(W,c) \subset  \Nc(\Sigma_0)\times \Nc(\Sigma_1)$ de manière tout à fait analogue à ce que l'on a fait pour des cobordismes à bord verticaux. Il faut alors vérifier que $L^c(W,c)$ est compatible avec chaque hypersurface $R_i$, simplement connexe, et vérifie des hypothèses adéquates afin d'exclure d'éventuels phènomènes de bubbling. On aura alors construit un foncteur de $\Coincobelem$ vers la catégorie $\widetilde{\mathbf{Symp}}$. Il faudra ensuite vérifier les six mouvements de Cerf pour obtenir un foncteur partant de $\Coincob$.

  \subsection{Homologie pour des entrelacs}

  Soit $L\subset Y$ un entrelacs, choisissons deux points $p$ et $q$ sur $L$ et retirons deux boules de $Y$ centrées en $p$ et $q$, de façon à obtenir un cobordisme $\widetilde{W}$ de $S^2$ vers $S^2$. Retirons à  présent un voisinage tubulaire de $L$,  de manière à obtenir un cobordisme à coins $W$ de $\Sigma$ vers $\Sigma$, où $\Sigma$ est la sphère privée de deux disques (on prend comme suture le bord du voisinage tubulaire de $L$). En appliquant la construction précédente, on obtient une correspondance Lagrangienne généralisée $\underline{L}(W)$ de $\Nc(\Sigma)$ vers lui-même. Par ailleurs, $\Nc(\Sigma)$ contient une Lagrangienne $L_0$ privilégiée, correspondant aux connexions nulles au voisinage du bord : ce sont exactement les connexions qui s'étendent aux deux boules que l'on a retiré. Supposons que l'on puisse définir l'homologie de Floer matelassée dans $\mathbf{\widetilde{Symp}}$, on poserait alors : \[HSI(Y,L) = HF(L_0, \underline{L}(W), L_0^T).\]

\begin{remark}Cette construction faisant intervenir les deux points $p$ et $q$, le groupe obtenu devrait en dépendre à priori. Néanmoins, le choix de la Lagrangienne $L_0$ devrait annuler cette dépendance, car rajouter $L_0$ correspond précisément à reboucher les boules.
\end{remark}

\begin{remark}\label{rem:groupesut} Dans \cite{Juhaszsut}, Juhasz parvient à définir des groupes d'homologie associés à un certain type de variétés suturées, appelées "balanced", généralisant l'homologie d'Heegaard-Floer pour les 3-variétés et pour les entrelacs. Il considère pour cela des produits symétriques $Sym^k \Sigma$ de surfaces, où $k$ n'est pas nécessairement égal au genre de $\Sigma$. Ne disposant pas de tels analogues, nous ne pouvons définir de tels groupes. Néanmoins, suivant la remarque \ref{foncteurainfini}, la construction de \WW devrait permettre d'associer  à une variété suturée qui est un morphisme de $\Coincob$ un foncteur entre des catégories de Donaldson, ou de Fukaya dérivées.
\end{remark}


\backmatter

\chapter{Notations}

\renewcommand{\arraystretch}{1.5}
\hspace{-1.5cm}\begin{tabular}{ l p{11cm} r }
$\#$ & Somme connexe, ou cardinal algébrique d'une variété compacte orientée de dimension zéro   & 17 \\

 $\odot$     & Composition dans la catégorie $\Cobelem$ & 20 \\

$\mathbb{A}$   & Application antipodale de la sphère $S^n$ & 47 \\

$a_{\underline{L}}$   & Action matelassée & 62\\ 
 
$\Cob$ & Catégorie des cobordismes à bord verticaux & 18\\

$\Cobelem$ & Catégorie des cobordismes élémentaires à bord verticaux & 19\\

$\mathcal{E}$, $\mathcal{E}_-$, $\mathcal{E}_+$  &  Bouts d'une surface matelassée, bouts entrants, bouts sortants   & 51\\

$F_{W,c_W}$     & Application associée au cobordisme $W$ et à la classe $c_W$ & 83 \\

$\I(\underline{L})$     & Points d'intersections généralisés & 11 \\

$HF(\underline{L})$     & Homologie matelassée associée à la correspondance Lagrangienne généralisée & 17\\

$\mathrm{Hol}_\gamma A$    & Holonomie de la connexion $A$ le long d'un chemin $\gamma$ & \\

$HSI(Y,c,z)$   & Homologie Instanton-Symplectique & 38 \\

$\underline{L}^T$     & Correspondance Lagrangienne généralisée avec l'ordre inversé & 10 \\

$\underline{L}(W,p,c)$     & Correspondance Lagrangienne généralisée associée au cobordisme à bord vertical $(W,p,c)$ & 36 \\

$\Lambda$    & Anneau du groupe $\rr$ & 48\\

$\hat{\Lambda}$    & Complétion de $\Lambda$ & 68 \\

$M^-$     & Variété symplectique $M$ munie de la forme symplectique opposée & 10 \\

$\Mg(\Sigma,p)$     & Espace des modules étendu & 24 \\

$\N(\Sigma,p)$    & Partie de l'espace des modules étendu correspondant à $\abs{\theta} <\pi \sqrt{2}$ & 24 \\

$\Nc(\Sigma,p)$     & Découpage de  l'espace des modules étendu en $\abs{\theta} = \pi \sqrt{2}$ & \\

\end{tabular}

\renewcommand{\arraystretch}{1.5}
\hspace{-1.5cm}\begin{tabular}{ l p{11cm} r }

$\omega$   & Forme symplectique non-monotone associée à un objet de $\Symp$ & 15\\

$\tilde{\omega}$   & Forme fermée monotone associée à un objet de $\Symp$ & 15\\

$pt$    & Variété symplectique réduite à un point & 11 \\

$R$  &  Primitive de la fonction angle d'un twist de Dehn & 47\\

$R$, $\underline{R}$  &  Hypersurface symplectique associée à un objet de $\Symp$ & 15\\

$\Symp$     & Variante de la catégorie symplectique de Weinstein & 15\\

$\sigma\in \mathcal{S}$ & Couture    & 11\\

$\Sigma_{cut}$    & Surface $\Sigma$ découpée le long d'une courbe & 71\\

$T(\lambda)$    & Covecteurs de $T^*S^n$ de norme $<\lambda$  & 47 \\

$(W, \pi_{\Sigma_0},  \pi_{\Sigma_1}, p,  c)$      & Cobordisme à bord vertical & 18 \\

$\chi (\tilde{x}_0, \underline{x}, \underline{y})$   & Quantité correspondant à l'aire d'un triangle modulo $M$ & 63\\
 $ \chi_\tau (\tilde{x}_0, \underline{x}, \underline{y})$    & Quantité correspondant à l'aire d'un triangle modulo $M$ associée à un twist $\tau$ & 63\\

 $\Z{n}$ & Anneau $\zz / n \zz$ & \\
\end{tabular}

\bibliographystyle{alpha}
\bibliography{biblio}

\end{document}